\documentclass[11pt]{amsart}

\usepackage[T1]{fontenc}
\usepackage[utf8]{inputenc}
\usepackage[english]{babel}
\usepackage{amscd,amsmath,amsthm,amssymb}
\usepackage{mathtools}
\usepackage{mathrsfs}
\usepackage{comment}
\usepackage{xcolor}
\usepackage{array}
\usepackage{caption}
\usepackage{mathabx}
\usepackage{standalone}
\usepackage{scalerel}
\usepackage{csquotes}

\usepackage{accents} 

\usepackage{enumitem}
\setlist[enumerate]{label=\textnormal{(\arabic*)}}

\definecolor{cadmiumgreen}{rgb}{0.0, 0.42, 0.24}
\usepackage[
  colorlinks,
  citecolor=cadmiumgreen,
  pdfstartview ={FitV},
]{hyperref}

\usepackage[
  alphabetic,
  msc-links,
  nobysame,
  lite,
]{amsrefs}

\usepackage{tikz, tikz-cd}
\usetikzlibrary{calc, backgrounds, quotes, arrows.meta}

\usepackage[left=3cm,top=3.5cm,right=3cm]{geometry}
\def\roundface#1#2[#3]{\fill[#3] (I) -- (I#1) .. controls ($.3*(I#1)+.7*(I#1#1#2)$) and ($.3*(I#1#2)+.7*(I#1#1#2)$) .. (I#1#2) .. controls ($.3*(I#1#2)+.7*(I#2#1#2)$) and ($.3*(I#2)+.7*(I#2#1#2)$) .. (I#2) -- cycle}

\newtheorem{thm}{Theorem}[section]

\newtheorem{lemma}[thm]{Lemma}

\newtheorem{prop}[thm]{Proposition}

\newtheorem{cor}[thm]{Corollary}

\newtheorem{conj}[thm]{Conjecture}

\newtheorem*{theoremA}{Theorem A}
\newtheorem*{theoremB}{Theorem B}
\newtheorem*{theoremC}{Theorem C}

\newtheorem*{theoremD'}{Theorem D'}

\theoremstyle{definition}

\newtheorem{definition}[thm]{Definition}

\newtheorem{remm}[thm]{Remark}
\newenvironment{remark}
  {\pushQED{\qed}\renewcommand{\qedsymbol}{$\diamond$}\remm}
  {\popQED\endremm}
\newtheorem{exx}[thm]{Example}
\newenvironment{example}
  {\pushQED{\qed}\renewcommand{\qedsymbol}{$\diamond$}\exx}
  {\popQED\endexx}

\numberwithin{equation}{section}

\newcommand{\Q}{\mathbb{Q}}
\newcommand{\N}{\mathbb N}
\newcommand{\R}{\ssub{\mathbb{R}}}
\newcommand{\Z}{\ssub{\mathbb{Z}}}

\renewcommand{\emptyset}{\varnothing}

\let\oldchi\chi
  \newcommand{\raisechi}[2]{\raisebox{.4ex}{$#1#2$}}
  \renewcommand{\chi}{{\mathpalette\raisechi\oldchi}}

\makeatletter
\let\oldsum\sum
\renewcommand{\sum}{\@ifnextchar_\@mysum\oldsum}
\def\@mysum_#1{\oldsum_{\substack{#1}}}
\let\oldbigoplus\bigoplus
\renewcommand{\bigoplus}{\@ifnextchar_\@mybigoplus\oldbigoplus}
\def\@mybigoplus_#1{\oldbigoplus_{\substack{#1}}}
\let\oldprod\prod
\renewcommand{\prod}{\@ifnextchar_\@myprod\oldprod}
\def\@myprod_#1{\oldprod_{\substack{#1}}}
\makeatother

\let\oldbigwedge\bigwedge
\renewcommand{\bigwedge}{{\textstyle\oldbigwedge\!}}

\DeclareMathOperator{\cl}{cl} 
\DeclareMathOperator{\ord}{ord} 

\newcommand{\G}{\mathbb G} 
\renewcommand{\P}{\mathbb P} 
\newcommand{\TP}{\mathbb{TP}} 

\renewcommand{\div}{\mathrm{div}} 

\newcommand{\CC}{\mathcal C} 

\newcommand{\shiftcomp}[2][0]{{}\mkern#1mu\overline{\mkern-#1mu#2}}
\newcommand{\comp}[1]{\ifx#1N \shiftcomp[3]{#1}\else\ifx#1X \shiftcomp[3]{#1}\else\ifx#1Z \shiftcomp[3]{#1}\else\ifx#1F \shiftcomp[3]{#1} \else \shiftcomp{#1}\fi\fi\fi\fi} 

\newcommand{\maxideal}{\mathfrak m_{K}}

\newcommand{\Trop}{\mathrm{Trop}}

\newcommand{\Spec}{\operatorname{Spec}}

\newcommand{\MA}{\ssub{\mathrm{MA}}!}

\newcommand{\an}{{\mathrm{an}}}

\newcommand{\NA}{{\mathrm{NA}}}

\newcommand{\env}{\ssub[-2pt]{\mathrm{P}}!} 

\NewDocumentCommand{\ssub}{O{0pt} O{.9} m t! e{_^}}{
  #3%
  \IfValueT{#5}{
    \IfBooleanTF{#4}{\sb{\hspace{#1}\scaleobj{#2}{#5}}}{\sb{#5}}
  }
  \IfValueT{#6}{\sp{\scaleobj{#2}{#6}}}
}

\ExplSyntaxOn
\NewDocumentCommand{\tossub}{o o m}{
  \expandafter\let\csname old\cs_to_str:N #3\endcsname#3
  \renewcommand#3%
  {\ssub[#1][#2]{\csname old\cs_to_str:N #3\endcsname}}
}
\ExplSyntaxOff

\tossub\varpi
\tossub\deg
\tossub\omega
\tossub\nu
\tossub\rho

\renewcommand{\setminus}{\smallsetminus}

\newcommand{\be}{\begin{equation}}
\newcommand{\ee}{\end{equation}}
\newcommand{\bes}{\begin{equation*}}
\newcommand{\ees}{\end{equation*}}
\newcommand{\ben}{\begin{eqnarray}}
\newcommand{\een}{\end{eqnarray}}
\newcommand{\bens}{\begin{eqnarray*}}
\newcommand{\eens}{\end{eqnarray*}}

\newcommand{\Supp}{\mathrm{Supp}}

\newcommand{\MongeA}{Monge--Ampère }

\newcommand{\PSH}{\mathrm{PSH}}

\newcommand{\Conv}{\mathrm{Conv}}

\newcommand{\LConv}{\mathrm{LConv}}

\newcommand{\PL}{\mathrm{PL}}
\newcommand{\Aff}{\mathrm{Aff}}

\newcommand{\interior}{\mathrm{int}}

\newcommand{\kahler}{K\"ahler }

\newcommand{\MAo}{\MA_{\omega}}
\newcommand{\envomega}{\env_{\omega}}

\newcommand{\BA}{{\mathbb {A}}} 
\newcommand{\BC}{{\mathbb {C}}}

\newcommand{\BK}{{\mathbb {K}}}

 \newcommand{\BR}{{\mathbb {R}}}
 \newcommand{\BT}{{\mathbb {T}}}

 \newcommand{\CH}{{\mathcal {H}}}
\newcommand{\CI}{{\mathcal {I}}} 
\newcommand{\CK}{{\mathcal {K}}} 
\newcommand{\CM}{{\mathcal {M}}} 
\newcommand{\CO}{{\mathcal {O}}} 
 
 \newcommand{\CT}{{\mathcal {T}}}
\newcommand{\CU}{{\mathcal {U}}} \newcommand{\CV}{{\mathcal {V}}}
\newcommand{\CW}{{\mathcal {W}}} 
\newcommand{\CX}{{\mathcal {X}}}

 \newcommand{\CZ}{{\mathcal {Z}}}

\newcommand{\fm}{{\mathfrak{m}}} 
 \newcommand{\fp}{{\mathfrak{p}}}
\newcommand{\fq}{{\mathfrak{q}}} \newcommand{\fr}{{\mathfrak{r}}}
 \newcommand{\ft}{{\mathfrak{t}}}

\newcommand{\fA}{{\mathfrak{A}}} 
 \newcommand{\fD}{{\mathfrak{D}}}

 \newcommand{\fL}{{\mathfrak{L}}}

 \newcommand{\fV}{{\mathfrak{V}}}
 \newcommand{\fX}{{\mathfrak{X}}}

\usepackage[disable]{todonotes} 

\newcommand{\valfield}{K}
\newcommand{\valring}{K^{\circ}}
\newcommand{\resfield}{\kappa}

\newcommand{\lnote}[1]{\todo[color=cyan!50]{{\scriptsize Internal Notes}: #1}}

\newcommand{\dm}{\mathrm{d}}
\newcommand{\dmb}{\mathrm{d}^c}

\DeclareMathOperator{\reg}{reg}
\DeclareMathOperator{\sing}{sing}
\DeclareMathOperator{\val}{val}
\DeclareMathOperator{\TConv}{TConv}

\newcommand{\GPSH}{\mathrm{PSH}}
\newcommand{\LPSH}{\mathrm{LPSH}}
\newcommand{\WPSH}{\mathrm{WPSH}}

\newcommand{\GPSHo}{\mathrm{PSH}_{\omega}}
\newcommand{\LPSHo}{\mathrm{LPSH}_{\omega}}
\newcommand{\WPSHo}{\mathrm{WPSH}_{\omega}}

\newcommand{\PSHsm}{\mathrm{PSH}^{\mathrm{sm}}}
\newcommand{\LPSHsmo}{\mathrm{LPSH}^{\mathrm{sm}}_{\omega}}

\newcommand{\PSHb}{\PSH \cap L^{\infty}_{\mathrm{loc}}}
\newcommand{\GPSHob}{\GPSHo \cap L^{\infty}}
\newcommand{\LPSHob}{\LPSHo \cap  L^{\infty}}
\newcommand{\PSHlog}{\PSH^{\mathrm{log}}}
\newcommand{\PSHlogreg}{\PSH^{\mathrm{log}}_{\mathrm{reg}}}
\newcommand{\WPSHlog}{\WPSH^{\mathrm{log}}}
\newcommand{\FS}{\mathrm{FS}}

\newcommand{\nef}{\CZ^{1,1}_{\ge 0}}
\newcommand{\amp}{\CZ^{1,1}_{> 0}}
\newcommand{\PLforms}{\CZ^{1,1}}
\newcommand{\Amp}{\mathrm{Amp}}
\newcommand{\Nef}{\mathrm{Nef}}
\newcommand{\Pic}{\mathrm{Pic}}

\newcommand{\Radonp}[1]{\mathrm{Radon}_+(#1)} 

\newcommand{\mult}{\mathrm{mult}}

\begin{document}
\title{On the local theory of non-archimedean psh functions}
\author{Lyuhui Wu}
\address{IMJ-PRG, Sorbonne Universit\'e, Paris, France}
\email{lyuhui.wu@imj-prg.fr}
\date{September 22, 2026}
\maketitle

\begin{abstract} 
The purposes of this paper are two-folded. First, we define a class of local psh functions on Berkovich spaces, and study their properties. In particular, we show that this class of psh functions is equivalent on the one hand to the one defined by Chambert-Loir--Ducros, and is equivalent on the other hand to the class of global psh functions of Boucksom--Favre--Jonsson on projective varieties satisfying the envelope property. Second, making use of our local theory of non-archimedean psh functions, we prove the envelope conjecture for projective normal varieties over a discretely or trivially valued non-archimedean field of residue characteristic $0$. 
\end{abstract}

\section{Introduction}
The study of psh (plurisubharmonic) functions on complex analytic spaces plays important roles in complex and \kahler geometry. In the non-archimedean setting, a fruitful theory of psh functions (metrics) on projective varieties  has been developed in the last two decades (\cites{BFJ15, BE21, BJ22}). 

Let $K$ be a complete non-archimedean field with valuation ring $\valring$. Let $X$ be a projective variety over $K$. For a $\Q$-line bundle $L$  on $X$, a model $(\fX, \fL)$ of $(X,L)$ consists of a scheme $\fX$ which is flat and proper over $\Spec \valring$, a $\Q$-line bundle $\fL$ over $\fX$, together with an identification $(\fX_{K}, \fL_K)\simeq (X,L)$. Every model defined a model metric $|\cdot |$ on $L$. A function $f$   on $X^{\an}$ is \emph{piecewise linear (PL)} if there exists two model metrics on an ample line bundle $L$ such that $|\cdot|_{\omega_1} = |\cdot|_{\omega_2} \cdot e^{-f}$. 

Following the idea of \cite{Zha95}, a model metric defined by the model $(\fX, \fL)$ is called semipositive if $\fL$ is relatively nef on $\fX$. A closed semipositive $(1,1)$-form on $X^{\an}$ is an isomorphism class of semipositive model metrics on $\Q$-nef line bundles  on $X$. Fix a closed semipositive $(1,1)$-form $\omega$.  A PL function $\varphi$ is called $\omega$-psh if $|\cdot|_{\omega} \cdot e^{-\varphi}$ is a semipositive model metric. Following \cite{BJ22}, when $L$ is ample, a function $\varphi \colon X^{\an} \to \R \cup \{-\infty\}$ is called  $\omega$-psh ($\varphi \in \GPSHo(X^\an)$) if $\varphi$ is generically finite (\emph{i.e.} $\varphi$ is not constantly equal to $-\infty$ on each connected component of $X^{\an}$) and $\varphi$ can be written as the pointwise limit of a decreasing sequence of $\omega$-psh PL functions. When $L$ is nef, $\varphi \in \GPSHo(X^{\an})$ if $\varphi \in \GPSH_{\omega+\epsilon \theta}(X^{\an})$ for each $\epsilon \in \Q_{>0}$, where $\theta$ is a  closed semipositive $(1,1)$-form satisfying $c_1(\theta) \in \Amp(X)$.

We say that $\GPSHo(X^{\an})$ satisfies the envelope property if the usc upper envelope of any bounded above family of $\omega$-psh functions remains $\omega$-psh on $X^{\an}$. Equivalently, when $L$ is ample, $\GPSHo(X^{\an})$ satisfies the envelope property if $\envomega(f)$ is continuous for any $f \in C^0(X^{\an})$ (\cite[Lemma 7.30]{BE21}, \cite[Lemma 5.17]{BJ22}). Here, the $\omega$-psh envelope $\envomega(f)$ is defined as the pointwise supreme of the family of $\omega$-psh functions $\varphi$ satisfying $\varphi \le f$.  The envelope property is conjectured to hold on any  projective normal variety $X$. It has been established when (a) $X$ is a smooth curve (\cite{Thu05}); (b) $X$ is smooth, $K$ is discretely or trivially valued, and $K$ is of equicharacteristic $0$ (\cites{BFJ16, BJ22}); (c) $X$ is smooth, $K$ is discretely or trivially valued, and $\dim X \le 2$ (\cites{GJKM19, FGK25}).

The envelope property is plays a  central role in non-archimedean pluripotential theory. In particular, it is essential in a variational proof of a non-archimedean analogue of Calabi--Yau theorem (\cite{BFJ15}), and in the non-archimedean approach  to Yau--Tian--Donaldson conjecture (see for example \cite{BBJ21}). 

In this paper, we establish the envelope property in the following setting.
\begin{theoremA} [Theorem \ref{thm:envelope_conjecture_equichar_zero}] \label{mainthm:envelope_conjecture_equichar_zero} 
  Let $K$ be a discretely valued or trivially valued complete non-archimedean field of equicharacteristic $0$. Let $X$ be a projective normal variety over $K$, and let $\omega$ be a closed semipositive $(1,1)$-form on $X$. Then, $\GPSH_{\omega}(X^{\an})$ satisfies the envelope property. 
\end{theoremA}

Theorem A implies a Calabi--Yau theorem in this setting (see \cite[Theorem 12.8]{BJ22}), and has applications to Yau--Tian--Donaldson conjecture in the singular settings (see for example \cite{HL25}).

The proof of Theorem A relies on establishing the following  non-archimedean analogue of Fornaess-Narasimhan theorem. For a function $\varphi$ on $X^{\an}$, we say that $\varphi$ is a \emph{weakly $\omega$-psh function} ($\varphi \in \WPSHo(X^{\an})$) if there exists a birational proper morphism $p \colon X' \to X$ such that $\pi^* \varphi \in \GPSH_{\pi^{*} \omega}((X')^{\an})$. 

\begin{theoremB} [Theorem \ref{thm:global_NA_FN_theorem}] \label{mainthm:global_NA_FN_theorem}
  Assume that $K$ is a complete discretely valued or trivially valued NA field of equicharacteristic $0$ such that $K$ admits a countable dense subfield. Let $X$ be a projective variety over $K$, and let $\omega$ be a closed semipositive $(1,1)$-form on $X$ with $c_1(\omega) \in \Amp(X)$. Then, $\WPSHo(X^{\an})=\GPSHo(X^{\an})$.
\end{theoremB}

We first explain how to deduce Theorem A from Theorem B. Since $X$ is of finite type over $K$, there exists a complete NA subfield $K_1 \subset K$ such that $K_1$ admits a countable dense subfield, and $(X, \omega)$ is defined over $K_1$. After some pluripotential theoretical arguments, one can show that the envelope property of $(X, \omega)$ over $K$ is equivalent to that over $K_1$. Thus, we may assume that $K$ admits a countable dense subfield.  Let $\pi \colon \widetilde{X} \to X$ be a resolution of singularity of $X$. As we mentioned before, $\GPSH_{\pi^{*} \omega}(\widetilde{X}^{\an})$ satisfies the envelope property after \cites{BFJ16, BJ22}. Therefore,  $(\sup_{i \in I} \{\pi^{*}\varphi_i\})^{\star} \in \GPSH_{\pi^{*} \omega}(\widetilde{X}^{\an})$. On the other hand, since $X$ is normal, Zariski's main theorem implies that each fibre of $\pi^{\an}$ is connected. As $\pi^{\an}$ is proper and  $\pi^* \omega$ equals to zero when restricted to each fibre of $\pi^{\an}$, any $\pi^{*} \omega$-psh function $\varphi$ is constant on each fibre of $\pi^{\an}$. Thus, $(\sup_{i \in I} \{\pi^{*}\varphi_i\})^{\star}$ descends to an upper-semicontinuous  function $\psi$ on $X^{\an}$. As $\pi^{\an}$ is bijective on divisorial points, $\psi$ is the minimal upper-semicontinuous extension of $\psi|_{X^{\div}}$, and $X^{\div}$ is nonnegligible, it is not very hard to deduce that $\psi=(\sup_{i \in I} \{\varphi_i\})^{\star}$ on $X^{\an}$. This shows that $(\sup_{i \in I} \{\varphi_i\})^{\star} \in \WPSHo(X^{\an})$, and thereby $(\sup_{i \in I} \{\varphi_i\})^{\star} \in \GPSHo(X^{\an})$ by Theorem A.

\medskip

Now we explain the motivations behind Theorem B. Let $V$ be a connected complex variety, and $\varphi \colon V \to \R \cup \{-\infty\}$ be a function which is not constantly equal to $-\infty$ on $V$. Following \cite{FN80}, \footnote{We remind the reader that, in the complex setting, the notion of weakly psh functions is not entirely uniform across the literature. In particular, different sources may adopt definitions that are not equivalent without additional hypotheses.} $\varphi$ is called a weakly psh function if $\varphi$ is upper-semicontinuous, and for any holomorphic map $f \colon \Delta \to V$ of the complex unit disc into $V$, $\varphi \circ f$ is subharmonic on $\Delta$. $\varphi$ is called a psh function if  $\varphi$ is locally a restriction of some weakly psh function on an open subset in $\BC^N$ for a local embedding of $V$. It is obvious that the notions of psh functions and weakly psh functions are equivalent on complex manifolds. If $V$ is a  singular complex variety, and $\mu \colon \widetilde{V} \to V$ is a resolution of singularities, then it is not hard to see that a function $\varphi$ on $V$ is weakly psh iff $\varphi \circ \mu$ is weakly psh, and hence psh on $\widetilde{V}$. \footnote{That $\varphi$ is weakly psh implies $\varphi \circ \mu$ is weakly psh is obvious. The converse direction follows from the observation that for any germ of holomorphic map $f \colon(\Delta,0) \to (V,x)$, there exists a germ of holomorphic map $\widetilde{f} \colon (\Delta, 0) \to (\widetilde{V}, x')$ such that $f(t^k)=\mu(\widetilde{f}(t))$ (see the last step of the proof of \cite[Theorem 1.7]{Dem85}).} The celebrated Fornaess-Narasimhan theorem proven in \cite{FN80} then asserts that the notion of weakly psh functions and psh functions are actually equivalent on any complex varieties. This partially  justifies the terminology of non-archimedean weakly $\omega$-psh functions,  and the analogy between between Theorem B and Fornaess-Narasimhan theorem in the complex setting.

However, this analogy is not very precise: the definition of psh functions on complex varieties is local, while the definition of $\omega$-psh functions in Theorem B is  global. Actually, a more precise analogue of Theorem B in the complex analytic setting is the following: every weakly $\omega$-psh function can be written as a pointwise limit of a decreasing sequence of  $\omega$-Fubini--Study  potentials. This is a statement stronger than Fornaess-Narasimhan theorem, whose proof \footnote{We refer the readers to \cite[Theorem 8.1]{BE21} for more details.} requires additionally an extension theorem of psh functions of Coman--Guedj--Zeriahi (\cite{CGZ13}) and Demailly's regularization theorem on smooth projective manifolds (\cite{Dem92}).

Our strategy of proving Theorem B follows the line of the proof of Fornaess-Narasimhan theorem in \cite{FN80}. However, as discussed in the previous paragraph,  the complex analytic proof is also local in nature.  This requires us to discuss about a local theory of psh functions on non-archimedean spaces.

\medskip

Parallel to the global approach to NA psh functions discussed above, \cite{CLD} has developed a local theory of psh functions. Recall that for an open subset $U \subset X^{\an}$ and a finite collection of invertible analytic functions $f_1, \cdots, f_k \in \CO^{\times}(U)$, $(-\log |f_1|, \cdots, -\log|f_k|)$  defines a tropicalization map $U \to \R^k$, whose image is a finite polyhedral complex. Following \cite{CLD}, a function $u$ on $U$ is called a smooth convex function if locally $u$ can be written as $g(-\log |f_1|, \cdots, -\log|f_k|)$, where $g$ is a smooth function on $\R^k$ such that the restriction of $g$ on each face of the image of the tropicalization map is convex.  A function $u$ on $U$ is called a  \emph{smooth approximable  psh  function} if  any $x \in U$ admits an open neighborhood $V$ such that $u$ equals to the pointwise limit of a decreasing sequence of smooth convex functions on $V$.

A fundamental problem in the local theory of NA psh functions is whether the local definition and the global definition are equivalent. More precisely, we say that a function $\varphi$ on $X^{\an}$ is a locally smooth approximable  $\omega$-psh  function ($\varphi \in \LPSHsmo(X^{\an})$) if  any $x \in X^{\an}$ admits an open neighborhood $U$ such that $\varphi+g \in \PSHsm(U)$, where $g$ is a local potential of $\omega$ on $U$ in an appropriate sense (see Definition \ref{def:potential}).  It is not hard to prove that $\GPSHo(X) \subset \LPSHsmo(X)$. In \cite[Corollary 18.8.8]{CLD}, it is proven that $\LPSHsmo \cap \PL(X)=\GPSHo \cap \PL(X)$.  However,  it is not known at all whether in general, we have $\LPSHsmo(X)=\GPSHo(X)$.

One incompatibility between the local and global theory of psh functions is that smooth approximable psh functions are defined to be regularized by smooth functions, while $\omega$-psh functions are defined to be regularized by PL functions. 
To fill in this gap, we introduce the  notion of local psh functions as follows in our paper.  A function $u$ on an open subset $U\subset X^{\an}$ is called a tropically convex PL function if there exists $f_1,\cdots, f_k \in \CO^{\times}(U)$ and a convex PL function $h$ on $\R^k$ such that $u=h(\log|f_1|, \cdots, \log|f_k|)$ on $U$. In this paper, we say that a function $u$ is psh on $U$ ($u \in \PSH(U)$) if $u$ is generically finite and is locally a pointwise limit of a decreasing sequence of tropically convex PL functions on $U$.  Similarly, we say that a function $\varphi$ on $X^{\an}$ is locally $\omega$-psh ($\varphi \in \LPSHo(X)$) if  any $x \in X^{\an}$ admits an open neighborhood $U$ such that $\varphi+g \in \PSH(U)$ for some local potential $g$ of $\omega$  on $U$. 

We establish the following result which is of independent interest. 

\begin{theoremC} [Theorem \ref{thm:GPSHo_equals_LPSHo}+ Theorem \ref{thm:psh_equals_to_smooth_app_psh}]  \label{mainthm:equivalence_between_different_notions_of_psh_functions}
  Let $K$ be a discretely or trivially valued field.  
  \begin{enumerate}
    \item $\LPSHsmo(X)=\LPSHo(X)$.
    \item Suppose further that $\GPSHo(X)$ satisfies the envelope property. Then, $\GPSHo(X)=\LPSHsmo(X)=\LPSHo(X)$.
  \end{enumerate} 
\end{theoremC}
The first part of Theorem C relies on the regularization theorem of convex functions on tropical varieties established in \cite{APW2}, while the second part is proven via establishing a Bedford--Taylor theory and a domination principle for locally $\omega$-psh functions. Combining Theorem A and Theorem C, we obtain that $\GPSHo(X)=\LPSHsmo(X)=\LPSHo(X)$ for any projective normal variety over a discretely or trivially valued NA field of residue characteristic zero. 

\medskip 

With our local theory of non-archimedean psh functions,  we go back to the strategy for proving Theorem B. We first reduce Theorem B to showing that every weakly psh function on an affine variety $Y^{\an}$ is a (regularizable) psh function. To show this, we show that for any weakly psh function $\varphi$ on $Y^{\an}$, the open subset $D \coloneqq \{\log|w|+\varphi(z)<0\}$ is Runge in $(Y \times \BA^1)^{\an}$ (Theorem \ref{thm:NA_FN_theorem}). In our paper, an open subset $\Omega$ of an affine variety is defined to be Runge if the polynomial convex hull of any compact subset of $\Omega$ is contained in $\Omega$, and the polynomial convex hull is defined in the same way as in the complex setting (see Section \ref{subsec:polynomially_convexity_and_Runge_open_subset}). For any Runge open subset $\Omega$ in $(Y \times \BA^1)^{\an}$ containing $Y^{\an}\times \{0\}$  and for any point $y \in Y^{\an}$, we consider the distance function $d_{\Omega}(y)$ from the point $(y,0) \in (Y \times \BA^1)^{\an}$ to the boundary $\partial \Omega$.  Then, Theorem B follows from the fact that  $-\log d_{\Omega}$ is a (regularizable) psh function  on $Y^{\an}$ (Proposition \ref{prop:Runge_domains_to_psh}). 

The proof of Theorem \ref{thm:NA_FN_theorem} follows the line of the proof of \cite[Lemma 5.2]{FN80}. Key ingredients of the proof include Rossi's local maximum principle and norm (trace) of weakly psh functions under finite morphisms, whose non-archimedean versions are established in Section \ref{subsec:Rossi_local_maximum_principle} and \ref{subsec:norms_of_weakly_psh_functions} respectively. 
We would like to mention to the readers that however, the proof of Theorem \ref{thm:NA_FN_theorem} has two major differences from the proof in the complex setting. First, the proof in the complex setting makes use of  strictly psh functions, which does not known to exist in general on non-archimedean spaces. However, we construct in Section \ref{subsec:strictly_psh_functions} some functions which serve a role as strictly psh functions when $K$ admits a countable dense subfield (this is why we need this condition in Theorem B). Second, several steps in the complex analytic proof rely essentially on taking a small enough polydisk around a point $x$ in the complex Stein space. However, in the non-archimedean setting, a point $x \in Y^{\an}$ might not admits a neighborhood basis formed by polydisks. More generally, a single point might not be a polynomially convex compact subset in $Y^{\an}$. To resolve this problem, we replace these local arguments by some more global arguments on affine varieties in our paper (see Proposition \ref{prop:pshlogreg_runge_domain} and Proposition \ref{prop:weakly_psh_induces_Runge_under_injective_conditions}).

\medskip

We end this introduction by indicating the structure of the paper. In Section \ref{sec:preliminaries}, we recall some backgrounds about Berkovich spaces, tropicalizations  and the theory of $\omega$-psh functions developed in \cites{BFJ15, BJ22}. In Section \ref{sec:local_psh_functions}, we give our definition of local psh functions and study their  properties. In Section \ref{sec:Monge_Ampere_measures}, we give the definition of \MongeA operators for locally bounded psh functions, which allows us to prove the second part of Theorem C.

In Section \ref{sec:Runge_domains}, we develop a theory of polynomially convexity and Runge domains on non-archimedean spaces. In particular, we prove a non-archimedean version of Rossi's local maximum principle. In Section \ref{sec:weakly_psh_functions}, we introduce weakly psh functions and prove the non-archimedean Fornaess-Narasimhan theorem. We finally prove the envelope property in Section \ref{sec:envelope_conjecture}. In addition, in Appendix \ref{appendix:psh_functions_CLD}, we recall the definition of local psh functions of Chambert-Loir--Ducros, and explain that it is equivalent to our definition of local psh functions.

\subsection*{Notations and Conventions}
\begin{itemize}
  \item We use the standard abbreviations usc for ‘upper semicontinuous’ and psh for ‘plurisubharmonic’.
  \item We use $\interior(V)$ to denote the interior of a subset $V$ of a topological space $W$.
  \item For a topological space $W$, a  function $u \colon W \to \R \cup \{-\infty\}$ is called \emph{generically finite} if $u$ does not constantly equal to $-\infty$ on each connected component of $W$.
  \item If $W$ is a Hausdorff topological space, and $\varphi \colon W \to \R \cup \{\pm \infty\}$ is any function. Then, the \emph{usc envelope} $\varphi^{\star}$ of $\varphi$ is the smallest usc function with $\varphi^{\star} \ge \varphi$. Concretely, $\varphi^{\star}(x) = \limsup_{y \to x} \varphi(y)$. 
  \item  Let $W$ be a locally compact Hausdorff space.  The notation $\Subset$ indicates being compactly contained, \textit{i.e.}  $V \Subset W$ if the closure of $V$ is compact and contained  in $W$. 
  \item By varieties, we mean reduced schemes of finite type over a field $K$. For a variety $X$ over $K$, we use $|X|_{\cl}$ to denote the set of closed scheme points of $X$. 
  \item Let $\pi \colon X_1 \to X_2$ be a morphism between varieties over a field $K$. Let $Z_1 \subset X_1$ and $Z_2 \subset X_2$ be subvarieties. Unless indicated otherwise, $\pi(Z_1)$, $\pi^{-1}(Z_2)$ and $\overline{Z_1}$ denote the \emph{reduction of the scheme theoretical} image, preimage and Zariski closure respectively. 
\end{itemize}

\subsection*{Acknowledgements}
The author would like to thank S\'ebastien Boucksom for helpful discussions and reading drafts of this paper. He would also like to thank Antoine Ducros for answering questions about Berkovich spaces.

\subsection*{Usage of AI}
All the ideas and texts are generated by the author (who is a human being).

\section{Preliminaries} \label{sec:preliminaries}

\subsection{Berkovich spaces and tropicalizations}

Let $\valfield$ be a complete non-archimedean field  with valuation $v_K \colon \valfield \to \R\cup\{\infty\}$. Let $\valring$ denote the valuation ring of $K$, $\maxideal$ denote the maximal ideal of $\valring$, and $\resfield=\valring/\maxideal$ denote the residue field. The norm on $K$ is given by $|f| =\exp(-v(f))$ for $f \in K$. The (additive) value group $\Gamma_K \coloneqq v_K(K^{\times})$ is an additive subgroup of $\R$. We say that $K$ is trivially valued if $\Gamma_K=\{0\}$; we say that $K$ is discretely valued if $K$ is non-trivially valued and $\Gamma_K$ is a discrete subgroup of $\R$. Otherwise, $\Gamma_K$ is a dense subgroup of $\R$. In this paper, we always assume that $\Z \subset \Gamma_K$ if $K$ is non-trivially valued. 

Let $Z$ be an irreducible  projective variety over $\valfield$, with function field $\valfield(Z)$. A valuation $v$ on $Z$ is a real-valued function
\[
v \colon \valfield(Z)^{\times} \to \R
\]
satisfying $v(ab)=v(a)+v(b)$, $v(a+b) \ge \min \{ v(a), v(b)\}$ and $v|_{\valfield}=v_{\valfield}$. We use $Z^{\val}$ to denote the set of valuations on $Z$.

Let $X$ be a  projective (not necessarily irreducible) variety over $\valfield$. A semivaluation $v$ on $X$ is a valuation on some irreducible subvariety $Z$ of $X$, in which case $Z$ is called the support of $v$.  The Berkovich analytification $X^{\an}$ of $X$ is a compact Hausdorff topological space, whose underlying set is the set of semivaluations on $X$, and topology is the coarsest topology such that for each subvariety $Z \subset X$, 
\begin{itemize}
  \item the set $(X \setminus Z)^{\an} \subset X^{\an}$  of semivaluations whose support is not included in $Z$ is open;
  \item for each regular function $f \in \CO(X \setminus Z)$, the function $|f|(v) \coloneqq \exp(-v(f))$ is continuous on $(X \setminus Z)^{\an}$.
\end{itemize}
We have $X^{\an}=\coprod_{Z} Z^{\val}$ with $Z$ ranging over all irreducible subvariety of $X$.

\subsubsection{Tropicalizations} \label{subsubsec:tropicalizations}
Let $Y$ be an affine variety over $K$, and let $f_1, \cdots, f_k \in K[Y]$ be a finite set of regular functions on $Y$. Denote by $\BT \coloneqq \R \cup \{+\infty\}$. The tropicalization map $\Trop \colon Y^{\an} \to \BT^k$ induced by $\{f_1, \cdots, f_k\}$ is defined by $\Trop(y) \coloneqq (-\log|f_1|(y), \cdots, -\log|f_k|(y)) \in \BT^k$. The image of $Y^{\an}$ under $\Trop$ is a tropical variety, \textit{i.e.} the support of a polyhedral complex $\Pi$ equipped with a canonical weight function $\varpi$ satisfying the balancing condition. More generally, suppose that $X$ is a closed variety of a toric variety $\P_{\Sigma}$ over $K$ with a fixed toric structure, then we have an extended tropicalization map $\Trop \colon X^{\an} \to \TP_{\Sigma}$, where $\TP_{\Sigma}$ is the tropical toric variety corresponding to $\P_{\Sigma}$ (see \cites{CLD, APW1} for more details). 

As a topological space, $Y^{\an}$ is homeomorphic to the inverse limit of $\Trop(Y^{\an})$ over all affine embeddings of $Y$, and $X^{\an}$ is homeomorphic to the inverse limit of $\Trop(X^{\an})$ over all embeddings of $X$ into toric varieties (see \cite{Pay09}). We also remark that if $K_1 \subset K$ is a dense open subfield, then $Y^{\an}$ is homeomorphic to the inverse limit of $\Trop(Y^{\an})$ over all affine embedding of $Y$ over $K_1$ (\cite[Propsition 2.16]{Bou25}).

\subsubsection{Models and closed (semi-)positive (1,1)-forms} \label{subsubsec:models_and_NA_semipositive forms}
A \emph{model} of $X$ is a flat and proper scheme $\fX$ over $\Spec(\valring)$ together with
an identification of $K$-schemes $\fX_{K} \coloneqq \fX \otimes_{\valring} \valfield \simeq X$. When $K$ is trivially valued, a model should be understood as a test configuration, \textit{i.e.} a $\G_m$-equivariant model of $X_{K((t))}$ over $K[[t]]$ (see \cite{BJ22}). We denote by $\fX_s \coloneqq \fX \otimes_{\valring} \resfield$ its special fiber.  For a $\Q$-line bundle $L$ on $X$,  a \emph{model} of $L$ is a pair $(\fX, \fL)$ consisting of the data of a model $\fX$ of $X$ and a $\Q$-line bundle $\fL$ on $\fX$ whose restriction to $X$ (as generic fiber of $\fX$), denoted by $\fL_{\valfield}$, is isomorphic to $L$. When no ambiguity arises, we simply say  $\fL$ is a model of $L$.

We say that a model $\fX'$ dominates $\fX$ if there exists a proper birational morphism $\pi \colon \fX' \to \fX$ over $\Spec \valring$ such that $\pi|_{X}=\mathrm{id}_{X}$. We say that $(\fX', \fL')$ dominates $(\fX, \fL)$ if $\fX'$ dominates $\fX$ and $\pi^{*}\fL=\fL'$. We say that two models $(\fX', \fL')$ and $(\fX'', \fL'')$ are equivalent if there exists a model $(\fX, \fL)$ dominating both $(\fX', \fL')$ and $(\fX'', \fL'')$.

Let $\fX$ be a model of $X$, and let $\fL, \fL'$ be a $\Q$-line bundles on $\fX$.  We say that $\fL$ relatively nef on $\fX$ if $\fL|_{X}$ is nef on $X$, and for any complete curve $C$ supported in the specical fibre $\fX_s$, we have $\mathrm{deg}(\fL|_{C}) \ge 0$. We use 
\[
\Nef_{\Q}(\fX/\Spec \valring) \subset \Pic_{\Q}(\fX /\Spec (\valring))
\]
to denote the isomorphism classes of relatively nef $\Q$-line bundle on $\fX$. We say that $(\fX, \fL)$ is a nef model of $(X, L)$ if $\fL \in \Nef_{\Q}(\fX/\Spec \valring)$. 

Relatively nefness is a birational equivalent condition. More precisely, suppose that $\fX'$ dominated $\fX$ with a proper birational morphism $\pi \colon \fX' \to \fX$, and suppose $\fL \in \Pic_{\Q}(\fX /\Spec (\valring))$.  Then $\fL \in \Nef_{\Q}(\fX/\Spec \valring)$ iff $\pi^* \fL \in \Nef_{\Q}(\fX'/\Spec \valring)$. In particular, if $(\fX', \fL')$ and $(\fX'', \fL'')$ are equivalent models, then $(\fX', \fL')$ is a nef model iff $(\fX'', \fL'')$ is a nef model.  Following~\cite{BFJ15}, we define the space of \emph{closed $(1,1)$-forms} on $X^{\an}$  as  
\[
  \PLforms(X^{\an}) \coloneqq \varinjlim_{\fX} \Pic_{\Q}(\fX/\Spec (\valring) ),
\]
and define the space of closed \emph{semipositive} $(1,1)$-forms on $X^{\an}$  as  
\[
  \nef(X^{\an}) \coloneqq \varinjlim_{\fX} \Nef_{\Q}(\fX/\Spec (\valring) ),
\]
with  $\fX$ ranging over all models of $X$. For $\omega \in \PLforms(X)$, we use $L(\omega)$ to denote $\fL|_{X}$ for some (equivalently any) $(\fX, \fL)$ representing $\omega$.

\subsubsection{Metrics on line bundles}
We now recall the definition of metrics following \cite[Section 5]{BE21}. Let $L$ be a $\Q$-line bundle on $X$. An \emph{upper semicontinuous metric} (usc metric) $\phi$ on $L$ is a family of norms
\[\lvert\cdot\rvert_{(m\phi)_x}\colon (mL)_x \mathrel{:=} mL\otimes\mathcal{H}(x)\longrightarrow[0,+\infty),\quad x\in X^{\mathrm{an}},\]
with $m \ge 1$ such that $mL$ is an honest line bundle, and for any local section $s$ of $mL$ on an open $U\subset X$, the induced function $\lvert s\rvert_{m\phi}$ on $U^{\mathrm{an}}$ is usc. If $\phi,\psi$ are usc metrics on $\Q$-line bundles $L,M$, then $\phi\pm\psi$ denotes the induced metric on $L\pm M=L\otimes M^{\pm1}$.  A usc metric $\phi$ on the trivial line bundle $L=\mathcal{O}_X$ is identified with the usc function $-\log\lvert1\rvert_\phi$ on $X^{\mathrm{an}}$. In particular, if $\phi,\psi$ are two usc metrics on the same $Q$-line bundle $L$, $\phi-\psi$ is thus a usc function on $X^{\mathrm{an}}$. Conversely, suppose that $\phi$ is a usc metric on $L$ and $f \coloneqq X^{\an} \to \R \cup \{-\infty\}$ is a usc function, then 
\[
|\cdot|_{\phi_f} \coloneqq |\cdot|_{\phi} \cdot e^{-f}
\]
is a usc metric on $L$.

\subsubsection{Model metrics}

Let $\fX$ be a model and $\fL$ a line bundle on
$\fX$ with $\fL|_X = L$. Then one can
define a unique metric $\phi_{\fL}$ on $L$ with the following property: for any non-vanishing local section $\tau$ of
$\fL$ on an open set $\mathcal{U} \subset \fX$, we have
\(
    | \tau |_{\phi_{\fL}} \equiv 1
\)
on $U := \mathcal{U} \cap X$. This is well-defined since such a section $\tau$ is uniquely defined up to multiplication by a unit $u \in \mathcal{O}_{\fX}^{\times}(\CU)$ and $|u| \equiv 1$ on $U$. More generally, for a $\Q$-line bundle $\fL $ on $\fX$ with $\fL|_X = L$, we define $\phi_{\fL} \coloneqq \frac{1}{m}\phi_{m\fL}$ for some  $m \ge 1$ such that $m\fL$ is an honest line bundle. The definition of $\phi_{m\fL}$ does not depend on the choice of $m$. (See \cite[Section 5.3]{BE21} for more details.)

\begin{definition} \label{def:model_metrics_and_PL_functions}
  \begin{enumerate}
    \item A model metric on a \(\mathbb{Q}\)-line bundle \(L\) is a metric of the form \(\phi = \phi_{\fL}\), where \(\fL\) is a \(\mathbb{Q}\)-model of \(L\).
    \item A continuous function $f$ on $X^{\an}$ is called a PL (piecewise linear) function if there exists  a model metric \(\phi\) on \(\mathcal{O}_X\) such that \(f=-\log |1|_{\phi}\) on \(X^{\mathrm{an}}\). We use $\PL(X)$ to denote the space of PL functions on $X^{\an}$.
  \end{enumerate}
\end{definition}

We list several facts about model metrics and PL functions. The proofs can be found in \cite[Section 5]{BE21}.
\begin{enumerate}
  \item $\PL(X^{\an})$ is a dense subspace of $C^0(X^{\an})$.
  \item Let $(\fX', \fL')$ and $(\fX, \fL)$ are models of $(X,L)$. Then, $\phi_{\fL'}=\phi_{\fL}$ iff $(\fX', \fL')$ and $(\fX, \fL)$ are equivalent. Therefore, for any $\omega \in \PLforms(X)$, $\omega$ induces a model metric $|\cdot|_{\omega}$ on $L(\omega)$. 
  \item  Let $\omega \in \PLforms(X)$ and $f \in \PL(X)$. Then, there exists a model $\fX$ of $X$ and $\Q$-line bundles $\fL$ and $\fD_f$ representing $\omega$ and $f$ respectively such that $L=\fL|_X$, and $(\fD_f)|_X=0$. Then, $\fL_f \coloneqq \fL+\fD_f$ is a $\Q$-line bundle on $\fX$ satisfying $(\fL_f)|_{X}=L$.  We use $\omega_f$ to denote the class of $(\fX, \fL_f)$ in $\PLforms(X)$, and we write $\omega_f=\omega+\dm \dmb f$. Moreover, we have
  \[
  |\cdot|_{\omega_f} \coloneqq |\cdot|_{\omega} \cdot e^{-f} 
  \]
  as model metrics on $L$.
  \item Let $\omega, \omega' \in \nef(X)$. If $L(\omega) \simeq L(\omega')$, then there exists $\varphi \in \PL(X)$ such that $\omega' =\omega+\dm \dmb \varphi$. If $\psi \in \GPSHo \cap \PL(X^{\an})$, then $\psi \in \GPSH_{\omega+\omega'}(X^{\an})$. 
\end{enumerate}

\subsection{$\omega$-psh functions} \label{subsec:GPSH_omega_functions}
In Section \ref{subsec:GPSH_omega_functions}, we assume that $K$ be a complete non-archimedean field. 
\subsubsection{$\omega$-psh PL functions and Fubini-Study metrics}

\begin{definition}
  Let $\omega \in \nef(X)$ and let $f \in \PL(X)$. We say that $f$ is an $\omega$-psh PL function if $\omega_f =\omega +\dm \dmb f \in \nef(X)$.  We use $\GPSHo \cap \PL(X^{\an})$ to denote the space of $\omega$-psh PL function on $X^{\an}$. 
\end{definition}

\begin{definition} [{\cite[Definition 5.2]{BE21}}] \label{def:FS_metrics}
  Let $L$ be a $\Q$-line bundle over $X$ and let $\phi$ be a metric on $L$. We say that $\phi$ is a Fubini-Study metric if there exists $m \in \Z_{>0}$ such that $mL$ is a honest line bundle,  a finite set of sections $\{s_i\}_{1 \le i \le k} \subset H^0(X, mL)$ without common zeroes, and constants $C_i \in \Q$  such that 
  \[
    \phi=\frac{1}{m}\max_{1 \le i \le k}\{\log|s_i|+C_i\}.
  \]
\end{definition}

Here, the formula for $\phi$ is understood as follows. For any local nowhere vanishing section $\tau$ of $L$, we have $-\log|\tau|_{\phi}=\frac{1}{m}\max_{1 \le i \le k}\{\log|s_i/\tau^m|+C_i\}$.

The following result is a slight variant of \cite[Theorem 5.14]{BE21}. 
\begin{prop} \label{prop:equivalence_FS_metrics_relatively_ample_models}
  Let $\phi$ be a metric on an ample $\Q$-line bundle $L$ over $X$. Then, $\phi$ is a Fubini-Study metric  iff $\phi$ is a model metric of a model $(\fX, \fL)$ such that $\fL$ is relatively ample over $\fX$. \hfill \qedsymbol
\end{prop} 

\begin{definition}
  \begin{enumerate}
    \item Let $\theta \in \nef(X)$. We say that $\theta$ is a closed positive $(1,1)$-form if $L(\theta)$ is an ample $\Q$-line bundle on $X$, and $|\cdot|_{\theta}$ is a Fubini-Study metric on $L(\theta)$. We use $\amp(X)$ to denote the space of closed positive $(1,1)$-form on $X^{\an}$.
    \item Let $\theta \in \amp(X)$, and let $\varphi$ be a function on $X^{\an}$.  We say that $\varphi \in \FS_{\theta}(X^{\an})$ if $\varphi \in \GPSH_{\theta} \cap \PL(X)$ and $\theta_{\varphi} \in \amp(X)$. 
  \end{enumerate}
\end{definition}

\begin{remark}
  By \cite[Proposition 4.11]{GM19}, for any $\Q$-ample line bundle $L$ on $X$, there exists $\theta \in \amp(X)$ such that $L(\theta) \simeq L$. 
\end{remark}

\subsubsection{}
Recall that for  a  function $\varphi \colon X^{\an} \to \R \cup \{-\infty\}$ is called \emph{generically finite} if $\varphi$ does not constantly equal to $-\infty$ on each connected component of $X^{\an}$.

\begin{definition}  \label{def:glocal_omega-psh_functions}
  For a function $\varphi \colon X^{\an} \to  \R \cup \{-\infty\}$, we say that  $\varphi$ is $\omega$-psh if
  \begin{enumerate}
    \item $\varphi$ is generically finite and upper semi-continuous on $X^{\an}$.
    \item There exists $\theta \in \nef(X)$ with $c_1(L(\theta)) \in \Amp(X)$, a decreasing sequence of positive rational numbers $\{\epsilon_i\}_{i \ge 1}$ converging to $0$, and a decreasing sequence of functions $\{\varphi_i\}_{i \ge 1}$, such that $\varphi_i \in \GPSH_{\omega+\epsilon_i \theta} \cap \PL(X^{\an})$ for each $i \ge 1$, and $\varphi_i$ converges pointwise to $\varphi$. 
  \end{enumerate}
  We use $\GPSHo(X^{\an})$  to denote the class of $\omega$-psh functions on $X^{\an}$. 
\end{definition}

\begin{remark}
  When $K$ is trivially valued, \cite[Definition 4.1]{BJ22} defines an $\omega$-psh function to be the pointwise limit of a decreasing \emph{net} of ($\omega+\epsilon \theta$)-Fubini-Study functions.  However, these two definitions of $\omega$-psh functions are equivalent after some non-trivial arguments (see Proposition \ref{prop:regularization_omega_psh_by_omega_FS_functions} and \cite[Corollary 12.18]{BJ22}). Therefore, Definition \ref{def:glocal_omega-psh_functions} should be considered as an extension of \cite[Definition 4.1]{BJ22} to a general NA field.
\end{remark}

\begin{remark}
  It is a nontrivial fact that for $\varphi \in \PL(X)$, $\varphi \in \GPSHo(X^{\an})$ iff $\varphi \in \GPSHo \cap \PL(X^{\an})$. See \cite[Lemma 5.12]{BFJ16}, \cite[Corollary 1.4]{GM19} and \cite[Lemma 4.3]{BJ22}. 
\end{remark}

\begin{prop} \label{prop:def_of_omega_psh_does_not_depend_on_theta}
  Let $\omega \in \nef(X)$. Let $\varphi \in \GPSHo(X)$. 
  \begin{enumerate}
    \item For any $\theta \in \nef(X)$ such that $c_1(L(\theta)) \in \Amp(X)$, there exists a decreasing sequence of positive rational numbers $\{\epsilon_i\}_{i \ge 1}$ converging to $0$, and a decreasing function sequence  $\{\varphi_i\}_{i \ge 1}$, such that $\varphi_i \in \GPSH_{\omega+\epsilon_i \theta} \cap \PL(X^{\an})$ for each $i \ge 1$, and $\varphi_i$ converges pointwise to $\varphi$. 
    \item Suppose further that $c_1(L(\omega)) \in \Amp(X)$. Then, there exists a decreasing function sequence  $\{\psi_i\}_{i \ge 1} \subset \GPSHo \cap \PL(X)$ converging pointwise to $\varphi$ on $X^{\an}$. 
  \end{enumerate}
\end{prop}
\begin{proof}
  (1) By definition, there exists $\theta' \in \nef(X)$ with $c_1(L(\theta')) \in \Amp(X)$, a decreasing sequence $\{\epsilon'_i\}_{i \ge 1} \subset \Q_{>0}$ and a decreasing sequence $\{\varphi'_i\}_{i \ge 1}$ converging pointwise to $\varphi$ such that $\varphi'_i \in \GPSH_{\omega+\epsilon'_i \theta'}(X^{\an})$. 
  
  By \cite[Proposition 4.11 and Lemma 4.12]{GM19}, there exists a model $\fX$ of $X$ and $\Q$-line bundles $\fL, \fL', \fL''$ on $\fX$ such that $(\fX, \fL), (\fX, \fL')$ represents $\theta, \theta'$ respectively, and $\fL''$ is relatively ample on $\fX$ satisfying $\fL''|_X \simeq L(\theta)$. Then, there exists $m \ge 1$ such that $m\fL''-\fL'$ is relatively nef. This implies that $\varphi'_i \in \GPSH_{\omega+\epsilon_i  \theta''}(X^{\an})$ with $\epsilon_i \coloneqq \epsilon'_i m$ for each $i \ge 1$. 
  
  Let $\psi \in \GPSH_{\theta} \cap \PL(X^{\an})$ such that $\theta''=\theta+\dm \dmb \psi$. Then, $\varphi'_i+ \epsilon_i \psi \in \GPSH_{\omega+\epsilon_i  \theta} \cap \PL(X^{\an})$. After taking a subsequence, we may assume that $\epsilon_i \left\| \psi \right\|_{C^0(X^{\an})} < \frac{1}{2^{i+3}}$ for each $i \ge 1$. This implies that $\{\varphi_i \coloneqq \varphi'_i+ \epsilon_i \psi +\frac{1}{2^i} \}_{i \ge 1}$ is a decreasing function sequence converging to $\varphi$ such that $\varphi_i \in \GPSH_{\omega+\epsilon_i  \theta} \cap \PL(X^{\an})$ for each $i \ge 1$. 

  (2) Replacing $\varphi$ by $\varphi-\sup_{X^{\an}} \varphi-1$, we may assume that $\varphi<0$ on $X^{\an}$. Statement (1) implies that there exists a decreasing sequence $\{c_i\}_{i \ge 1} \subset \Q_{>0}$ with $c_i \to 1$ and a decreasing function sequence $\{\psi'_i\}_{i \ge 1}$ converging pointwise to $\varphi$, such that $\psi'_i \in \GPSH_{c_i \omega} \cap \PL(X^{\an})$. By Dini's theorem, we may assume that $\psi'_i \le 0$ for each $i \ge 1$. This implies that  
  \[
  \{\psi_i \coloneqq \psi'_i/c_i\}_{i \ge 1} \subset \GPSHo \cap \PL(X) 
  \]
  is a decreasing function sequence converging pointwise to $\varphi$.
\end{proof}

\begin{prop} \label{prop:regularization_omega_psh_by_omega_FS_functions}
  Let $\theta \in \amp(X)$. Then,
  \begin{enumerate}
    \item For any $\psi \in \GPSH_{\theta} \cap \PL(X^{\an})$, there exists a sequence $\{\psi_i\}_{i \ge 1} \subset \FS_{\theta}(X^{\an})$  converging uniformly to $\psi $ on $X^{\an}$.
    \item Moreover, for any $\varphi \in \GPSH_{\theta}(X^{\an})$, there exists a decreasing function sequence $\{\varphi_i\}_{i \ge 1} \subset \FS_{\theta}(X^{\an})$ converging to $\varphi$ pointwise on $X^{\an}$. 
  \end{enumerate}
\end{prop}
\begin{proof}
  For statement (1), since $L(\theta) \in \Amp(X)$, by \cite[Proposition 4.11 and Lemma 4.12]{GM19}, there exists a model $\fX$ of $X$ with $\Q$-line bundles $\fL, \fL'$ such that $(\fX, \fL)$ represents $\theta_{\psi} \in \nef(X)$,  $\fL'|_{X} \simeq L(\theta)$ and $\fL'$ is relatively ample on $\fX$. Note that $(\fL' - \fL)|_{X}=0$.  Let $\psi' \in \PL(X)$ represents $(\fX, \fL'- \fL)$. Then, for any $\epsilon >0$, $\theta+\dm \dmb (\psi+\epsilon \psi')$ represents the relatively ample model $(\fX, (1-\epsilon)\fL+\epsilon \fL')$ of $(X,L)$. Hence, $\psi+\epsilon \psi' \in \FS_{\theta}(X^{\an})$. Therefore, statement (1) is proved. Statement (2) follows from Proposition \ref{prop:def_of_omega_psh_does_not_depend_on_theta} (2) and statement (1).
\end{proof}

\begin{thm} [{\cite[Theorem 4.5 + Corollary 12.18]{BJ22}}] \label{thm:decreasing_limit_of_omega-psh_functions_is_omega-psh}
  Assume that $K$ is discretely valued or trivially valued.   Let $\{\varphi_i\}_{i \ge 1} \subset \GPSHo(X)$ is a decreasing function sequence with the pointwise limit $\varphi \colon X^{\an} \to \R \cup \{-\infty\}$. If $\varphi$ is generically finite, then $\varphi \in \GPSHo(X^{\an})$.  \hfill \qedsymbol
\end{thm}

\subsection{Divisorial points} \label{subsec:divisorial_points}
In Section \ref{subsec:divisorial_points}, we assume that $K$ is discretely valued or trivially valued. Let $\fX$ be a model of $X$, and $\fX^{\nu}$ be the normalization of $\fX$. Let $E$ be an irreducible component of the special fibre $(\fX^{\nu})_s$. Note that $\CO_{\fX^{\nu}, E}$ is a discrete valuation ring. This gives a unique  point $v_E \in X^{\an}$ by rescaling the  valuation $\ord_{E} \colon K(\fX^{\nu})^{\times}= K(X)^{\times} \to \R$. Such points are called divisorial points. We use $X^{\div}$ to denote the set of divisorial points in $X^{\an}$. In particular, $X^{\div} \subset X^{\val}$.

\begin{prop}  \label{prop:divisorial_points_are_dense}
  Assume that $K$ is discretely valued or trivially valued. Then, $X^{\div}$ is dense in $X^{\an}$. 
\end{prop}
\begin{proof}
  This is proven in \cite[Proposition 2.4.9]{MN15} when $K$ is discretely valued, and in \cite[Theorem 2.14]{BJ22} when $K$ is trivially valued. 
\end{proof}

We recall the following result in \cite{BJ22}. These results are proven in \cite{BJ22} only in the setting when $K$ is trivially valued, but the proof work also in the setting when $K$ is discretely valued. 

\begin{thm} [{\cite[Proposition 1.28+Corollary 4.17]{BJ22}}] \label{thm:global_divisorial_points_are_nonpluripolar}
  Let $\varphi \in \GPSHo(X)$.  Then, $\varphi >-\infty$ on $X^{\div}$. 
\end{thm}

\begin{thm} [{\cite[Theorem 4.22]{BJ22}}] \label{thm:global_comparison_divisorial_points_to_full_spaces}
  Let $\varphi \in \GPSHo(X)$, and let $\psi \colon X^{\an} \to \BR \cup \{\pm \infty\}$ be a usc function. Suppose that $\varphi \le \psi$ on $X^{\div}$. Then, $\varphi \le \psi$ on $X^{an}$.  In particular, for any $y \in X^{\an}$, we have
  \[  
  \varphi(y)=\limsup_{x \to y, \, x \in X^{\div}} \varphi(x).
  \]
\end{thm}

\begin{thm} [{\cite[Theorem 5.6]{BJ22}}] \label{thm:global_divisorial_points_is_nonnegligible}
  Let $\{\varphi_i\}_{i \in I} \subset \GPSHo(X)$ be a class of functions. Then, \(  \sup_{i \in I} \{\varphi_i\}  = (\sup_{i \in I} \{\varphi_i\})^{\star} \) on $X^{\div}$.
\end{thm}

Now we recall the envelope property for non-archimedean $\omega$-psh functions.

\begin{definition} [{\cite[Definition 5.8]{BJ22}}]
  Let $\omega \in \nef(X)$. We say that $\GPSHo(X)$ satisfies the envelope property if, for any uniformly bounded above family $\{\varphi_i\}_{i \in I} \subset \GPSHo(X)$, the usc upper envelope $(\sup_{i \in I} \varphi_i)^{\star} \in \GPSHo(X)$. We say that $X$ satisfies the envelope property over $K$ if $\GPSHo(X)$ satisfies the envelope property for any $\omega \in \nef(X)$. 
\end{definition}

\begin{conj} [Envelope conjecture, {\cite[Conjecture 5.14]{BJ22}}]
  Let $K$ be a discretely valued or trivially valued  non-archimedean field, and $X$ be unibranch projective variety over $K$. Then $X$ satisfies the envelope property over $K$. 
\end{conj}

\begin{thm}  [{\cite[Theorem 5.20]{BJ22}}] \label{thm:smooth_envelope_property}
  If $K$ is a discretely valued or trivially valued  non-archimedean field of equicharacteristic $0$ and $X$ is a smooth projective variety over $K$, then $X$ satisfies the envelope property over $K$. 
\end{thm}

For a function $f \colon X^{\an} \to \R \cup \{\pm \infty\}$, the $\omega$-psh envelope of $f$ is defined as 
\[
\envomega(f) \coloneqq \{\varphi \in \GPSHo(X) \,|\, \varphi \le f\}.
\]

The following result is a slight variant of \cite[Lemma 5.17]{BJ22}. 

\begin{lemma}  \label{lem:equivalent_form_of_envelope_conjecture}
  Let $\omega \in \nef(X)$ such that $c_1(L(\omega)) \in \Amp(X)$. Then, the followings are equivalent.
  \begin{enumerate}
    \item $\GPSHo(X)$ satisfies the envelope property.
    \item $\envomega(f) \in C^0(X^{\an})$ for any $f \in \PL(X)$. 
    \item $\envomega(f) \in \GPSHo \cap C^0(X^{\an})$ for any $f \in C^0(X^{\an})$. 
  \end{enumerate}
\end{lemma}

\begin{prop} \label{prop:limit_of_GPSH_functions}
  Assume that $\GPSHo(X)$ satisfies the envelope property. Let $\{\varphi_i\}_{i \ge 1}$ be a bounded-above sequence of functions in $\GPSHo(X)$. Suppose that $\varphi_i$ converges pointwise to $\varphi$. Then, $\varphi^{\star} \in \GPSHo(U)$ and $\varphi^{\star}=\varphi$ on $X^{\div}$. 
\end{prop}
\begin{proof}
  For each $k \ge 1$, define $\psi_k \coloneqq (\sup_{i \ge k} \varphi_i)^{\star}$. The assumption that $\GPSHo(X)$ satisfies the envelope property implies that $\{\psi_{k}\}_{k \ge 1} \subset \GPSHo(X)$. Note that $\{\psi_{k}\}_{k \ge 1}$ is a decreasing function sequence. Let $\psi$ be the pointwise limit of $\{\psi_{k}\}_{k \ge 1}$. Then, Theorem \ref{thm:decreasing_limit_of_omega-psh_functions_is_omega-psh} implies that $\psi \in \GPSHo(X)$. Observe that $\varphi=\lim_{k \to +\infty} (\sup_{i \ge k} \varphi_i)$. Therefore, we have $\varphi \le \psi$ and hence $\varphi^{\star} \le \psi^{\star}=\psi$. On the other hand,  Theorem \ref{thm:global_divisorial_points_is_nonnegligible} implies that $\varphi=\psi$ on $X^{\div}$. Thus, by Theorem \ref{thm:global_comparison_divisorial_points_to_full_spaces}, we have $\varphi^{\star} \ge \psi$. Combining all these together,  the proposition is proved.      
\end{proof}

\subsection{\MongeA measures} \label{subsec:MA_measures_omega_psh_functions}
In Section \ref{subsec:MA_measures_omega_psh_functions}, we assume that $K$ is discretely valued or trivially valued. 
Let $\omega \in \PLforms(X)$ and $\varphi \in \PL(X)$. Let $(\fX, \fL)$ be a model corresponding to $\omega_{\varphi} \in \PLforms(X)$.  Let $\nu \colon \fX^{\nu} \to \fX$ be the normalization map, and $\fL^{\nu} \coloneqq (\nu)^*(\fL)$. Suppose that $(\fX^{\nu})_s=\div(\mathsf{t})=\sum_{j \in J} b_j E_j$ on $\fX^{\nu}$, where $\mathsf{t}$ is a generator of $\maxideal$, and $\{E_j\}_{j \in J}$ is the set of irreducible components of $(\fX^{\nu})_s$.  Denote by $\delta_{v_{E_j}}$ the Dirac mass supported at the divisorial point $v_{E_j}$ for each $j \in J$.  Then, the non-archimedean \MongeA measure for the PL function $\varphi$ is a Radon measure on $X^{\an}$ defined by
\[
\MA_{\omega, \NA} (\varphi) \coloneqq \sum_{j \in J} b_j ((\fL_{D}^{\nu})|_{E_j})^d \, \delta_{v_{E_j}}
\]

We have the following Bedford-Taylor theory for non-archimedean psh functions. 
\begin{thm} [{\cite[Theorem 7.18, Proposition 7.21]{BJ22}}] \label{thm:Bedford-Taylor_global_omega-psh}
  Let $X$ be a projective variety over $K$ of dimension $n$. Let $\omega_0, \omega_1 \cdots, \omega_n \in \nef(X)$.  Then, there is a mixed \MongeA unique operator
\bens
  \left( \GPSH_{\omega_1} \cap L^{\infty}(X^{\an}) \right) \times \cdots \times \left( \GPSH_{\omega_n} \cap L^{\infty}(X^{\an}) \right)  &\longrightarrow& \Radonp{X^{\an}} \\
  (\varphi_1, \cdots, \varphi_n) &\longmapsto& (\omega_1+\dm \dmb \varphi_1) \wedge \cdots \wedge (\omega_n+\dm \dmb \varphi_n)
\eens
such that
\begin{enumerate}[label=$(\arabic*)$]
  \item It extends the \MongeA measures for $\omega$-psh PL functions.
  \item Let $\varphi_k \in \GPSH_{\omega_k} \cap L^{\infty}(X^{\an})$ for each $0 \le k \le n$. Denote $\mu \coloneqq \bigwedge_{1 \le k \le n} (\omega_k+\dm \dmb \varphi_k)$. Then, $\varphi_0 \in L^1(X^{\an}, \mu)$. 
  \item Under the hypothesis of (2), suppose further that for each $0 \le k \le n$, $\{\varphi_k^j\}_{j \ge 1} \subset \GPSH_{\omega_k}  \cap \PL(X^{\an})$ is a  decreasing function sequence converging pointwise to $\varphi_k$.  Denote $\mu_j \coloneqq \bigwedge_{1 \le k \le n} (\omega_k+\dm \dmb \varphi_k^j) $ for each $j \ge 1$. Then, $\varphi_0^j \mu_j$ converges weakly to $\varphi_0 \mu$ on $X^{\an}$. 
\end{enumerate}
\end{thm}
For $\varphi \in \GPSHob(X^{\an})$, its \MongeA measure is defined to be 
\[
\MAo(\varphi) \coloneqq (\omega+\dm \dmb \varphi)^{\wedge n}.
\]
We also recall the locality for \MongeA measures of $\omega$-psh functions.
\begin{thm} [{\cite[Theorem 7.40, Corollary 7.41]{BJ22}}] \label{thm:locality_MA_omega-psh} 
  Let $\omega \in \nef(X)$ with $c_1(L(\omega)) \in \Amp(X)$.
  \begin{enumerate}
    \item For all $\varphi,\varphi'\in \GPSHob(X^{\an})$, we have
    \[
    \mathbf{1}_{\{\varphi>\varphi'\}}\MA\bigl(\max\{\varphi,\varphi'\}\bigr)=\mathbf{1}_{\{\varphi>\varphi'\}}\MA(\varphi)
    \]
    \item Let $G\subset X^{\mathrm{an}}$ be an open set. If  $\varphi_i,\psi_i\in \GPSH_{\omega_i} \cap L^{\infty} (X^{\an})$, $1\leq i\leq n$,
    are such that $\varphi_i=\psi_i$ on $G$, then
    \[
    (\omega_1+\mathrm{dd}^c\varphi_1)\wedge\cdots\wedge(\omega_n+\mathrm{dd}^c\varphi_n)
    =
    (\omega_1+\mathrm{dd}^c\psi_1)\wedge\cdots\wedge(\omega_n+\mathrm{dd}^c\psi_n)
    \quad\text{on }G.
    \]
  \end{enumerate}
\end{thm}

\subsection{}

Let $X, Z$ be projective varieties over $K$, and let $\pi \colon X \to Z$ be a proper morphism. Let $z \in Z^{\an}$, and $\CH(z)$ be the completed residue field of $z$. Consider the morphism $\CH(z) \times_Z X \to X$ and its  analytification of $(\CH(z) \times_Z X)^{\an} \to Z^{\an}$. Here $(\CH(z) \times_Z X)^{\an}$ is viewed as Berkovich analytification of over $\CH(z)$. Then, this map induces a homeomorphism between $(\CH(z) \times_Z X)^{\an}$ and $(\pi^{\an})^{-1}(z) \subset X^{\an}$. Moreover, for any $\omega \in \nef(X^{\an})$ and $\varphi \in \GPSHo(X^{\an})$,  we have $\omega|_{(\pi^{\an})^{-1}(z)} \in  \nef(\CH(z) \times_Z X)$ and $\varphi|_{(\pi^{\an})^{-1}(z)} \in \GPSH_{\omega|_{(\pi^{\an})^{-1}(z)}} ((\CH(z) \times_Z X)^{\an})$ under this homeomorphism. To see this, denote $\BK \coloneqq \CH(z)$, and let $\pi_{\BK} \colon X_{\BK} \to Z_{\BK}$ be the base change of $\pi$  to $K$. Then, $z$ can be lifted to a $\BK$-rational point in $(Z_{\BK})^{\an}$. This induces a lift of the morphism $\CH(z) \times_Z X \to X$ to a $\BK$-morphism $\CH(z) \times_Z X \to X_{\BK}$. Note that an $\omega$-psh functions (\emph{resp.} closed semipositive form) remains being psh (\emph{resp.} closed semipositive form) under base change and restrictions to closed subvarieties. Then, it follows that $\omega|_{(\pi^{\an})^{-1}(z)}$ is semipositive and $\varphi|_{(\pi^{\an})^{-1}(z)}$ is psh.

\begin{prop} \label{prop:descent_psh_functions_resolution}
  Assume that $K$ is discretely valued or trivially valued. Let $X, Z$ be projective varieties over $K$, and let $\pi \colon X \to Z$ be a proper surjective morphism. Let $\omega \in \nef(Z)$ and $\varphi  \in \GPSH_{\pi^{*}\omega}(X)$. Suppose that every geometric fibre of $\pi$ is connected. Then, there exists a usc function $\psi \colon Z^{\an} \to \R \cup \{-\infty\}$ such that 
  \begin{enumerate}
    \item $\varphi=\psi \circ \pi$. 
    \item For any $z \in Z^{\an}$, we have 
    \[
    \psi(x)=\limsup_{y \to z,\, y \in Z^{\div}} \psi(z).
    \]
  \end{enumerate} 
\end{prop}

We first recall the following result. 
\begin{lemma} \label{lem:zero-psh_functions_are_constant}
  Suppose that $K$ is a discrete valued or trivially valued field, and $X$ is connected. Let $\omega=0 \in \nef(X)$. Then, every $\varphi \in \GPSH_{\omega}(X^{\an})$ is a constant function.
\end{lemma}
\begin{proof}
  The proof is similar to the proof of \cite[Corollary 4.24]{BJ22}. We reproduce here for the convenience of readers. By GAGA, $X^{\an}$ is connected.   After adding a constant to $\varphi$, we may assume that $\sup_{X^{\an}} \varphi=0$. Suppose that $\varphi=0$ on $X^{\div}$, then Theorem \ref{thm:global_comparison_divisorial_points_to_full_spaces} implies that $\varphi=0$. Otherwise, suppose that there exists $x \in X^{\div}$ such that $\varphi(x)<0$. Let $\psi$ be the pointwise limit of the decreasing function sequence $\{-m\varphi\}_{m \ge 1}$. By assumption, there exists $y \in X^{\an}$ such that $\varphi(y)=0$ and thereby $\psi(y)=0$. This implies that $\psi$  is generically finite on $X^{\an}$, and therefore $\psi \in \GPSH_{\omega}(X^{\an})$. On the other hand, we have $\psi(x)=-\infty$, which contradicts with Theorem \ref{thm:global_divisorial_points_are_nonpluripolar}.
\end{proof}

Now, we are ready to prove Proposition \ref{prop:descent_psh_functions_resolution}.

\begin{proof} [Proof of Proposition \ref{prop:descent_psh_functions_resolution}] \label{proof:prop:descent_psh_functions_resolution}
  \textbf{Step 1}. Take $z \in Z^{\an}$. Assume that there exists an irreducible subvariety $T$ of $Z$ such that $z \in T^{\div}$. Then, $\CH(z)$ is a discretely or trivially valued field. By assumption,  $\CH(z) \times_Z X$ is a connected scheme. Note that 
  $\omega|_{(\pi^{\an})^{-1}(z)}=0 \in \nef(\CH(z) \times_Z X)$. Therefore, Lemma \ref{lem:zero-psh_functions_are_constant} implies that $\varphi|_{(\pi^{\an})^{-1}(z)}$ is constant in this case. 

  \smallskip

  \textbf{Step 2}. We show that $\varphi$ is constant on $(\pi^{\an})^{-1}(z)$ for any $z \in Z^{\an}$. We may assume that $Z$ is irreducible, and $\dim Z=n$. By generic flatness (\cite[Proposition 29.28.1]{stacks-project}), there exists a descending sequence of subvarieties $Z=Z_0 \supsetneq Z_1 \supsetneq \cdots \supsetneq Z_n \supsetneq Z_{n+1}=\emptyset$ such that for each $0 \le k \le n$, we have $\dim Z_{k+1} < \dim Z_k$, and $\pi$ is flat over $Z_k \setminus Z_{k+1}$. Let $0 \le k \le n$ such that $z \in (Z_k \setminus Z_{k+1})^{\an}$.  Denote $W \coloneqq Z_k \setminus Z_{k+1}$, and  $U \coloneqq \pi^{-1}(W)$.  For any $y \in W^{\div}=(Z_k)^{\div}$,  we've already shown that $\varphi|_{(\pi^{\an})^{-1}(y)}$ is constant, and we thereby denote $\psi(y) \coloneqq \varphi((\pi^{\an})^{-1}(y))$. 
  
  Recall that   \cite[Theorem 9.2.3]{Duc18} shows that $\pi^{\an}|_{U^{\an}} \colon U^{\an} \to W^{\an}$ is an open map. \lnote{see also Theorem 6.6 in Ducros flatness survey.} Fix $\widetilde{z} \in (\pi^{\an})^{-1}(z)$, and let $\{\mathscr{U}_i\}_{i \in I}$ be an open neighborhood basis of $\widetilde{z}$ in $U^{\an}$. Then, $\{\pi^{\an}(\mathscr{U}_i)\}_{i \in I}$ is open neighborhood basis of $z$ in $W^{\an}$. Since $U^{\div}$ is dense in $U^{\an}$, these imply that for any net $\{y_j\}_{j \in J} \subset W^{\div}$ satisfying $y_j \to z$, there exists $\{\widetilde{y_j}\} \subset U^{\div}$ such that $\widetilde{y_j} \to \widetilde{z}$, and $\widetilde{y_j} \in (\pi^{\an})^{-1}(y_j)$ for each $j \in J$. Conversely, for any net $\{\widetilde{y_j}\} \subset U^{\div}$ such that $\widetilde{y_j} \to \widetilde{z}$, $\{\pi^{\an}(\widetilde{y_j})\}$ is a net of points in $W^{\div}$ satisfying  $\pi^{\an}(\widetilde{y_j}) \to z$. Combining these together, we conclude that
  \[
    \limsup_{y \to z,\, y \in W^{\div}} \psi(y) = \limsup_{\widetilde{y} \to \widetilde{z},\, \widetilde{y} \in U^{\div}} \psi(\widetilde{y})= \limsup_{\widetilde{y} \to \widetilde{z},\, \widetilde{y} \in U^{\div}} \varphi(\widetilde{y}) =\varphi(\widetilde{z}). 
  \]
  where the last equality follows from Theorem \ref{thm:global_comparison_divisorial_points_to_full_spaces}. This implies in particular that $\varphi$ is constant on $(\pi^{\an})^{-1}(z)$. 

  \smallskip
  
  \textbf{Step 3}. For each $z \in Z^{\an}$, define $\psi(z) \coloneqq \varphi((\pi^{\an})^{-1}(z))$. Let $C \in \R$. Since $\varphi$ is usc, $\{\varphi  \ge C\}$ is a closed subset of $X^{\an}$.   Since $X^{\an}$ is a compact Hausdorff space, $\pi^{\an}$ is a closed map. This implies that $\{\psi \ge C\}=\pi^{\an} \left( \{u  \ge C\} \right)$ is a closed subset of $Z^{\an}$. This implies that $\psi$ is usc on $Z^{\an}$. On the other hand, we have $\pi^{\an}(X^{\div}) \subset Z^{\div}$, which implies that  
  \[
  \psi(z)=\varphi(\widetilde{z})=\limsup_{\widetilde{y} \to \widetilde{z},\, \widetilde{y} \in X^{\div}} \varphi(\widetilde{y}) = \limsup_{\widetilde{y} \to \widetilde{z},\, \widetilde{y} \in X^{\div}} \psi(\pi(\widetilde{y})) \le \limsup_{y \to z,\, y \in X^{\div}} \psi(y).
  \]
  where $\widetilde{z} \in (\pi^{\an})^{-1}(z)$. Combined these together, the proposition is proved.
\end{proof}

\section{A local theory of PSH functions} \label{sec:local_psh_functions}
In this section, unless indicated otherwise, we assume that $K$ is any completed NA field,  $X$ is a projective variety over $K$, and  $Y$ is an affine variety over $K$. 

\subsection{Tropically convex functions}
Let $\BT \coloneqq \R \cup \{+\infty\}$. We make the convention that $x+(+\infty)=+\infty$ for any $x \in \BT$, and $0 \cdot (+\infty)=0$, $a \cdot (+\infty)=+\infty$ for any $a \in \R_{>0}$, $a \cdot (+\infty)=-\infty$ for any $a \in \R_{<0}$.  We equip $\BT$ with the usual extended topology.  Let $U \subset \BT^k$ be an open subset. 

\begin{definition}
  \begin{enumerate}
    \item A $\Q$-affine (\emph{resp.} $\R$-affine) function $l$ on $U$ is a \emph{real-valued continuous}  function $l \colon U \to \R$  of the form
    \[
    l(\ft)=\sum_{1 \le i \le k} a_i t_i +b \qquad \text{for any } \ft=(t_1, \cdots, t_k) \in U,
    \]
    with $a_1, \cdots, a_k, b \in \Q$ (\emph{resp.} $a_1, \cdots, a_k, b \in \R$), such that for any $\ft=(t_1, \cdots, t_k) \in U$ and for each $1 \le i \le k$, we have $-\infty < a_i t_i <+\infty $.  We use $\Aff_{\Q}(U)$ (\emph{resp.} $\Aff_{\R}(U)$) to denote the space of $\Q$-affine (\emph{resp.} $\R$-affine) functions on $U$.
    \item A $\Q$-PL (\emph{resp.} $\R$-PL) function $f$ on $U$ is a \emph{real-valued continuous}  function $f \colon U \to \R$ such that there exists a \emph{finite open cover} $\{U_{\alpha}\}_{\alpha \in \CI}$, and a \emph{finite collection} of $\Q$-affine (\emph{resp.} $\R$-affine) functions $\{l_{\alpha, i}\}_{i \in \CI_{\alpha}}$ for each $\alpha \in\CI$ such that 
    \[
      f(x) \in \{l_{\alpha,i}(x)\}_{\alpha \in \CI,\, i \in \CI_{\alpha}} \qquad \text{for any } x \in U_{\alpha}. 
    \]
    We use $\PL_{\Q}(U)$ (\emph{resp.} $\PL_{\R}(U)$) to denote the space of $\Q$-PL (\emph{resp.} $\R$-PL) functions on $U$. When there is no ambiguity, we write $\PL(U) \coloneqq \PL_{\Q}(U)$. 
    \item A convex function $f$ on $U$ is a \emph{real-valued continuous} function $f \colon U \to \R$ such that for any $x \in U$, there exists an open neighborhood $V$ of $x$ in $U$, $C  \in \R$ and $l \in \Aff_{\R}(V)$ such that 
    \[
    u(x)=l(x)+C, \quad \text{and} \quad u(x') \ge l(x') \ \text{for any } x' \in V. 
    \] 
  \end{enumerate}
\end{definition}

\begin{remark} \label{rmk:PL_functions_near_boundaries}
  Let $x=(\underline{\infty}, y) \in \BT^{k_1} \times \R^{k_2} \subset \BT^{k}$ with $\underline{\infty} =(+\infty, \cdots, +\infty) \in \BT^{k_1}$ and $k=k_1+k_2$. Suppose that $f$ is a PL (\emph{resp.} affine) function on an open subset of $\BT^{k}$ containing $x$. Then, by definition, there exists a neighborhood $U$ of $x$ in $\BT^{k}$ and $f' \in \PL(p(U))$ such that $p(U) \subset \R^{k_2}$ and $f=f' \circ p$ on $U$, where $p \colon \BT^{k} \to \BT^{k_2}$ denotes the projection map. 
\end{remark}

\begin{remark}
  Here we give some non-examples of affine functions and PL functions.
  \begin{enumerate}
    \item $l(t_1, t_2)=2t_1-3t_2$ is not an affine function in a neighborhood of $(+\infty, +\infty)$ in $\BT^2$ since $2 \cdot (+\infty)=+\infty$ and $(-3) \cdot (+\infty)=-\infty$. 
    \item In our definition, a  PL (\emph{resp.} convex) function is always assumed to take values in $\R$. For example, $f(x)=x$ is not a PL (\emph{resp.}  convex) function on $\BT^1$ since $f(+\infty)=+\infty$. 
    \item $f(t_1, t_2)=\max\{t_1-t_2, 0\}$ is not a PL function on $\BT^2$ because it is not continuous around $(+\infty, +\infty)$. 
  \end{enumerate}
\end{remark}

\begin{lemma} \label{lem:property_tropical_conv_functions}
  \begin{enumerate}
    \item Let $f \in C^0(U)$. Suppose that $f|_{U \cap \R^k}$ is convex. Then, $f \in \Conv(U)$.
    \item Let $\{f_i\}_{i \in I} \subset \Conv(U)$ be a finite class of convex functions. Then, $\max \{f_i\}_{i \in I} \in \Conv(U)$.
    \item Suppose that $\overline{U}$ is compact. Let $f \in \Conv(\BT^k)$.   
    Then, for any $\epsilon >0$, there exists $g \in \Conv \cap \PL(U)$ such that $\left\|f-g\right\|_{C^0(U)} < \epsilon$. 
    \item For any $f \in \PL(U)$, there exists $f^+, f^- \in \TConv \cap \PL(U)$ such that $f=f^{+}-f^{-}$.
  \end{enumerate}
\end{lemma}
\begin{proof}
  Statement (1), (2), (3) are proven respectively in Proposition 3.10, Proposition 3.15 (3), Theorem 5.14 of \cite{APW2}. Statement (4) is classical in convex analysis.
\end{proof}

Now, let $X$ be a projective variety over $K$. 
\begin{lemma} [{\cite[Theorem 1]{GM19}}] \label{lem:model_functions_equals_to_PL_functions}
  Let $u$ be a function on $X^{\an}$. Then, $u \in \PL(X)$ (after Definition \ref{def:model_metrics_and_PL_functions} in our paper) iff there exists a finite affine open cover $\{Y_i\}_{1 \le i \le r}$ of $X$, tropicalization maps $\Trop_i \colon Y^{\an}_i \to \BT^{k_i}$ and $f_i \in \PL(\BT^{k_i})$ for each $i$, such that $u=f_i \circ \Trop_i$ on $Y_i^{\an}$.  \hfill \qedsymbol
\end{lemma}

Therefore, we introduce the following definitions. 

\begin{definition}
  Let $W$ be an open subset in $X^{\an}$, and $u$ is a function on $W$. We say that $u$ is a PL function on $W$ if there exists a finite affine open cover $\{Y_i\}_{1 \le i \le r}$ of $X$, tropicalization maps $\Trop_i \colon Y^{\an}_i \to \BT^{k_i}$ and $f_i \in \PL(\Trop_i(W \cap Y_i^{\an}))$ for each $i$, such that $u=f_i \circ \Trop_i$ on $W \cap Y_i^{\an}$. We use $\PL(W)$ to denote the space of PL functions on $W$.
\end{definition}

\begin{definition}
  Let $W$ be an open subset in $X^{\an}$, and $u \colon W \to \R$ be a function. We say that $u$ is tropically convex if there exists an affine open subvariety $Y \subset X$, a tropicalization map $\Trop \colon Y^{\an} \to \BT^n$ and $f \in \Conv  (\Trop(W))$ such that $W \subset Y^{\an}$, and $u=f \circ \Trop$.  We use $\TConv(W)$ to denote the space of tropically convex functions on $W$, and use $\TConv \cap \PL(W)$ to denote the space of tropically convex PL functions on $W$. 
\end{definition}

\subsection{Local psh functions}
In this subsection, we assume that $X$ is a projective variety over $K$, and $U \subset X^{\an}$ is an open subset. 

\begin{definition} \label{def:psh_functions_on_local_domains}
  Let $u \colon U \to \R \cup \{-\infty\}$ be a function. We say that  $u$ is a psh function if 
  \begin{enumerate}
    \item $u$ is generically finite and usc on $U$.
    \item For each $x \in U$, there exists an open neighborhood $U_x$ of $x$ in $U$, and a decreasing sequence of functions $\{u_i\}_{i \ge 1} \subset \TConv \cap \PL(U_x)$ which converges pointwise to $u$ on $U_x$. 
  \end{enumerate}
  We say that $u$ is locally bounded if for each $x \in U$, there exists an open neighborhood $V_x$ of $x$ in $U$, such that $u$ is bounded on $V_x$. 

  We use $\PSH(U)$ to denote the set of psh functions on $U$, and use $\PSHb(U)$ to denote the set of locally bounded psh functions on $U$. 
\end{definition}

\begin{remark} \label{rmk:psh_is_analytic_invariant}
  Let  $\CO_{U}$ denote the structure sheaf of the Berkovich space $U$, \textit{i.e.} the sheaf of analytic functions on $U$. For an open subset $W \subset U$ and a function $u$ on $W$,  we say that $u$ is \emph{analytically tropically convex PL function} if there exists $g_1, \cdots, g_k \in \CO(W)$ and $f \in \Conv \cap \PL (\BT^k)$ such that $u=f \circ \Trop$, where $\Trop \colon W \to \BT^k$ is the tropicalization map defined by $(-\log|g_1|, \cdots, -\log|g_k|)$.  
  
  Let $x \in U$, and let $V$ be a  rational domain $V$ of $X^{\an}$ such that $x \in V \subset U$. Then, there exists an affine open subvariety $Y \subset X$ with $V \subset Y^{\an}$. Moreover, for any $p \in \CO(V)$, there exists a sequence of polynomials $\{p_i\}_{i \ge 1} \subset K[Y]$ such that $|p_i|$ converges uniformly to $|p|$. On the other hand, consider the map $-\mathrm{Log} \colon \R^k_{\ge 0} \to \BT^k$ defined by $-\mathrm{Log}(t_1, \cdots, t_k) \coloneqq (-\log t_1, \cdots, -\log t_k)$, and $f \in \PL(\BT^k)$. With Remark \ref{rmk:PL_functions_near_boundaries}, it's not hard to show that $f \circ (-\mathrm{Log})$ is Lipschitz  on any compact subset of $\R^k_{\ge 0}$.  Combining these together, we obtain that for any analytically tropically convex PL function $u$ on $V$, there exists a sequence $\{u_i\}_{i \ge 1} \subset \TConv \cap \PL(V)$  converging locally uniformly to $u$ on $V$.

  Thus, the pointwise limit of any decreasing  sequence of analytically tropically convex PL function on $V'$ can be written locally as the pointwise limit of a decreasing sequence of tropically convex PL function, and thereby is a psh function (after Definition \ref{def:psh_functions_on_local_domains}) on $V$. This shows that the definition of $\PSH(U)$ only depends on the analytic structure on $U$. In particular, it does not depend on the choice of the projective variety $X$ containing $U$. 
\end{remark}

\begin{definition} \label{def:potential}
  Let $\omega \in \PLforms(X)$. A potential function of $\omega$ on $U$ is a \emph{finite-valued} PL function $g \in \PL(U)$ which can be written as $g=-\frac{1}{m}\log|s|_{\omega}+\lambda$, where $m \ge 1$ such that $mL(\omega)$ is an actual line bundle on $X$,  $s \in H^0(U, mL(\omega)|_{U})$ is a nowhere vanishing section, and $\lambda \in \Q$. 
\end{definition}

\begin{remark} \label{rmk:existence_of_potentials}
  \begin{enumerate}
    \item For any $x \in U$, there exists an affine open neighborhood $V $ of $x$ in $X^{\an}$ and a nowhere vanishing section $s \in H^0(V, L|_{V})$. Therefore, $\omega$ admits a potential   $g \in \PL(V)$ on $U \cap V$. 
    \item Let $g_1, g_2$ be potentials of $\omega$ on $U$. Then, $g_1-g_2=\log|f|$ for some $f \in \CO^{\times}(U)$. \qedhere
  \end{enumerate}
\end{remark}

\begin{definition} \label{def:locally_omega-psh_functions}
  Let $\omega \in \nef(X)$ and let $\varphi \colon X^{\an} \to \R \cup \{-\infty\}$ be a function. We say that  $\varphi$ is a \emph{locally $\omega$-psh function} if for each $x \in X^{\an}$, there exists an open neighborhood $U$ of $x$ in $X^{\an}$ and a potential $g \in \PL(U)$ of $\omega$ such that $\varphi+g \in \PSH(U)$. In particular,  $\varphi$ is generically finite and usc on $X^{\an}$.
\end{definition}

\begin{lemma} \label{lem:finite_sup_lpsh}
  \begin{enumerate}
    \item Let $\{u_i\}_{i \in I} \subset \PSH(U)$ be a finite class of functions. Then, $\max \{u_i\}_{i \in I} \in \PSH(U)$.  
    \item Let $\{\varphi_i\}_{i \in I} \subset \LPSHo(X)$ be a finite class of functions. Then, $\max \{\varphi_i\}_{i \in I} \in \LPSHo(X)$.  
  \end{enumerate}
  
\end{lemma}
\begin{proof}
  The lemma follows from the observation that the supreme of any finite collection of tropically convex functions is tropically convex, which is an easy corollary of  Lemma \ref{lem:property_tropical_conv_functions} (2).
\end{proof}

\begin{remark} \label{rmk:extension_of_psh_functions_by_sup}
  Let $V \subset U \subset X^{\an}$ be open subsets. Let $u \in \PSH(U)$ and $v \in \PSH(V)$.    Define the function $u' \colon U \to \R \cup \{-\infty\}$ by 
  \[
    u'=\max\{u,v\} \text{ on } V, \qquad \text{ and } \qquad u'=u \text{ on } U \setminus V.
  \]
  Set $W \coloneqq \{x \in V, \, u(x)<v(x)\}$, and assume that  $\overline{W} \subset V$. Then, $u' \in \PSH(U)$. Moreover, if $u \in \PSH \cap \PL(U)$ and $v \in \PSH \cap \PL(V)$, then $u' \in \PSH \cap \PL(U)$.  To see this, we have $u'|_V \in \PSH(V)$ by Lemma \ref{lem:finite_sup_lpsh}. On the other hand, since $U$ is locally compact Hausdorff, for any $x \in U \setminus V$, there exists an open neighborhood $U_x$ of $x$ in $U$ such that $U_x \cap \overline{W}=\emptyset$. Then, by assumption, $u'|_{U_x}=u|_{U_x} \in \PSH(U_x)$. 
\end{remark} 

\subsection{Fubini-Study functions}
Let $Y$ be an affine variety over $K$.

\begin{definition}  \label{def:psh_functions_log_growth}
  \begin{enumerate}
    \item A function $u$ on $Y^{\an}$ is called a Fubini-Study function if there exists $f_1, \cdots, f_k \in K[Y]$, $c_1, \cdots, c_k \in \Q_{>0}$ and $C_1, \cdots, C_k, \lambda \in  \Q$ such that $f_1, \cdots f_k$ are generators of  $K[Y]$, and $u=\max_{1 \le i \le k}\{c_i \log|f_i|+C_i, \lambda\}$. We use $\FS(Y^{\an})$ to denote the set of Fubini-Study functions on $Y^{\an}$. 
    \item We say that a psh function $v \in \PSH(Y^{\an})$ is  of \emph{logarithmic growth} if there exists $u \in \FS(Y^{\an})$ such that $u \le v+ O(1)$ on $Y^{\an}$.  We use $\PSHlog(Y^{\an})$ to denote the set of psh functions of logarithmic growth on $Y^{\an}$.
    \item We say that a psh function of logarithmic growth $v$ is regularizable if there exists a decreasing function sequence $\{v_i\}_{i \ge 1} \subset \FS(Y^{\an})$ converging pointwise to $v$. We use $\PSHlogreg(Y^{\an})$ to denote the set of regularizable psh functions of logarithmic growth on $Y^{\an}$.
  \end{enumerate}
\end{definition}

\begin{lemma} \label{lem:FS_functions_are_comparable}
  Let $u \in \FS(Y^{\an})$.
  \begin{enumerate}
    \item For any $v \in \FS(Y^{\an})$, there exists $C >0$ such that $u \le Cv+O(1)$.
    \item Suppose that $S$ is another affine variety over $K$ and $\pi \colon Y \to S$ is a finite surjective morphism. Then, for any $w \in \FS(S^{\an})$, there exists $C'>0$ such that $u \le C' (w \circ \pi^{\an}) +O(1)$.
  \end{enumerate}
\end{lemma}
\begin{proof}
  (1) Suppose that $v \coloneqq \max\{\log|z_1|, \cdots, \log|z_k|, 0\}+1  \in \FS(Y^{\an})$, where $z_1, \cdots, z_k$ are generators of $K[Y]$. Let $f=\sum_{|\alpha| \le N} a_{\alpha} z^{\alpha} \in K[Y]$. Then, $\log|f| \le \sum_{|\alpha| \le N} \log|a_{\alpha}| ( \alpha \cdot \log|z| )  \le C_f v+O(1)$ for $C_f \gg 0$. This proves the first statement. 

  \smallskip

  (2) Let $f \in K[Y]$. By definition, $K[Y]$ is a finite algebra over $K[S]$. Then, there exists $m \in \N$ and $b_0, \cdots, b_{m-1} \in K[S]$ such that $f^m+b_{m-1}f^{m-1}+\cdots + b_0=0$. By \cite[Lemma 11 in Section 3.1]{Bos14}, we have $\log|f|(y) \le \max_{0 \le i \le m-1}\{\frac{1}{i} \log|b_i| \left(\pi^{\an}(y) \right) \}$ for any $y \in Y^{\an}$. Combining this with  statement (1) applied to $\FS(S^{\an})$ proves statement (2). 
\end{proof}

\begin{prop} \label{prop:potential_of_positive_form_is_Fubini-Study}
  \begin{enumerate}
    \item Let $X$ be a projective variety over $K$ and let $\theta \in \amp(X)$ such that $L(\theta)$ is an actual ample line bundle on $X$. Denote by $|\cdot|_{\theta}$ the Fubini-Study metric on $L(\theta)$ induced by $\theta$, and    denote $Y \coloneqq X \setminus Z(s)$. Then, $g \coloneqq -\log|s|_{\theta} \in \FS(Y^{\an})$.  
    \item Conversely, let $Y$ be an affine variety over $K$ and $g \in \FS(Y^{\an})$. Then, there exists a projective variety $X$ over $K$, an ample line bundle $L$ on $K$, $\theta \in \amp(X)$, $c \in \Q_{>0}$, $\lambda \in \Q$  and $s \in H^0(X,L)$ such that $L \simeq L(\theta)$,  $Y = X \setminus Z(s)$, and $g=-c\log|s|_{\theta}+\lambda$, where $|\cdot|_{\theta}$ is the Fubini-Study metric on $L$ induced by $\theta$. 
  \end{enumerate}
\end{prop}
\begin{proof}
  (1) Let $m \ge 1$ such that $mL$ is very ample on $X$. Then, by \cite[Lemma 7.17 and Corollary 7.18]{BE21}, there exists a base $\{s_1, \cdots, s_k\}$ of $H^0(X, mL)$ such that $|\cdot|_{\theta}=\frac{1}{m} \max_{1 \le i \le k} \{\log|s_i|+C_i\}$ with $C_i \in \Q$. Then, 
  \[
  g=-\log|s|_{\theta}=\max_{1 \le i \le k} \frac{1}{m} \{\log|s_i/s^{\otimes m}| +C_i, \ 0 \}
  \]
  with $ s_i/s^{\otimes m} \in K[Y]$. By assumption, $[s^{\otimes m} \colon s_1 \colon \cdots \colon s_k]$ induces a closed embedding $X \hookrightarrow \P^k$. Then, $\{s_i/s^{\otimes m}\}_{1 \le i \le k}$ induces a closed embedding $Y \hookrightarrow \BA^k$, and is therefore a set of generator of $K[Y]$. This implies that $g \in \FS(Y^{\an})$. 

  (2) Let $g=\max_{1 \le i \le k}\{c_i \log|f_i|+C_i, \lambda\}$ for $f_1, \cdots, f_k \in K[Y]$, $c_1, \cdots, c_k \in \Q_{>0}$ and $C_1, \cdots, C_k, \lambda \in  \Q$. After multiplying $g$ by a large enough positive integer and adding a constant, we may assume that $c_1, \cdots, c_k \in \Z_{>0}$ and $\lambda=0$. Let $M \coloneqq \max\{|C_1|, \cdots, |C_k|\} \in \Z_{>0}$. Then, it's not hard to see that for each $1 \le i \le k$, we have 
  \bens
  \log|f_i|-M \le \max_{1 \le i \le k}\{c_i \log|f_i|, \ 0\} -M \le \max_{1 \le i \le k}\{c_i \log|f_i|+C_i, 0 \}= g.
  \eens
  This implies that 
  \[
  g=\max_{1 \le i \le k} \{  \log|f_i^{c_i}|+C_i, \ \log|f_i|-M, 0 \}
  \]
  with $f_i^{c_i} \in K[Y]$. By assumption, $f_1, \cdots, f_k$ are generators of $K[Y]$. Therefore, 
  \[
  (f_1^{c_1},\ldots,f_k^{c_k},f_1,\ldots,f_k)\colon Y\hookrightarrow \mathbb{A}^{2k}
  \]
  is a closed embedding. Let \(x_1,\ldots,x_{2k}\) denote the corresponding affine coordinates on \(\mathbb{A}^{2k}\). We compactify \(\mathbb{A}^{2k}\) to \(\mathbb{P}^{2k}\) with homogeneous coordinates
  \(
  [s_0:s_1:\cdots:s_{2k}],
  \)
  so that, on the affine chart \(s_0\neq 0\), we have \(x_i=s_i/s_0\).

  Let $X$ be the Zariski closure of $Y$ in $\P^{2k}$ and let $L=\CO(1)|_X$. Viewing $s_0, s_1, \cdots, s_{2k}$ as a base of $H^0(X, L)$, we define 
  \[
  \phi \coloneqq \max_{1 \le i \le k} \{\log|s_0|,\ \log|s_i|+C_i, \ \log|s_{i+k}|-M\}
  \]
  as a Fubini-Study metric on $L$. Now, Proposition \ref{prop:equivalence_FS_metrics_relatively_ample_models} implies that there exists $\theta \in \amp(X)$ such that $L \simeq L(\theta)$ and $|\cdot|_{\theta}=|\cdot|_{\phi}$.  Then, we have $Y=X \setminus Z(s_0)$ and $g=-\log|s_0|_{\theta}$ on $Y^{\an}$. Hence, statement (2) is proved. 
\end{proof}

\begin{cor} \label{cor:GPSH_plus_potential_of_positive_forms_is_pshlogreg}
  Under the assumptions of Proposition \ref{prop:potential_of_positive_form_is_Fubini-Study} (1), for any $\varphi \in \GPSH_{\theta}(X^{\an})$, we have $g+\varphi \in \PSHlogreg(Y^{\an})$.
\end{cor}
\begin{proof}
  Let $\varphi \in \GPSH_{\theta}(X^{\an})$. Since $L(\theta)$ is ample, Proposition \ref{prop:regularization_omega_psh_by_omega_FS_functions} (2) implies that there exists a decreasing sequence $\{\varphi_i\}_{i \ge 1} \subset \FS_{\theta}(X^{\an})$ converging pointwise to $\varphi$. By definition, for each $i \ge 1$, $\theta_{\varphi_i} \in \amp(X)$ and $-\log|s|_{\theta_{\varphi_i}} =g+\varphi_i$. Then, applying Proposition \ref{prop:potential_of_positive_form_is_Fubini-Study} (1) to $\theta_{\varphi_i}$, we obtain that $\{g+\varphi_i\}_{i \ge 1} \subset \FS(Y^{\an})$ is a decreasing function sequence converging pointwise to $g$. This implies that $g \in \PSHlogreg(Y^{\an})$. 
\end{proof}

\begin{cor} \label{cor:GPSH_is_LPSH}
  Let $X$ be a projective variety over $K$ and let $\omega \in \nef(X)$. Let $U \subset X^{\an}$ be an open subset and $g \in \PL(U)$ be a potential of $\omega$ on $U$. Then, for any $\varphi \in \GPSH_{\omega}(X)$, we have $g+\varphi \in\PSH(U)$. In particular, we have $\GPSH_{\omega}(X^{\an}) \subset \LPSH_{\omega}(X^{\an})$.
\end{cor}
\begin{proof}
  Let $(\fX, \fL)$ be a model representing $\omega$. By \cite[Proposition 4.11]{GM19}, after passing to a model dominating $\fX$, we may assume that there exists a $\Q$-line bundle $\fL'$ relatively ample on $\fX$. Let $\theta \in \PLforms(X)$ be the class of $(\fX, \fL')$. Then, Proposition \ref{prop:equivalence_FS_metrics_relatively_ample_models} implies that $\omega + \epsilon \theta \in \amp(X)$ for any $\epsilon >0$.
  
  We may assume that $U$ is an affine open subset of $X^{\an}$ and $\theta$ admits  a potential $h$ on $U$. By  Proposition \ref{prop:potential_of_positive_form_is_Fubini-Study} (1), we have $h \in \FS(U)$. Thus, by definition, $\inf_{x \in U} h(x)=C>-\infty$ for some $C \in \R$. Then, for each $\epsilon >0$, $g+\epsilon (h-C)$ is a potential of $\omega + \epsilon \theta$ on $U$. 

  Let $\varphi \in\GPSHo(X^{\an})$. Then, Proposition \ref{prop:def_of_omega_psh_does_not_depend_on_theta} (1) together with Proposition \ref{prop:regularization_omega_psh_by_omega_FS_functions} (1) imply that there exists a decreasing sequence of numbers $\{\epsilon_i\}_{i \ge 1} \subset \Q_{>0}$, and a decreasing function sequence $\{\psi_i\}_{i \ge 1}$ converging pointwise to $\varphi$ such that $\psi_i \in \FS_{\omega+\epsilon_i \theta}(X^{\an})$ for each $i \ge 1$.  Note that $g+\psi_i+\epsilon (h-C)$ is a potential of $(\omega+\epsilon_i \theta)_{\psi_i} \in \amp(X)$, and Proposition \ref{prop:potential_of_positive_form_is_Fubini-Study} (1) implies that $g+\psi_i+\epsilon (h-C) \in \FS(U)$.   Since $h-C \ge 0$ on $U$, $\left\{g+\psi_i+\epsilon_i (h-C) \right\}_{i \ge 1}$ is a decreasing function sequence converging pointwise to $g+\varphi$. This implies that $g+\varphi \in \PSH(U)$. 
\end{proof}

\subsection{Tropical polytopes}

\begin{definition} \label{def:tropical_polytopes}
  A subset $P \subset Y^{\an}$ is called a tropical polytope if $P$ is  compact, and there exists $f_1, \cdots, f_k \in K[Y]$ and rational numbers $t_1, \cdots, t_k \in \Q$ such that 
  \[
  P = \bigcap_{1 \le i \le k} \{-\log|f_i| \ge t_i\} \subset Y^{\an}.
  \]
\end{definition}

\begin{example} \label{example:tropical_polytopes_in_affine_varieties}
  Let $z_1, \cdots, z_k$ be generators of the coordinate ring $ K[Y]$. For each $r \in \N$, let $E_r \coloneqq \bigcap_{1 \le i \le k} \{-\log|z_i| \ge -r \} \subset Y^{\an}$. Then, $E_r$ is compact  and hence a tropical polytope. Moreover, we have $E_r \subset \interior (E_{r+1})$, and $Y^{\an}=\bigcup_{r \in \N} \interior( E_r)$. We call $\{E_r\}_{r \ge 0}$ a sequence of tropical polytopes exhausting $Y^{\an}$.  To see that $E_r$ is compact, consider the closed embedding $Y \hookrightarrow \BA^k$ induced by $\{z_i\}_{1 \le i \le k}$, and let $X$ be the Zariski closure of $Y$ in $\P^k$. Then,  $E_r$ is a closed subset of $X^{\an}$ by the definition of the topology on $X^{\an}$. On the other hand,  $X^{\an}$ is a compact Hausdorff space. Therefore, $E_r$ is compact.     
\end{example}

\begin{lemma} \label{lem:tropical_polytopes_and_FS_functions}
  For any $u \in \FS(Y^{\an})$ and $t \in \Q$, $\{u \le t\}$ is a tropical polytope. Conversely for any tropical polytope $P \subset Y^{\an}$, there exists $u \in \FS(Y^{\an})$  such that $P=\{ u=0 \} \subset Y^{\an}$ and $u \ge 0$ on $Y^{\an}$.
\end{lemma}
\begin{proof}
  For the first half of the statement, suppose that $f_1, \cdots, f_k \in K[Y]$, $c_1, \cdots, c_k \in \Q_{>0}$ and $t_1, \cdots, t_k, \lambda \in  \Q$ such that $f_1, \cdots f_k$ are generators of  $K[Y]$, and $u=\max_{1 \le i \le k}\{c_i \log|f_i|+t_i, \lambda\}$. If $\lambda >t$, then $\{u \ge t\}=\emptyset$. If $\lambda \le t$, then  $\{u \le t\}=\bigcap_{1 \le i \le k} \{-\log|f_i| \ge \frac{1}{C}(t_i-t)\}$, and     Example \ref{example:tropical_polytopes_in_affine_varieties} implies that $\{u \le t\}$ is a tropical polytope. For the second half, suppose that $P = \bigcap_{1 \le i \le l} \{-\log|f_i| \ge t_i\} \subset Y^{\an}$ with $f_i \in K[Y]$ and $t_i \in \Q$. Let $z_1, \cdots, z_k $ be generators of $K[Y]$ and let $E_r$ be the sequence of tropical polytopes constructed in Example \ref{example:tropical_polytopes_in_affine_varieties}. Since $P$ is compact, there exists $r \in \N$ such that $P \subset E_r$. Thus, we have $P=\left( \bigcap_{1 \le i \le l} \{-\log|f_i| \ge t_i\} \right) \cap \left(\bigcap_{1 \le i \le k} \{-\log|z_i| \ge -r \} \right) \subset Y^{\an}$. Then, $u \coloneqq \max_{i,j}\{\log|f_i|+t_i, \log|z_j|-r, 0 \} \in \FS(Y^{\an})$ with $P=\{u=0\} \subset Y^{\an}$.
\end{proof}

\subsection{Psh PL functions}
We start from the following result.
\begin{thm} \label{thm:GPSH_PL_equals_to_LPSH_PL}
  Let $X$ be a projective variety over $K$, and let $\omega \in \nef(X)$. Then, 
  \[
  \GPSHo \cap \PL(X^{\an})=\LPSHo \cap \PL(X^{\an}).
  \] 
\end{thm}
\begin{proof}
  That $\GPSHo \cap \PL(X^{\an}) \subset \LPSHo \cap \PL(X^{\an})$ is implied by Corollary \ref{cor:GPSH_is_LPSH}. The converse direction is a corollary of \cite[Theorem 18.8.4 and Corollary 18.8.8]{CLD}. More precisely, any tropically convex PL function can be locally approximated uniformly by a sequence of tropically convex smooth functions. Thus,  for any open subset $U \subset X^{\an}$ and $f \in \PSH(U)$, $f$ is psh after \cite[Definition 8.1.1]{CLD}. Then, $\dm \dmb f$ is a semipositive current on $U$ by \cite[Definition 7.2.1, Proposition 8.3.1]{CLD}. Therefore,  for any $\varphi \in \LPSHo \cap \PL(X)$, $\omega+\dm \dmb \varphi$ is a closed semipositive $(1,1)$-current on $X^{\an}$. Then, \cite[Corollary 18.8.8]{CLD} implies that any model representing $\omega_{\varphi}$ is a nef model, and hence $\varphi \in \GPSHo \cap \PL(X)$. 
\end{proof}

Now, we give applications of Theorem \ref{thm:GPSH_PL_equals_to_LPSH_PL}.

\begin{lemma} \label{lem:psh_PL_functions_are_pshlogreg}
  Let $Y$ be an affine variety over $K$.  
  \begin{enumerate}
    \item Let $u \in \PSH \cap \PL(Y^{\an})$. Suppose that there exists a compact subset $E \subset Y^{\an}$ and $g \in \FS(Y^{\an})$ such that $u=g$ on $Y^{\an} \setminus E$. Then, any $\epsilon>0$, there exists $v \in \FS(Y^{\an})$ such that $\left\| u-v \right\|_{C^0(Y^{an})} < \epsilon$. 
    \item $\PSH \cap \PL(Y^{\an}) \subset \PSHlogreg(Y^{\an})$.
    \item Let $\{w_i\}_{i \ge 1} \subset \PSHlog (Y^{\an})$ be a decreasing function sequence converging pointwise to $w$. Let $E_1 \subset \cdots E_r \subset \cdots \subset Y^{\an}$ be a sequence of tropical polytopes exhausting $Y^{\an}$. Suppose that $w$ is generically finite, and  $w_i|_{\interior(E_r)} \in \PL(\interior(E_r))$ for each $i, r \ge 1$. Then, $w \in \PSHlogreg(Y^{\an})$. 
  \end{enumerate}
\end{lemma}

\begin{proof}
  For statement (1), by Proposition \ref{prop:potential_of_positive_form_is_Fubini-Study} (2), there exists a projective compactification $X$ of $Y$, an ample line bundle $L$ on $X$ with $s \in H^0(X,L)$ satisfying $Y=X \setminus Z(s)$, $\theta \in \amp(X)$ with $L \simeq L(\theta)$, and $c \in \Q_{>0}$, $\lambda \in \Q$ such that $g=-c\log|s|_{\theta}+\lambda$  on $Y^{\an}$. Define $\varphi \colon X^{\an} \to \R \cup \{-\infty\}$ by $\varphi=u-g$  on $Y^{\an}$, and $\varphi=0$ on $X^{\an} \setminus E$. By definition, $\varphi \in \LPSH_{c\theta} \cap \PL(X^{\an})$. Then, Theorem \ref{thm:GPSH_PL_equals_to_LPSH_PL} implies that $\varphi \in \GPSH_{c \theta} \cap \PL(X^{\an})$, and Proposition \ref{prop:regularization_omega_psh_by_omega_FS_functions} (1) implies that there exists $\psi \in \FS_{c \theta}(X^{\an})$ such that $\left\|\varphi-\psi\right\|_{C^0(X^{\an})}< \epsilon$. By assumption, $c \theta_{\psi}$ induces a Fubini-Study metric on $L$, and $g+\psi=-c\log|s|_{\theta_{\psi}}+\lambda$. Then, Proposition \ref{prop:potential_of_positive_form_is_Fubini-Study} (1) implies that $v \coloneqq g+\psi \in \FS(Y^{\an})$, and we have $\left\|v-u\right\|_{C^0(Y^{\an})}=\left\|\psi-\varphi\right\|_{C^0(Y^{\an})}< \epsilon$. 

  For statement (2), take $w \in \PSH \cap \PL(Y^{\an})$. By definition, there exists a tropicalization $\Trop \colon Y^{\an} \to \BT^k$ and $f \in \PL(\BT^k)$ such that $w =f \circ \Trop$. Then, Lemma \ref{lem:property_tropical_conv_functions} (4) implies that there exists $f' \in \Conv \cap \PL(\BT^k)$ such that $f \le f'$ on $\BT^k$. Note that $w \le f' \circ \Trop \in \PSHlog(Y^{\an})$. Therefore, $w \in \PSHlog \cap \PL(Y^{\an})$, and $w \in \PSHlogreg(Y^{\an})$ follows from statement (3).

  For statement (3), let $h \in \FS(Y^{\an})$ such that $w_i \le h$ for each $i \ge 1$.   Take $c>1$, and take a decreasing sequence of rational numbers $\{C_i\}_{i \ge 1} \subset \Q$ such that $C_i \to -\infty$. For each $i \ge 1$, set $w'_i \coloneqq \max\{w_i, ch+C_i\}$, and set $P_i \coloneqq \{(c-1)h \le -C_i\}$. By calculations, we have $w'_i=ch-C_i$ on $Y^{\an} \setminus P_i$.  On the other hand, by Lemma \ref{lem:tropical_polytopes_and_FS_functions}, $P_i$ is a tropical polytope and hence compact. Thus, there exists $r \ge 1$ such that $P_i \subset \interior(E_r)$. Then, by assumption, $w'_i|_{\interior(E_r)}=\max\{w_i|_{\interior(E_r)}, ch+C_i\} \in \PSH \cap \PL(\interior(E_r))$. Therefore, we have $w'_i \in \PSH \cap \PL(Y^{\an})$, and  statement (1) implies that there exists $w''_i \in \FS(Y^{\an})$ satisfying $\left\|w''_i -w'_i\right\|_{C^0(Y^{\an})} \le 1/2^{i+3}$. Note that  $\{w'_i\}_{i \ge 1}$ is a decreasing function sequence converging pointwise to $w$.  Then, $\{w''_i+1/2^i\}_{i \ge 1} \subset \FS(Y^{\an})$ is a decreasing function sequence converging pointwise to $w$. This implies that $w \in \PSHlogreg(Y^{\an})$ and proves statement (3). 
\end{proof}

\begin{lemma} \label{lem:psh_regularization_with_growth_control} 
  Assume that  $K$ is discretely valued or trivially valued. Let $Y$ be an affine variety over $K$.
  \begin{enumerate}
    \item Under the hypotheses of Proposition \ref{prop:potential_of_positive_form_is_Fubini-Study} (1), suppose further that $u \in \PSHlogreg(Y^{\an})$ such that $u \le g+O(1)$. Let $c>1$. Define the function $\varphi \coloneqq X^{\an} \to \R \cup \{-\infty\}$ by 
    \[
    \varphi=u-cg \text{ on } Y^{\an}, \qquad \text{and} \qquad \varphi=-\infty \text{ on } X^{\an} \setminus Y^{\an}. 
    \]
    Then, $\varphi \in \GPSH_{c\theta}(X^{\an})$. 
    \item Let $u \in \PSHlogreg(Y^{\an})$, and let $g \in \FS(Y^{\an})$ such that $u \le g+O(1)$.  Then, for each $c>1$, there exists a decreasing function sequence $\{u_i\}_{i \ge 1} \subset \FS(Y^{\an})$ converging pointwise to $u$ on $Y^{\an}$ such that $|u_i - cg|=O(1)$ on $Y^{\an}$ for each $i \ge 1$.
    \item Let $\{u_i\}_{i \ge 1} \subset \PSHlogreg(Y^{\an})$ be a decreasing function sequence converging pointwise to $u$. Suppose that $u$ is generically finite. Then, $u \in \PSHlogreg(Y^{\an})$. 
  \end{enumerate}
\end{lemma}

\begin{proof}
  (1). For each $C \in \Q$, set $\varphi_C \coloneqq \max\{\varphi, C\}$. By Theorem \ref{thm:decreasing_limit_of_omega-psh_functions_is_omega-psh}, it is enough to show that for any $C \in \Q$, we have $\varphi_C \in \GPSH_{c \theta}(X^{\an})$. 

  Fix $C \in \Q$. After adding a constant to $g$, we may assume that $u \le g$. Set $E \coloneqq \{(c-1)g \le -C+1\} \subset Y^{\an}$.  Note that $\{u-cg \ge C\} \subset \{g+1-cg > C\}=\interior(E)$. Therefore, we have $\varphi_C=\max\{u-cg, C\}$ on $\interior(E)$, and $\varphi_C=C$ on  $X^{\an} \setminus \interior(E)$.
  
  By Lemma \ref{lem:tropical_polytopes_and_FS_functions}, $E$ is a tropical polytope and hence compact. Since $X^{\an}$ is compact Hausdorff, there exists an open neighborhood $U$ of $E$ in $X^{\an}$ such that $\overline{U}$ is compact and $\overline{U} \subset Y^{\an}$. By definition,  there exists a decreasing function sequence $\{u_i\}_{i \ge 1} \subset \FS(Y^{\an})$ converging pointwise to $u$. Since $\overline{U} \setminus \interior(E)$ is compact, and $cg+C$ is continuous, Dini's theorem implies that there exists $j \ge 1$ such that $u_i < cg+C$ on $\overline{U} \setminus \interior(E)$ for each $i  \ge j$. In particular, we have $\overline{W} \subset E \subset U$, where  $W \coloneqq \{y \in U \,|\, u_i(y) > cg(y)+C\}$. 

  For each $i \ge j$, define the function $\varphi_{i,C} \colon X^{\an} \to \R \cup \{-\infty\}$ by 
  \[
  \varphi_{i,C}=\max\{u_i-cg, C\} \text{ on } U, \quad \text{and} \quad \varphi_{i,C}=C \text{ on } X^{\an} \setminus U.
  \]
  By Remark \ref{rmk:extension_of_psh_functions_by_sup}, $\varphi_{i,C} \in \LPSH_{c\theta} \cap \PL(X^{\an})$. Then, Theorem \ref{thm:GPSH_PL_equals_to_LPSH_PL} implies that $\varphi_{i,C} \in \GPSH_{c\theta} \cap \PL(X^{\an})$. Note that $\{\varphi_{i,C}\}_{i \ge j}$ is a decreasing function sequence converging pointwise to $\varphi_C$. This implies that $\varphi_C \in \GPSH_{c\theta}(X^{\an})$ and proves statement (1). 

  (2).  Let $X$ be a projective compactification of $Y$ and let $\theta \in \amp(X)$ be a closed positive $(1,1)$-form associated to $g \in \FS(Y^{\an})$ given by Proposition \ref{prop:potential_of_positive_form_is_Fubini-Study} (2). Define the function $\varphi \colon X^{\an} \to \R \cup \{-\infty \}$ by $\varphi= u-cg$ on $Y^{\an}$ and $\varphi=-\infty$ on $X^{\an} \setminus Y^{\an}$. Then, statement (1) implies that $\varphi \in \GPSH_{c\theta}(X^{\an})$. Now, Proposition \ref{prop:regularization_omega_psh_by_omega_FS_functions} (2) implies that there exists a decreasing function sequence $\{ \varphi_i \}_{i \ge 1} \subset \FS_{c\theta}(X^{\an})$ converging pointwise to $\varphi$. Then, by Proposition \ref{prop:potential_of_positive_form_is_Fubini-Study} (1), $u_i \coloneqq \varphi_i+cg \in \FS(Y^{\an})$. Moreover, $|u_i-cg|=|\varphi_i|=O(1)$ on $Y^{\an}$ since $\varphi_i \in C^0(X^{\an})$, and $\{u_i\}_{i \ge 1}$ is a decreasing function sequence converging pointwise to $u$ on $Y^{\an}$. Hence, statement (2) is proved. 

  (3). Similar to the proof of statement (2), we make use of Proposition \ref{prop:potential_of_positive_form_is_Fubini-Study} (2) and statement (1)  and Corollary \ref{cor:GPSH_plus_potential_of_positive_forms_is_pshlogreg} to reduce this statement to the global setting, which follows from Theorem \ref{thm:decreasing_limit_of_omega-psh_functions_is_omega-psh}. The details are omitted.  
\end{proof}

\begin{lemma} \label{lem:approximable_on_exhausting_tropical_polytopes_implies_pshlogreg}
  Assume that  $K$ is discretely valued or trivially valued. Let $Y$ be an affine variety over $K$. Let $u \in \PSHlog(Y^{\an})$. Let $E_1 \subset \cdots \subset E_r \subset \cdots Y^{\an}$ be a sequence of tropical polytopes exhausting $Y^{\an}$. Suppose that for each $r \ge 1$, there exists a decreasing function sequence $\{u_{r, i}\}_{i \ge 1} \subset \TConv \cap \PL(\interior(E_r))$ converging to $u|_{\interior(E_r)}$ pointwise. Then, $u \in \PSHlogreg(Y^{\an})$. 
\end{lemma}
\begin{proof}
  The proof proceeds along the same lines as that of Lemma \ref{lem:psh_regularization_with_growth_control} (1). Let $g \in \FS(Y^{\an})$ such that $u \le g$ on $Y^{\an}$.  Let $X$ be a projective compactification of $Y$ and let $\theta \in \amp(X)$ be a closed positive $(1,1)$-form associated to $g \in \FS(Y^{\an})$ given by Proposition \ref{prop:potential_of_positive_form_is_Fubini-Study} (2). Fix $c>1$. Define the function $\varphi \colon X^{\an} \to \R \cup \{-\infty \}$ by $\varphi= u-cg$ on $Y^{\an}$ and $\varphi=-\infty$ on $X^{\an} \setminus Y^{\an}$. Fix $C \in \Q$ and set $\varphi_C \coloneqq \max\{\varphi,C\}$. Similar to the proof of Lemma \ref{lem:psh_regularization_with_growth_control} (1), we obtain that $\{u-cg \ge C\} \subset E$ for a compact subset $E \subset Y^{\an}$. Since $Y^{\an}=\bigcup_{r \ge 1} \interior(E_r)$, there exists $r \ge 1$ such that $E \subset \interior(E_r)$. Define the function $\varphi_{i,C} \colon X^{\an} \to \R \cup \{-\infty\}$ by
  \[
  \varphi_{i,C}=\max\{u_{r,i}-cg, C\} \text{ on } \interior(E_r), \quad \text{and} \quad \varphi_{i,C}=C \text{ on } X^{\an} \setminus \interior(E_r).
  \]
  Then, a similar arguments to the one in the proof of Lemma \ref{lem:psh_regularization_with_growth_control} (1) shows that $\{\varphi_{i,C}\}_{i \ge j} \subset \GPSH_{c\theta}(X^{\an})$ for some $j \ge 1$, and it is a decreasing function sequence converging pointwise to $\varphi_C$. This implies that $\varphi_C \in \GPSH_{c\theta}(X^{\an})$, and Theorem \ref{thm:decreasing_limit_of_omega-psh_functions_is_omega-psh} implies further that $\varphi \in \GPSH_{c\theta}(X^{\an})$. Therefore, by Corollary \ref{cor:GPSH_plus_potential_of_positive_forms_is_pshlogreg}, $u=\varphi+cg \in \PSHlogreg(Y^{\an})$.
\end{proof}

\begin{example} \label{example:psh_function_defining_closed_subvariety}
  Suppose that $K$ is discretely valued or trivially valued. Let $Y$ be an  affine variety over $K$, and let $Z =\bigcap_{1 \le i \in k} \{f_i=0\} \subset Y$ be a closed subvariety with $\{f_i\}_{1 \le i \le k} \subset K[Y]$. Denote $u \coloneqq \log(\sum_{1 \le i \le k}|f_i|)$. We claim that $u \in \PSHlogreg(Y^{\an})$ with $\{u=-\infty\}=Z^{\an}$.  
  
  The latter part of the claim is clear. To see the first part, consider the the log-sum-exp function $g \coloneqq \log(\sum_{1 \le i \le k} \exp(-x_i))$ on $\BT^k$. Then, $u=g \circ \Trop$, where $\Trop \colon Y^{\an} \to \BT^k$ is tropicalization map defined by $\{-\log|f_i|\}_{1 \le i \le k}$. It is a well-known fact that $g|_{\R^k} \in \Conv(\R^k)$ (see for example \cite[page 74]{BV04}). For each $M \ge 1$, denote $g_M \coloneqq \max\{g, -M\}$. Then, $g_M \in C^0(\BT^k)$ and $g_M$ is convex on $\R^k$. Thus, Lemma \ref{lem:property_tropical_conv_functions} (1) implies that $g_M \in \Conv(\BT^k)$.

  Let $E_1 \subset \cdots \subset E_r \subset \cdots \subset Y^{\an}$ be a sequence of tropical polytopes exhausting $Y^{\an}$. For each $r \ge 1$, $\Trop(E_r)$ is a compact subset $\BT^k$. Therefore, by Lemma \ref{lem:property_tropical_conv_functions} (3),  there exists a compact open neighborhood $W$ of $\Trop \left( E_r \right)$ in $\BT^k$ and a decreasing function sequence $\{g_{r,i}\}_{i \ge 1} \subset \Conv \cap \PL(W)$ converging uniformly to $g_{M}$ on $W$. Then, Lemma \ref{lem:approximable_on_exhausting_tropical_polytopes_implies_pshlogreg} implies that $(g_M \circ \Trop) \in \PSHlogreg(Y^{\an})$, and hence Lemma \ref{lem:psh_regularization_with_growth_control} (3) implies that $u \in \PSHlogreg(Y^{\an})$.
\end{example}

\subsection{}
The following proposition allows us to pass  problems about $\PSH(U)$ to problems about $\LPSHo(X)$. It's an analogue of Proposition 5.20 \cite{APW2}.
\begin{prop} \label{prop:passing_from_local_to_global}
  Let $X $ be a projective variety over $K$ and $U \subset X^{\an}$ be an open subset. Let $\omega \in \nef(X)$ such that $c_1(L(\omega)) \in \Amp(X)$. Then, 
  \begin{enumerate}
    \item For any $x \in U$ and any $M \in \BR_{>0}$, there exists an open neighborhood $U_x \subset U$ of $x$ (depending only on the choice of $x$), a constant $C \in \Q_{>0}$, and a potential $g \in \PL \cap \LConv (U_x)$ of $\omega$ (depending on $x$ and $M$), such that the following property is satisfied. For any $u \in \PSH(U)$ with $\| u\|_{L^{\infty}(U)} \le M$, there exists $\varphi \in \LPSHob(X)$, such that $\varphi+g=Cu$ on $U_x$.
    \item Suppose further that there exists a decreasing function sequence $\{u_i\}_{i \ge 1} \subset \TConv \cap \PL(U)$ converging pointwise to $u$ on $U$. Then, $\varphi$ may be taken inside $\GPSHob(X)$ in statement (1). 
  \end{enumerate}
\end{prop}

\begin{proof} 
  Fix $x \in U$. By definition, there exists $f_1, \cdots, f_k \in K(X)$ and $C_1, \cdots, C_k \in \Q$ such that $x \in \bigcap_{1 \le i \le k}\{\log|f_i| < C_i\} \Subset U$. Then, there exists $m \ge 1$ and a base $\{s_0, s_1, \cdots, s_r\}$ of $H^0(X, mL(\omega))$ such that  $mL(\omega)$ is very ample on $X$, $x \in X \setminus Z(s) \eqqcolon Y$, $k \le r$, and $f_i =s_i/s_0$ for each $1 \le i \le k$.    For each $k+1 \le j \le r$, we have $f_j \coloneqq s_j/s_0 \in K[Y]$, so there exists $C_j \in \Q$ such that $\log|f_j|(x)< C_j$.

  Set $\phi \coloneqq \max_{1 \le i \le r}\{\log|s_0|, \log|s_i| -C_j\}$. As $\phi$ is a Fubini-Study metric on $L(\omega)$, there exists $\theta \in \amp(X)$ such that $L(\theta) \simeq L(\omega)$ and $|\cdot|_{\theta}=|\cdot|_{\phi}$. Moreover, $v \coloneqq -\log|s_0|_{\theta} = \max_{1 \le i \le r}\{\log|f_i| -C_j\} \in \FS(Y^{\an})$. By constructions, we have $v(x)<0$. Fix $\epsilon \in \Q_{>0}$ such that $v(x)<-\epsilon$. Set $P \coloneqq \{y \in Y^{\an} \,|\, v(y) \le -\epsilon\}$ and $Q \coloneqq \{y \in Y^{\an} \,|\, v(y) \le 0\}$. Then, we have 
  \[
  x \in \interior(P) \subset Q \subset \bigcap_{1 \le i \le k} \{\log|f_i| < C_i\} \Subset U.
  \] 

  Fix $M \in \R_{>0}$ and let $u \in \PSH(U)$ such that $|u| \le M$ on $Q$. Set $A \coloneqq \frac{2M+1}{\epsilon}$, and  $h \coloneqq Av+M+1$. Then, $h \le -M$ on $P$, and 
  \[
  \{y \in U \,|\, h(y) \le u(y)\} \subset \{y \in U \,|\, h(y) \le M\}= \{y \in Y^{\an} \,|\, v(y)\le -\frac{1}{A}\} \Subset \interior(Q). 
  \]
  Set $Q' \coloneqq \{y \in Y^{\an} \,|\, v(y)\le -\frac{1}{A}\}$, and define the function $\psi \colon X^{\an}  \to \R \cup \{-\infty\}$ by 
  \[
  \psi \coloneqq \max\{u-h, 0\} \text{ on } Q \qquad \text{and} \qquad \psi \coloneqq 0 \text{ on } X^{\an} \setminus Q. 
  \]
  Then, $\psi+h \in \PSH(\interior(Q))$, and $\psi=0$ on $X^{\an} \setminus Q'$. Note that $h$ is a potential of $A\omega$ on $Y^{\an}$, and $X^{\an}=\interior(Q) \cup (X^{\an} \setminus Q')$.  This implies that $\psi \in \LPSH_{A\omega} \cap L^{\infty}(X^{\an})$. On the other hand, we have $\psi=u-g$ on  $\interior(P)$. Set $U_x \coloneqq \interior(P)$; $C \coloneqq \frac{1}{A} \in \Q_{>0}$; $g \coloneqq Ch=v+C(M+1)$ and $\varphi \coloneqq C\psi$. Then, $g$ is a potential of $\omega$ on $U$; $\varphi \in \LPSHob(X^{\an})$; and $\varphi+g=Cu$ on $U_x$. Note that $U_x$ only depends on the choice of $x$, and $C$, $g$ only depend on the choice of $x$ and $M$. This proves statement (1). 

  Now we prove statement (2).  Define the function $\psi_i \coloneqq X^{\an} \to \R \cup \{-\infty\}$ by $\psi_i \coloneqq \max\{u_i-h, 0\}$ on $Q$ and $\psi_i \coloneqq 0$ on $X^{\an} \setminus Q$. Similar to the arguments in the proof of Lemma \ref{lem:psh_regularization_with_growth_control} (1),  there exists $j \ge 1$ such that for each $i \ge j$,  $u_i \le h$ on a neighborhood of $\partial Q$, and $\psi_i \in \GPSH_{A\omega} \cap \PL(X^{\an})$. Since $\{\psi_i\}_{i \ge j}$ is a decreasing sequence converging pointwise to $\psi$ on $X^{\an}$, we deduce that $\psi \in \GPSH_{A\omega} \cap L^{\infty}(X^{\an})$, and hence $\varphi \in \GPSH_{\omega} \cap L^{\infty}(X^{\an})$. This proves statement (2).
\end{proof}

Now we give an application of Proposition \ref{prop:passing_from_local_to_global}. 
\begin{prop} \label{prop:local_comparison_divisorial_points_to_full_spaces}
  Assume that $K$ is discretely valued or trivially valued. Let $X$ be a projective variety over $K$ and let $U \subset X^{\an}$ be an open subset. 
  \begin{enumerate}
    \item Let $u \in \PSH(U)$, and let $v \colon U \to \BR \cup \{\pm \infty\}$ be a usc function. Suppose that $u \le v$ on $U \cap X^{\div}$. Then, $u \le v$ on $U$.  In particular, for any $y \in U$, we have
    \[  
    u(y)=\limsup_{x \to y, \, x \in U \cap X^{\div}} u(x).
    \]
    \item Let $\varphi \in \LPSHo(X)$, and let $\psi \colon X^{\an} \to \BR \cup \{ \pm \infty\}$ be a usc function. Suppose that $\varphi \le \psi$ on $X^{\div}$. Then, $\varphi \le \psi$ on $X^{an}$.  In particular, for any $y \in X^{\an}$, we have
    \[  
    \varphi(y)=\limsup_{x \to y, \, x \in X^{\div}} \varphi(x).
    \]
  \end{enumerate}
\end{prop}
\begin{proof}
  Since the statements are local, it is enough to verify statement (1) on a neighborhood $U_x \subset U$ for each $x \in U$. In particular, we may assume that there exists a decreasing function sequence $\{u_i\}_{i \ge 1} \subset \TConv \cap \PL(U)$ converging pointwise to $u$ on $U$. Note that for any $M \in \R$, $u_{M} \coloneqq \max \{u, M\} \in \PSH(U)$, and $u_{M} \le v_{M} \coloneqq \max \{v, M\}$ on $U \cap X^{\div}$. It suffices to show that $u_M \le v_M$ for each $M \in \R$. Therefore, we may henceforth assume that $u \in \PSHb(U)$.  
  
  Fix $x \in U$ and $\omega \in \nef(X)$ with $c_1(L(\omega)) \in \Amp(X)$. Applying Proposition \ref{prop:passing_from_local_to_global} (2), we obtain an open neighborhood $U_x \subset U$ of $x$, a  constant $C>0$, a potential $g \in \PL \cap \LConv (U_x)$ of $\omega$, and $\varphi \in \GPSHo(X)$, such that $\varphi+g=Cu$ on $U_x$. 
  On the other hand, we define a usc function $\psi$ on $X^{\an}$ by 
  \[
  \psi(x)= \begin{cases*}
    (Cv-g)(x), \quad x \in U_x \\
    +\infty, \quad x \notin U_x
  \end{cases*}
  \]
  Then, by assumption, $\varphi \le \psi$ on $X^{\div}$. Therefore, by Theorem \ref{thm:global_comparison_divisorial_points_to_full_spaces}, we have $\varphi \le \psi$ on $X^{\an}$, which implies further that $u \le v$ on $U_x$. 
\end{proof}

\subsection{An extension theorem}

\begin{prop} \label{prop:extending_pshlogreg_by_adding_small_functions} 
  Let $Y$ be an affine variety over $K$. Let $f \in K[Y]$. Let $Z \coloneqq \{f=0\} \subset Y$ be the hypersurface in $Y$ define by $f$, and let $Y_{f}=Y \setminus Z=\Spec K[Y]_f$ be the open affine subvariety of $Y$. Let $g \in \FS(Y^{\an})$ and let $v \in \PSHlogreg((Y_f)^{\an})$  such that $v \le g+O(1)$ on $(Y_f)^{\an}$. For any $c \in \Q_{>0}$, define the function $\widetilde{v}_c \colon Y^{\an} \to \R \cup \{-\infty\}$ by 
  \[
  \widetilde{v}_c=v+c\log|f| \text{ on } (Y_f)^{\an}, \qquad \text{and} \qquad \widetilde{v}_c=-\infty \text{ on } Z^{\an}. 
  \]
  Then, $\widetilde{v}_c \in \PSHlogreg(Y^{\an})$.
\end{prop}
\begin{proof}
  Fix $c \in \Q_{>0}$. Denote $\epsilon \coloneqq c/3$, and denote $g_{\epsilon} \coloneqq \max\{g, -\epsilon \log|f|\}=\max\{g, \epsilon \log|f^{-1}|\}$. Then, $g_{\epsilon} \in \FS((Y_f)^{\an})$. After adding a constant to $g$, we may assume that $v \le g \le g_{\epsilon}$ on $(Y_f)^{\an}$. Then, by Lemma \ref{lem:psh_regularization_with_growth_control} (2), there exists a decreasing function sequence $\{v_i\}_{i \ge 1} \subset \FS((Y_f)^{\an})$ converging pointwise to $v$ on $(Y_f)^{\an}$ such that $v_i \le 2u_{\epsilon}+O(1)$ for each $i \ge 1$. Let $\{C_i\}_{i \ge 1} \subset \Q$ be a decreasing sequence of rational numbers such that $C_i \to -\infty$. For each $i \ge 1$, define the function $\widetilde{v}_{c,i} \colon Y^{\an} \to \R$ by
  \[
  \widetilde{v}_{c,i}=\max\{v_i+c\log|f|, C_i \} \text{ on } (Y_f)^{\an}, \qquad \text{and} \qquad \widetilde{v}_{c,i}=C_i \text{ on } Z^{\an}. 
  \]
  Note that $\left(\widetilde{v}_{c,i} \right)|_{(Y_f)^{\an}} \in \PSH \cap \PL((Y_f)^{\an})$. On the other hand, we have
  \[
  v_i + c\log|f| \le 2u_{\epsilon}+c\log|f|  +O(1) =\max \left\{ 2g+c\log|f|, \frac{1}{3}c \log|f| \right\}+O(1) \quad \text{on } (Y_f)^{\an}. 
  \]
  Denote  $h \coloneqq \max\{2g+c\log|f|, \frac{1}{3}c \log|f|\}$. Since $h$ is a usc function on $Y^{\an}$ which equals to $-\infty$ on $Z^{\an}$, $U_i \coloneqq \{y \in Y^{\an} \,|\, h(y) < C_i\}$ is an open neighborhood of $Z^{\an}$. This implies that $\widetilde{v}_{c,i}=C_i$ on $U_i \cap (Y_f)^{\an}$, and thereby $\widetilde{v}_{c,i}=C_i$ on $U_i$. Thus, we obtain that $\widetilde{v}_{c,i} \in \PSH \cap \PL(Y^{\an})$ for each $i \ge 1$. Then, $\{\widetilde{v}_{c,i}\}_{i \ge 1} \subset \PSH \cap \PL(Y^{\an})$ is a decreasing function converging pointwise to $\widetilde{v}_c$ on $Y^{\an}$, and Lemma \ref{lem:psh_PL_functions_are_pshlogreg} (3) implies that $\widetilde{v}_c \in \PSHlogreg(Y^{\an})$. 
\end{proof}

\begin{thm} \label{thm:pshlogreg_extension}
  Let $K$ be a complete discretely valued or trivially valued NA field of equicharacteristic $0$. Let  $Y$ be a smooth affine variety over $K$, and let $U$ be a Zariski dense affine open subvariety of $Y$. Let $u \in \PSHlogreg(U^{\an})$ and $g \in \FS(Y^{\an})$ such that $u \le g+O(1)$ on $U^{\an}$. Then, $u$ can be extended uniquely to $\widetilde{u} \in \PSHlogreg(Y^{\an})$.
\end{thm}
\begin{proof}
  Suppose that $\widetilde{u} \in \PSHlogreg(Y^{\an})$ is an extension of $u$. Since $(Y \setminus U)^{\an} \cap Y^{\div}=\emptyset$,  Proposition \ref{prop:local_comparison_divisorial_points_to_full_spaces} (1)implies that $\widetilde{u}(y)=\limsup_{x \to y,\, x \in (Y \setminus U)^{\an}} u(x)$. Then, the uniqueness follows.

  It remains to show the existence. We may assume that $U=Y_f$ for some $f \in K[Y]$. Let $\{c_i\}_{i \ge 1} \subset \Q_{>0}$ be a decreasing sequence of positive rational numbers. For each $i \ge 1$, define the function $\widetilde{u}_i \colon Y^{\an} \to \R \cup \{-\infty\}$ by $\widetilde{u}_i = u+c_i \log|f|$ on $U^{\an}$ and $u_i=-\infty$ on $(Y \setminus U)^{\an}$. Then, Proposition \ref{prop:extending_pshlogreg_by_adding_small_functions} implies that $\widetilde{u}_i 
  \in \PSHlogreg(Y^{\an})$ for each $i \ge 1$. Note that $\widetilde{u}_i \le g+\max\{c_1 \log|f|, 0\}$ for each $i \ge 1$, and Lemma \ref{lem:psh_PL_functions_are_pshlogreg} (2) implies that there exists $h \in \FS(Y^{\an})$ with $g+\max\{c_1 \log|f|, 0\} \le h$ on $Y^{\an}$. Thus, $\widetilde{u}_i \le h$ on $Y^{\an}$ for each $i \ge 1$.

  By Hironaka's theorem and Lemma \ref{lem:blowup_pi_ample_divisor}, there exists a smooth projective variety $X$ over $K$ such that $Y$ is an affine open variety of $X$, and there exists an effective ample divisor $D$ on $X$ with $\mathrm{Supp}(D)=X \setminus Y$. Take $r \in \Z_{>0}$ large enough such that $L \coloneqq \CO(rD)$ is a very ample line bundle on $X$ with section $s_{rD} \in H^0(X, L)$. Let $\theta \in \amp(X)$ such that $L = L(\theta)$, and let $|\cdot|_{\theta}$ be a Fubini-Study metric on $L$. Then, by Proposition \ref{prop:potential_of_positive_form_is_Fubini-Study} (1), $g \coloneqq -\log|s_{rD}|_{\theta}$ is  a Fubini-Study function on $Y^{\an}$.  By Lemma \ref{lem:FS_functions_are_comparable}, there exists $A>0$ such that $\widetilde{u}_i \le Ag+O(1)$ for each $i \ge 1$.  Define the function $\varphi \colon X^{\an} \to \R \cup \{-\infty\}$ by $\varphi_i=\widetilde{u}_i-2Ag$ on $Y^{\an}$ and $\varphi_i=-\infty$ on $X^{\an} \setminus Y^{\an}$. Then, Lemma \ref{lem:psh_regularization_with_growth_control} (1) implies that $\varphi_i \in \GPSH_{2A \theta}(X)$. 
  
  Let $\varphi$ be the pointwise limit of $\{\varphi_i\}_{i \ge 1}$, and $\varphi^{\star}$ be its usc envelope. Then, $\varphi=u-2Ag$ on $U^{\an}$ and $\varphi=-\infty$ on $(X \setminus U)^{\an}$. Since $u-2Ag$ is usc on $U^{\an}$, we have $\varphi^{\star}=\varphi=u-2Ag$ on $U^{\an}$.  Recall that $X$ satisfies the envelope property over $K$ (Theorem \ref{thm:smooth_envelope_property}). Then, Proposition \ref{prop:limit_of_GPSH_functions} implies that $\varphi^{\star} \in \GPSH_{2A\theta}(X^{\an})$. Since $2A\theta \in \amp(X)$, Corollary \ref{cor:GPSH_plus_potential_of_positive_forms_is_pshlogreg}  implies that $\widetilde{u}=\varphi^{\star}+2Ag \in \PSHlogreg(Y^{\an})$ with $\widetilde{u}|_{U^{\an}}=u$. Therefore, the theorem is proven.
\end{proof}

\begin{lemma} \label{lem:blowup_pi_ample_divisor}
  Let $F$ be any field and let $X$ be a projective variety over $F$. Let $D$ be an effective ample divisor on $X$. Let $Z$ be a closed subscheme of $X$ such that $Z \subset \mathrm{Supp}(D)$, and let $\pi \colon \widetilde{ X} \coloneqq \mathrm{Bl}_Z X \to X$ be the blowup of $X$ centered at $Z$. Then, there exists an effective ample divisor $H$ on $\widetilde{X}$ such that $\mathrm{Supp}(H)=\pi^{-1}(\mathrm{Supp}(D))$. 
\end{lemma}
\begin{proof}
  We take $H \coloneqq m (\pi^{*}D)-E$ with $m$ being a large enough positive integer, and $E $ being the exceptional divisor of $\pi$, which is a $\pi$-ample effective Cartier divisor on $\widetilde{X}$ and satisfies $\CO_{\widetilde{X}}(E)=\pi^{-1}\CI_Z \cdot \CO_{\widetilde{X}}$. 
\end{proof}

\section{\MongeA measures of bounded psh functions} \label{sec:Monge_Ampere_measures}
In this section, unless indicated otherwise, we assume that $K$ is a discretely valued or trivially valued NA field, and $X$ is a projective variety over $K$ with $\dim X=n$. 

\subsection{The local \MongeA operator for PL functions}
We recall the definition of local non-archimedean \MongeA measures for PL functions following \cite{CLD}. We refer readers to \emph{loc. cit.} for more details. 
For smooth functions \footnote{see Appendix \ref{appendix:psh_functions_CLD} for the definition.} $u_1, \cdots, u_n$ on $U$ which factorize the tropicalization $\Trop \colon U \to \BR^k$,   the mixed \MongeA measure $\bigwedge_{k} \dm \dmb u_k$ equals to  the real mixed \MongeA measure multiplied by the tropical weight on each $n$-dimensional polyhedron of $\Trop(U)$, and has zero mass on each polyhedron of $\Trop(U)$ of dimension less than $n$.

Now suppose that $u_1, \cdots, u_n \in \TConv \cap \PL(U)$. For each $u_k$,  it can be  approximated locally uniformly by a sequence of tropically convex smooth functions $\{u_k^j\}_{j \ge 1}$. A local version of Chern--Levine--Nirenberg inequality established in \emph{loc. cit.} allows one to define $\bigwedge_{k} \dm \dmb u_k$ as the weakly limit of $\bigwedge_{k} \dm \dmb u_k^j$. In general, every PL function can be written as a difference of two tropically convex PL functions. Therefore, we can extend the local mixed \MongeA operator to $\PL(U)$ by multi-linear expansions.  

An alternative way to define local mixed \MongeA measures for PL functions is through Minkowski weights on tropicalizations. We refer the readers to  \cites{GK17, BBS22, BMPS26, APW1} for more details. These two ways of defining local mixed \MongeA measures are equivalent.   

We collect the following results on local mixed \MongeA measures.
\begin{prop} \label{prop:properties_local_MA_PL_functions}
  \begin{enumerate}
    \item Let $u_1, u_2, \cdots, u_n \in \TConv \cap \PL(U)$. Then, $\dm \dmb u_1 \wedge \cdots \wedge \dm \dmb u_n \ge 0$ on $U$. 
    \item Let $u_1, u_2, \cdots, u_n \in \PL(U)$. Suppose that $u_1=\log|f|$ for $f \in \CO^{\times }(U)$. Then, $\dm \dmb u_1 \wedge \cdots \wedge \dm \dmb u_n=0$. 
    \item (\cite[Theorem 18.4.7]{CLD}) The definition of local mixed \MongeA measures for PL functions is compatible with global definition of mixed \MongeA measures given in Section \ref{subsec:MA_measures_omega_psh_functions}. More precisely, let $\omega_1, \cdots, \omega_n \in \PLforms(X)$ with potential $g_1, \cdots, g_n \in \TConv \cap \PL(U)$ on $U$, and let $\varphi_1, \cdots, \varphi_n \in \PL(X)$. Denote   $u_k=\varphi_k +g_k \in \PL(U)$ for each $1 \le k \le n$. Then, we have  
    \[
      \bigwedge_{k} \dm \dmb u_k \coloneqq \bigwedge_{k}(\omega_k+\dm \dmb \varphi_k).
    \]
  \end{enumerate}
\end{prop}

\subsection{Local Bedford-Taylor theory}

\begin{thm} \label{thm:local_Bedford-Taylor_theory}
  For each open subset $U \subset X^{\an}$, there is a unique mixed \MongeA operator
  \bens
    (\PSHb(U))^n & \to & \Radonp{U}  \\
    (u_1, \cdots, u_n) &\mapsto& \dm \dmb u_1 \wedge \cdots \wedge \dm \dmb u_n
  \eens
  satisfying the following properties.
  \begin{enumerate}
  \item It extends the local mixed \MongeA operator on $\TConv \cap \PL(U)$. 
  \item Suppose that $V \subset U \subset X^{\an}$ are open subsets and  $u_1, \cdots u_n \in \PSHb(U)$. Then, $(\bigwedge_{k} \dm \dmb u_k)|_V= \bigwedge_{k} \dm \dmb (u_k|_V)$.
  \item Let $U \subset X^{\an}$ be an open subset and $u_0, u_1, \cdots, u_n \subset \PSHb(U)$. Denote $\mu \coloneqq \bigwedge_k \dm \dmb u_k$. Then, $u_0 \in L^1(U, \mu)$. 
  \item Under the hypothesis of (3), suppose further that for each $0 \le k \le n$, $\{u_k^j\}_{j \ge 1} \subset \TConv \cap \PL(U)$ is a  decreasing function sequence converging pointwise to $u_k$, and denote $\mu_j \coloneqq \bigwedge_k \dm \dmb u_k^j$. Then, $u_0^j \, \mu_j$ converges weakly to $u_0 \, \mu$ on $U$. 
\end{enumerate}
\end{thm}

\begin{proof}
  Let $U \subset X^{\an}$ be an open subset, we say that $u \in \PSH(U)$ is regularizable on $U$ if there exists a decreasing sequence of tropically convex PL functions converging pointwise to $u$ on $U$, and we denote by $\PSH_{\reg}(U)$ the space of regularizable psh functions on $U$. Property (1) and (4) show that the mixed \MongeA operator on $\PSH_{\reg} \cap L^{\infty}(U)$ is unique, and property (2) imply further that the mixed \MongeA operator is unique on $\PSHb(U)$. It remains to prove the existence. It suffices to show that for each open subset $U \subset X^{\an}$, there exists a mixed \MongeA operator on $\PSH_{\reg} \cap L^{\infty}(U)$ satisfying property (1) to (4) in the statement. 
  
  Fix $\omega \in \nef(X)$ such that $c_1(L(\omega)) \in \Amp(X)$.  Let $U \subset X^{\an}$ open, and let $u_0, u_1, \cdots, u_n \in \PSH_{\reg} \cap L^{\infty}(U)$. 
  Then, by Proposition \ref{prop:passing_from_local_to_global} (1), for any $x \in U$, there exists an open neighborhood $U_x$ of $x$ in $U$ only depending on $(U,x)$, constants $\{C_i\}_{0 \le k \le n} \subset \Q_{>0}$, $\{g_i\}_{0 \le k \le n} \subset \TConv \cap \PL(U_x)$ and $\{\varphi_k \}_{0 \le k \le n} \subset \GPSHob(X)$ such that for each $0 \le k \le n$, $g_k$ is a potential of $\omega$ on $U_x$, and $\varphi_k+g_k=C_k u$ on $U_x$. We define the nonnegative Radon measure on $U_x$ by 
  \[
  \mu(U_x; u_1, \cdots, u_n) \coloneqq (\prod_{1 \le k \le n} C_i)^{-1} \cdot \, \mathop{\bigwedge}_{1 \le k \le n}(\omega+\dm \dmb \varphi_k)  \qquad \text{on } U_x. 
  \]
  Then, 
  \begin{enumerate}[label=(\alph*)]
    \item Proposition \ref{prop:properties_local_MA_PL_functions} (3) implies that $\mu(U_x; u_1, \cdots, u_n)=\dm \dmb u_1 \wedge \cdots \wedge \dm \dmb u_n$ on $U_x$ when $\{u_k\}_{1 \le k \le n} \subset \TConv \cap \PL(U_x)$.
    \item Theorem \ref{thm:Bedford-Taylor_global_omega-psh} (2) implies that $u_0 \in L^1 \left(U_x, \mu(U_x; u_1, \cdots, u_n) \right)$.
    \item  Let $\{u_k^j\}_{j \ge 1} \subset \TConv \cap \PL(U_x)$ be a  decreasing function sequence converging pointwise to $u_k$ on $U_x$ for each $0 \le k \le n$. Then, after the proof of Proposition \ref{prop:passing_from_local_to_global} (2),  there exists $j_0 \in \N$ and  $\varphi_k^j \in \GPSHob(X)$ for each $0 \le k \le n$ and $j \ge j_0$ such that, $\varphi_k^j+g_k=C_k u$ on $U_x$, and $\{\varphi_k^j\}_{j \ge j_0}$ converges pointwise to $\varphi_k$ on $X^{\an}$. Then, Theorem \ref{thm:Bedford-Taylor_global_omega-psh} (3) implies that $u_0^j  \mu(U_x; u_1^j, \cdots, u_n^j)$ converges weakly to $u_0  \mu(U_x; u_1, \cdots, u_n)$ on $U_x$.
\end{enumerate}
  Combining (a) and (c), we obtain that for any $x, y\in U$, 
  \[
  \mu(U_x; u_1, \cdots, u_n)|_{U_x \cap U_y}=\mu(U_y; u_1, \cdots, u_n)|_{U_x \cap U_y}.
  \]
  Therefore, we define the mixed \MongeA measure for $u_1, \cdots, u_n \in \PSH_{\reg} \cap L^{\infty}(U)$ by 
  \[
    (\dm \dmb u_1 \wedge \cdots \wedge \dm \dmb u_n)|_{U_x}= \mu(U_x; u_1, \cdots, u_n) \qquad \text{for any } x \in U.
  \]
  Then, property (1), (3) and (4) follows from (a), (b) and (c) respectively. Property (2) follows from the locality of the mixed \MongeA operator on $\TConv \cap \PL(U)$ and property (1), (4). 
\end{proof}

\begin{remark}
  Combining Proposition \ref{prop:properties_local_MA_PL_functions} (3)  and Theorem \ref{thm:local_Bedford-Taylor_theory} (1), we obtain that the mixed \MongeA measures for bounded local psh functions  are compatible with the mixed \MongeA measures for bounded $\omega$-psh functions defined in Theorem \ref{thm:Bedford-Taylor_global_omega-psh}.
\end{remark}  

\begin{remark}
  Since the \MongeA operator for PL functions is local, Theorem \ref{thm:local_Bedford-Taylor_theory} (1) implies that the definition of mixed \MongeA operators  for locally bounded psh functions on $U$  does not depend on the choices of $X$. 
\end{remark}

\begin{remark} \label{rmk:mixed_Monge-Ampere_constant_functions}
  Let $u_1, \cdots, u_n \in \PSHb(U)$. By Proposition \ref{prop:properties_local_MA_PL_functions} (2) and Theorem \ref{thm:local_Bedford-Taylor_theory} (1), if $u_1=\log|f|$ for some $f \in \CO^{\times}(U)$, then, $\dm \dmb u_1 \wedge \cdots \wedge \dm \dmb u_n=0$ on $U$.  \qedhere
\end{remark}

\begin{remark} \label{rmk:decreasing_limit_of_psh_functions_is_not_known_to_be_psh_in_general}
  Suppose that $U \subset X^{\an}$ is an open subset and  $\{u_i\}_{i \ge 1} \subset \PSH(U)$ is a decreasing function sequence. Let $u$ be the pointwise limit of $\{u_i\}_{i \ge 1}$.  It is not known to the author whether in general,  $u  \in \PSH(U)$. This is because for $x \in U$, the open neighborhood $U_{x,i}$ on which $u_i$ is regularizable depends on $i$.  However, $u \in \PSH(U)$  when $X$ satisfies the envelope property (see Corollary \ref{cor:properties_PSH_after_envelope_properties} (1)).  
\end{remark}

\begin{definition}
  The \MongeA measure associated to $u \in \PSHb(U)$ is defined to be $\MA(u) \coloneqq (\dm \dmb u)^{\wedge n}$.
\end{definition}

\begin{definition} \label{def:mixed_Monge-Ampere_LPSHo}
  Let $\omega_1 \cdots, \omega_n \in \nef(X)$, and let $\varphi_1, \cdots, \varphi_n \in \LPSHob(X)$. The mixed \MongeA measure $\nu=\bigwedge_k (\omega_k + \dm \dmb \varphi_k)$ on $X^{\an}$ is defined in the following way. Let $U \subset X^{\an}$ be an affine open subset such that $\omega_1 \cdots, \omega_n$ admit potentials $g_1, \cdots, g_n \in \TConv \cap \PL(U)$ on $U$. Then, $\nu|_{U} \coloneqq \bigwedge_k \dm \dmb (g_k+\varphi_k)$. 
\end{definition}

\begin{remark}
  By Remark \ref{rmk:existence_of_potentials} (1), each $x \in X^{\an}$ admits such an affine open neighborhood. Moreover, the definition of $\nu$ does not depend on the potentials. This is because a different choice of potential of $\omega_k$ differs by $\log|f|$ for some $f \in \CO^{\times}(U)$ (Remark \ref{rmk:existence_of_potentials} (2)), and $\dm \dmb f=0$ by Remark \ref{rmk:mixed_Monge-Ampere_constant_functions}. 
\end{remark}

\begin{remark}
  Let $\varphi_1, \cdots, \varphi_n \in \GPSHob(X^{\an}) \subset \LPSHob(X^{\an})$. Then, by Proposition \ref{prop:properties_local_MA_PL_functions} (3) and Theorem \ref{thm:local_Bedford-Taylor_theory} (1), the mixed \MongeA measures $\bigwedge_k (\omega_k + \dm \dmb \varphi_k)$ defined by Theorem \ref{thm:Bedford-Taylor_global_omega-psh} and Definition \ref{def:mixed_Monge-Ampere_LPSHo} are equal.  
\end{remark}

\begin{thm} [Locality for \MongeA measures] \label{thm:locality_MA_measures}
  \leavevmode
  \begin{enumerate}
    \item Let $U \subset X^{\an}$ be an open subset and  $u,v \in \PSHb(U)$. Then, 
      \[
      \mathbf{1}_{\{ u> v \}} \MA(\max \{u,v\})= \mathbf{1}_{\{ u> v \}} \MA(v) \qquad \text{on } U.
      \]
    \item Let $\omega \in \nef(X)$ and $\varphi, \psi \in \LPSHob(X)$.  Then, 
      \[
      \mathbf{1}_{\{ \varphi> \psi \}} \MAo(\max \{\varphi, \psi\})= \mathbf{1}_{\{ \varphi> \psi \}} \MAo(\varphi) \qquad \text{on } X^{\an}.
      \]
  \end{enumerate}
\end{thm}
\begin{proof}
  This follows from the locality for \MongeA measures for $\omega$-psh functions (Theorem \ref{thm:locality_MA_omega-psh}) and Proposition \ref{prop:passing_from_local_to_global}. 
\end{proof}

\subsection{Global Properties of \MongeA measures}

\begin{prop} \label{prop:volume_local_psh}
  Let $\omega_1 \cdots, \omega_n \in \nef(X)$, and let $\varphi_1, \cdots, \varphi_n \in \LPSHob(X)$. Then, 
  \[
    \int_X \mathop{\bigwedge}_{k=1}^{n} (\omega_k+\dm \dmb \varphi_k)= \int_X \omega_1 \wedge \cdots \wedge \omega_n. 
  \]
  In particular, for any $\omega \in \nef(X)$ and $\varphi \in \LPSHob(X)$, $\int_X \MA_{\omega}(\varphi)=\deg(\omega^d)$.
\end{prop}

\begin{lemma} [Partition of unity for finite open covers by PL functions] \label{lem:partition_of_unity_PL_functions}
  \leavevmode
  Let $\{U_{\alpha}\}_{\alpha \in \CI}$ be a finite open cover of $X^{\an}$. Then, for each $\alpha \in \CI$, there exists $\rho_{\alpha} \in \PL(U_{\alpha})$ such that 
  \begin{enumerate}
    \item For each $\alpha \in \CI$, $\mathrm{Supp}(\rho_{\alpha}) \Subset U_{\alpha}$, and $0 \le \rho_{\alpha} \le 1$.
    \item $\sum_{\alpha \in \CI} \rho_{\alpha}=1$ on $X^{\an}$. 
  \end{enumerate}
\end{lemma}
\begin{proof}
  After shrinking each $U_{\alpha}$, we may assume that there exists a tropicalization map $\Trop \colon X^{\an} \to \TP_{\Sigma}$ and open subsets $V_{\alpha}$ for each $\alpha$ such that $\Trop^{-1}(V_{\alpha})=U_{\alpha}$ (see Section \ref{subsubsec:tropicalizations}). Since the composition of a PL function on $\TP_{\Sigma}$ and $\Trop$  is a PL function on $X^{\an}$ (see Lemma \ref{lem:model_functions_equals_to_PL_functions}), it is enough to prove the partition of unity for a finite open cover $\{V_{\alpha}\}_{\alpha \in \CI}$ of the tropical toric variety $\TP_{\Sigma}$. 

  We may take a simplicial polyhedral complex $\Pi$ whose support equal to $\TP_{\Sigma}$ having the following property: there exists a map $i \colon V \to \CI$ such that $\Pi^v \subset V_{i(v)}$ for each $v \in V$. Here, $V$ denotes the set of vertices of $\Pi$, and $\Pi^v$ denotes the star $\Pi^v$ of $v$.

  For each $v \in V$, define the function $\chi_v \colon V \to \R$ by $\chi_v(v)=1$ and $\chi_v(v')=0$ for any $v' \neq v \in V$. Define the facewise affine function $\rho_v$ on $\Pi$ as the linear extension of the $\chi_v$ to $\TP_{\Sigma}$. More precisely,  let $\sigma \in \Pi$ be a simplicial polyhedron. If $\sigma$ is a simplex, then $\rho_v$ equals to the unique affine function on $\sigma$ whose restriction to the vertices of $\sigma$ equals $\chi_v$. Otherwise, there exists a canonical  projection $p \colon \sigma \to \tau$ with $\tau \in \Pi$ a simplex, and  we have $\rho_v=\rho_v \circ p$ on $\sigma$. 
  
  Note that we have $\sum_{v \in V} \rho_v=1$ on $\TP_{\Sigma}$, and $\mathrm{Supp}(\rho_v)=\Pi^v$ for each $v \in V$. For each $\alpha \in \CI$, we define $\rho_{\alpha} \coloneqq \sum_{v \in i^{-1}(\alpha)} \rho_v$. Then, $\{\rho_{\alpha}\}_{\alpha \in \CI}$ gives a partition of unity subordinate to the cover $\{V_{\alpha}\}_{\alpha \in \CI}$.
\end{proof}

\begin{lemma} \label{lem:integration-by-parts_PL_functions}
  Let $U \subset X^{\an}$ be an open subset. Let $f_1, \cdots, f_n \in \PL(U)$, and let $g \in \PL(U)$ such that $\mathrm{Supp}(g) \Subset U$. Then, 
  \[
   \int_U g \, \dm \dmb f_1 \wedge \dm \dmb f_2 \wedge \cdots \wedge \dm \dmb f_n = \int_U f_1 \, \dm \dmb g \wedge \dm \dmb f_2 \wedge \cdots \wedge \dm \dmb f_n 
  \]
\end{lemma}
\begin{proof}
  Let $\widetilde{f}_1, \cdots, \widetilde{f}_n \in \PL(X)$ be extensions of $f_1, \cdots, f_n$. Since $\mathrm{Supp}(g) \Subset (U)$, we may extend $g$ to $\widetilde{g} \in \PL(X)$ by setting $\widetilde{g} \equiv 0$ on $X^{\an} \setminus U$. By Remark \ref{rmk:mixed_Monge-Ampere_constant_functions}, we have 
  \[
  \int_U g \, \dm \dmb f_1 \wedge \dm \dmb f_2 \wedge \cdots \wedge \dm \dmb f_n = \int_X \widetilde{g} \, \dm \dmb \widetilde{f}_1  \wedge \dm \dmb \widetilde{f}_2 \wedge \cdots \wedge \dm \dmb \widetilde{f}_n,
  \]
  and 
  \[
  \int_U f_1 \, \dm \dmb g \wedge \dm \dmb f_2 \wedge \cdots \wedge \dm \dmb f_n  = \int_X \widetilde{f}_1 \, \dm \dmb \widetilde{g} \wedge \dm \dmb \widetilde{f}_2 \wedge \cdots \wedge \dm \dmb \widetilde{f}_n.
  \]
  Then, the lemma follows from the global integration-by-parts for PL functions (\cite[Proposition 2.20]{BFJ15}).  
\end{proof}

\begin{proof} [Proof of Proposition \ref{prop:volume_local_psh}] \label{proof:prop:volume_local_psh}
  Denote $T=(\omega_2+\dm \dmb \psi_2) \wedge \cdots \wedge (\omega_n+\dm \dmb \psi_n)$. We show that 
  \be \label{equation:integration-by-parts_simple_version}
  \int_X 1 \, \dm \dmb \varphi_1  \wedge T=\int_X \varphi_1 \, (\dm \dmb 1) \wedge T=0 
  \ee
  Granted this equation, repeated application yields
  \[
  \int_X \mathop{\bigwedge}_{k=1}^{n} (\omega_k+\dm \dmb \varphi_k) = \int_X \omega_1 \wedge  \left( \mathop{\bigwedge}_{k=2}^{n} (\omega_k+\dm \dmb \varphi_k) \right) = \cdots =\int_X \omega_1 \wedge \cdots \omega_n.
  \]
  Therefore, the proposition is proved. 

  Now we prove Equation \ref{equation:integration-by-parts_simple_version}. Let $\{U_{\alpha}\}_{\alpha \in \CI}$ be a finite open cover of $X$ such that, for each $\alpha \in \CI$ and $0 \le k \le n$, there exists a potential $g_{k, \alpha} \in \TConv \cap \PL(U_{\alpha})$ of $\omega_k$, and  a decreasing function sequence $\{u_{k, \alpha}^i\}_{i \ge 1} \subset \PSH(U_{\alpha})$ satisfying $u_{k, \alpha}^i \to u_{k, \alpha}$ on $U_{\alpha}$, where $u_{k, \alpha} \coloneqq \varphi_k+g_{k, \alpha} \in \PSH(U_{\alpha})$. For each $\alpha \in \CI$ and $ i \ge 1$, denote $T_{\alpha}^i \coloneqq \mathop{\bigwedge}_{k=2}^{n}   \dm \dmb u_{k, \alpha}^j$. Let $(\rho_{\alpha}) \in \PL \cap C^{0}_{c}(U_{\alpha})$ be a partition of unity of $\{U_{\alpha}\}_{\alpha \in \CI}$. Then,
  \[
  \begin{aligned}
  \int_X 1 \, \dm \dmb \varphi_1 \wedge T &= \sum_{\alpha\in I} \int_{U_\alpha} \rho_\alpha \left(dd^c u_{1,\alpha}-dd^c g_{1,\alpha}\right)\wedge T \\
  &=  \sum_{\alpha\in I} \left( \lim_{i\to+\infty} \int_{U_\alpha} \rho_\alpha \, \dm \dmb u_{1,\alpha}^i \wedge T_{\alpha}^i - \int_{U_\alpha} \rho_\alpha \, \dm \dmb g_{1,\alpha} \wedge T_{\alpha}^i \right) \quad \left( \text{Theorem } \ref{thm:local_Bedford-Taylor_theory}~(4) \right) \\
  &= \sum_{\alpha\in I} \left( \lim_{i\to+\infty} \int_{U_\alpha} u_{1,\alpha}^i \, \dm \dmb \rho_\alpha \wedge T_{\alpha}^i - \int_{U_\alpha} g_{1,\alpha} \, \dm \dmb \rho_\alpha \wedge T_{\alpha}^i \right) \quad \left( \text{Lemma } \ref{lem:integration-by-parts_PL_functions} \right).
  \end{aligned}
  \]

  By Lemma \ref{lem:property_tropical_conv_functions} (4), for each $\alpha \in \CI$, there exists $\rho_{\alpha}^{+}, \rho_{\alpha}^{-} \in \TConv \cap PL(U_{\alpha})$ such that $\rho_{\alpha}=\rho_{\alpha}^{+}-\rho_{\alpha}^{-}$. On the other hand, Remark \ref{rmk:mixed_Monge-Ampere_constant_functions} implies that $\dm \dmb \rho_\alpha \wedge T_{\alpha}^i =0$ on $U_{\alpha} \setminus \mathrm{Supp}(\rho_{\alpha})$, and Lemma \ref{lem:partition_of_unity_PL_functions} implies that there exists $\rho_{\alpha}' \in \PL \cap C^{0}_{c}(U_{\alpha})$ such that $\rho_{\alpha}' \equiv 1$ on $\mathrm{Supp}(\rho_{\alpha})$. Hence, 
  \[
    \begin{aligned}
    \lim_{i\to+\infty}\int_{U_\alpha} u_{1,\alpha}^i\,\dm\dmb\rho_\alpha\wedge T_\alpha^i &= \lim_{i\to+\infty}\int_{U_\alpha}\rho_\alpha' u_{1,\alpha}^i\left(\dm\dmb\rho_\alpha^+-\dm\dmb\rho_\alpha^-\right)\wedge T_\alpha^i \\
    &= \int_{U_\alpha}\rho_\alpha' u_{1,\alpha}\left(\dm\dmb\rho_\alpha^+-\dm\dmb\rho_\alpha^-\right)\wedge T \qquad \left( \text{Theorem } \ref{thm:local_Bedford-Taylor_theory}~(4) \right) \\
    &= \int_{U_\alpha}u_{1,\alpha}\,\dm\dmb\rho_\alpha\wedge T.
    \end{aligned}
  \]
  Combining the above formulas, we have
  \[
  \begin{aligned}
    \int_X 1\, \dm \dmb \varphi_1 \wedge T &= \sum_{\alpha\in I}  \int_{U_\alpha} \left( u_{1,\alpha} - g_{1,\alpha} \right) \,\dm\dmb\rho_\alpha\wedge T  \\
    &=  \int_{X} \varphi_1 \,\dm\dmb 1 \wedge T =0.
  \end{aligned}
  \]
  Therefore, Equation \ref{equation:integration-by-parts_simple_version} is proved. 
\end{proof}

\begin{thm}[Comparison Principle]  \label{thm:comparison_principle}
Let  $\varphi, \psi \in \LPSHob(X)$. We have 
\begin{equation*}
    \int_{\{ \varphi < \psi \} } \MA_{\omega} (\psi) \le \int_{\{ \varphi < \psi \}} \MA_{\omega} (\varphi)
\end{equation*}
\end{thm}
\begin{proof}
    By Lemma \ref{lem:finite_sup_lpsh}, for any $\epsilon >0$, $\max \{ \varphi, \psi-\epsilon \} \in \LPSHob(X)$. Then,
    \bens
        && \deg (\omega^{d}) 
        = \int_{X } \MA_{\omega} (\max \{ \varphi, \psi-\epsilon \}) \qquad \qquad \qquad \qquad \qquad \qquad (\text{Proposition } \ref{prop:volume_local_psh}) \\
        &\ge& \int\limits_{\{ \varphi < \psi-\epsilon \} } \MA_{\omega} (\max \{ \varphi, \psi-\epsilon \}) + \int\limits_{\{ \varphi > \psi-\epsilon \} } \MA_{\omega} (\max \{ \varphi, \psi-\epsilon \})  \\
        &=& \int\limits_{\{ \varphi < \psi-\epsilon \} } \MA_{\omega} (\psi) + \int\limits_{\{ \varphi > \psi-\epsilon \} }  \MA_{\omega} (\varphi) \qquad \qquad \qquad \qquad \qquad \qquad (\text{Theorem } \ref{thm:locality_MA_measures})  \\
        &=&\int\limits_{\{ \varphi < \psi-\epsilon \} } \MA_{\omega} (\psi) + \deg(\omega^{d}) - \int\limits_{\{ \varphi \le \psi-\epsilon \} }  \MA_{\omega} (\varphi)
    \eens
    Therefore, we have $ \int_{\{ \varphi < \psi-\epsilon \} } \MA_{\omega} (\psi) \le \int_{\{ \varphi \le \psi-\epsilon \} }  \MA_{\omega} (\varphi)$. The proof is completed by letting $\epsilon \to 0$.
\end{proof}

\begin{thm}[Domination Principle] \label{thm:Domination_Principle}
  Assume that $X$ is irreducible. Let $\omega \in \nef(X)$ such that $c_1(L(\omega)) \in \Amp(X)$ and $\GPSHo(X)$ satisfies the envelope property. Let $\varphi, \psi \in \LPSHob(X)$.  Suppose  that $\varphi \le \psi$ almost everywhere on $\mathrm{Supp} (\MA_{\omega} (\psi))$. Then $\varphi \le \psi$ on $X$.  
\end{thm}
\begin{proof}
  By Proposition \ref{prop:local_comparison_divisorial_points_to_full_spaces}, it is enough to prove that $\varphi (x) \le \psi(x)$ for any   $x \in X^{\div}$. We  assume that $\deg(\omega^d)=1$. For any $v \in \LPSHob(X^{\an})$, denote $\omega_v \coloneqq \omega+\dm \dmb v$.   
  Let $\mu=\delta_x$ be the Dirac mass measure supported on $\{x\}$. Then by \cite{BFJ15}, there exists $u \in C^0 \cap \GPSHo(X)$ such that $\MA_{\omega}(u)=(\omega_{u})^d=\mu$. Since both $u$ and $\varphi$ are bounded,  after adding a constant to $u$, we may assume that $u <\varphi$.

  For any $\epsilon >0$, denote $\varphi_{\epsilon} \coloneqq (1-\epsilon)\varphi+\epsilon u \in \LPSHob(X)$. Then, $\MA_{\omega} (\varphi_{\epsilon})=((1-\epsilon)\omega_{\varphi} +\epsilon \omega_{u})^{d}=\sum_{j=0}^{d} ((1-\epsilon)\omega_{\varphi})^{j} \wedge (\epsilon \omega_{u})^{d-j} \ge \epsilon^d \omega^d =  \epsilon^d \mu$. Therefore, $\MA_{\omega} (\varphi_{\epsilon})(\{x\})>0$. 
  
  That $u < \varphi$ implies $\varphi_{\epsilon} < \varphi$ on $X^{\an}$. Then, by assumption, $\varphi_{\epsilon} < \varphi \le \psi$ on $\mathrm{Supp}(\MA_{\omega} (\psi))$. By comparison principle (Theorem \ref{thm:comparison_principle}) and Proposition \ref{prop:volume_local_psh}, we have 
\bens
\deg (\omega^{d}) = \int_{\{ \varphi_{\epsilon} < \psi \} } \MA_{\omega} (\psi) \le \int_{\{ \varphi_{\epsilon} < \psi \} }  \MA_{\omega} ( \varphi_{\epsilon})  \le  \deg (\omega^{d})
\eens
This implies $\int_{\{ \varphi_{\epsilon} < \psi \} }  \MA_{\omega} ( \varphi_{\epsilon})=\deg (\omega^{d})$, and hence  $x \in  \{ \varphi_{\epsilon} < \psi \}$. Now by letting $\epsilon \to 0+$, we conclude that $\varphi(x) \le \psi(x)$.
\end{proof}

\subsection{Comparison between two notions of $\omega$-psh functions}

\begin{thm} \label{thm:GPSHo_equals_LPSHo}
  Let $\omega \in \nef(X)$. Assume that $\GPSHo(X)$ satisfies the envelope property. Then, $\GPSHo(X)=\LPSHo(X)$.
\end{thm}
 
\begin{proof}
  Since $\GPSHo(X)$ satisfies the envelope property, \cite[Remark 5.15]{BJ22} implies that each connected component of $X$ is irreducible. Therefore, we may assume that $X$ is irreducible. Let $\theta \in \nef(X)$ such that $c_1(L(\theta)) \in \Amp(X)$. Note that for any function $\varphi \colon X^{\an} \to R \cup \{-\infty\}$, $\varphi \in \GPSHo(X)$ (\emph{resp.} $\LPSHo(X)$) iff $\varphi \in \GPSH_{\omega+\epsilon \theta}(X)$ (\emph{resp.} $\LPSH_{\omega+\epsilon \theta}(X)$) for each $\epsilon >0$, and $c_1(L(\omega+\epsilon \theta)) \in \Amp(X)$ for any $\epsilon >0$. Therefore, we may henceforth assume that $c_1(L(\omega)) \in \Amp(X)$. 
  
  By Corollary \ref{cor:GPSH_is_LPSH}, we have $\GPSHo(X) \subset \LPSHo(X)$. Thus, we only need to prove the reverse direction.  Let $\varphi \in \LPSHo(X)$, our aim is to show that $\varphi \in \GPSHo(X)$. To show this, it is enough to show that $\sup\{\varphi, M\} \in \GPSHo(X)$ for any $M \in \BR$ (Theorem \ref{thm:decreasing_limit_of_omega-psh_functions_is_omega-psh}). Therefore, we assume henceforth that $\varphi$ is bounded.

  Since $\varphi$ is usc on $X$, we may choose a decreasing sequence of functions $\{f_i\}_{i \ge 1} \subset C^0(X)$ such that $f_i$ converges pointwise to $\varphi$. For each $i \ge 1$, recall that the $\omega$-psh envelope of $f_i$ is defined to be
  \[
  \envomega(f_i) \coloneqq \sup \{\psi \in \GPSHo(X) \ |\ \psi \le f \}.
  \]
  Since $\GPSHo(X)$ satisfies the envelope property,  Lemma \ref{lem:equivalent_form_of_envelope_conjecture} implies that  $\envomega(f_i) \in  \GPSHo \cap C^0(X)$. Furthermore, the orthogonality property for $\GPSHo(X)$ (Theorem 13.1 in \cite{BJ22}) implies that for each  $i \ge 1$
  \[
  \Supp(\MAo(\envomega(f_i))) \subset \{f_i=\envomega(f_i)\}.
  \]
  Hence, $\varphi \le f_i= \envomega(f_i)$ on $\Supp(\MAo(\envomega(f_i)))$. By domination principle (Theorem \ref{thm:Domination_Principle}), we have $\varphi \le \envomega(f_i)$ on $X^{\an}$. Thus, $\{\envomega(f_i)\}_{i \ge 1} \subset \GPSHo \cap C^0(X)$ is a decreasing sequence converging pointwise to $\varphi$. Therefore, by Theorem \ref{thm:decreasing_limit_of_omega-psh_functions_is_omega-psh}, $\varphi \in \GPSHo(X)$.  
\end{proof}

\begin{cor} \label{cor:properties_PSH_after_envelope_properties}
Let $U \subset X^{\an}$ be an open subset. Suppose that $X$ satisfies the envelope property over $K$. Then,
  \begin{enumerate}
    \item Let $\{u_i\}_{i \ge 1} \subset \PSH(U)$ be  a decreasing function sequence with  the pointwise limit $u$. If $u$ is generically finite, then $u \in \PSH(U)$. 
    \item Let $\{u_i\}_{i \in I}$ be a bounded-above family in  $\PSH(U)$. Then, \(  \sup_{i \in I} \{u_i\}  = (\sup_{i \in I} \{u_i\})^{\star} \) on $U \cap X^{\div}$.
    \item Let $\{u_i\}_{i \in I}$ be a bounded-above family in  $\PSH(U)$. Then, $(\sup_{i \in I} \{u_i\})^{\star} \in \PSH(U)$. 
    \item \label{cor:local_pointwise_limit_of_lpsh_functions} Let $\{u_i\}_{i \ge 1}$ be a bounded-above sequence of functions in $\PSH(U)$. Suppose that $u_i$ converges pointwise to $u$. Then, $u^{\star} \in \PSH(U)$, and $u^{\star}=u$ on $U \cap X^{\div}$.
  \end{enumerate}
\end{cor}

\begin{proof}
  Since these are local statement, it is sufficient to prove them on an open neighborhood of $x$ for each $x \in U$. For each $M \in \R$ and for each $i$, denote $u_{i,M} \coloneqq \max\{u_i, M\}$. Observe that $\sup_{i \in I} \{ u_{i, M} \} =\max\{ \sup_{i \in I} \{u_i\}, M\}$, $(\sup_{i \in I} \{ u_{i, M} \} )^{\star}=\max\{ (\sup_{i \in I} \{u_i\})^{\star}, M\}$, and $\lim_{i \to +\infty} u_{i,M}=\max\{ u, M\}$ if $u=\lim_{i \to +\infty} u_i $. Therefore,  for each statement (1)--(4), after replacing $\{u_i\}$ by $\{u_{i,M}\}$ and replacing $U$ by a compact open neighborhood of $x$, we may assume that $\left\| u_i \right\|_{L^{\infty}(U)} \le M$ for some $M \in \R$. 

  Take $\omega \in \nef(X)$ with $c_1(L(\omega)) \in \Amp(X)$. Applying Proposition \ref{prop:passing_from_local_to_global} (1) to $M$ and $x \in U$, we obtain an open neighborhood $U_x \subset U$ of $x$, a constant $C>0$, a potential $g \in \PL \cap \LConv(U_x)$ of $\omega$, and $\varphi_i \in \LPSHo(X)$ for each $i$, such that $\varphi_i+g=Cu_i$ on $U_x$. Theorem \ref{thm:GPSHo_equals_LPSHo} implies that for each $i$, $\varphi_i \in \GPSHo(X)$. Therefore, statement (1), (2), (3) and (4) follows respectively from Theorem \ref{thm:decreasing_limit_of_omega-psh_functions_is_omega-psh}, Theorem \ref{thm:global_divisorial_points_is_nonnegligible}, the assumption that $\GPSHo(X)$ satisfies the envelope property and Proposition \ref{prop:limit_of_GPSH_functions}. 
\end{proof}

\begin{prop} \label{prop:pshlog_equals_pshlogreg_on_smooth_affine_varieties}
  Suppose that $K$ is a discretely valued or trivially valued NA field of equicharacteristic $0$, and $Y$ is a smooth affine variety over $K$. Then, $\PSHlog(Y^{\an})=\PSHlogreg(Y^{\an})$. 
\end{prop}
\begin{proof}
  As explained in the proof of Theorem \ref{thm:pshlogreg_extension},  Hironaka's theorem and Lemma \ref{lem:blowup_pi_ample_divisor} implies that there exists a smooth projective compactification $X$ of $Y$ over $K$, a very ample line bundle $L$ on $X$ with $s \in H^0(X, L)$ satisfying $Y=X \setminus Z(s)$.  Let $\theta \in \amp(X)$ such that $L = L(\theta)$, and let $|\cdot|_{\theta}$ be a Fubini-Study metric on $L$. Then, by Proposition \ref{prop:potential_of_positive_form_is_Fubini-Study} (1), $g \coloneqq -\log|s|_{\theta}$ is  a Fubini-Study function on $Y^{\an}$. 
  Let $u \in \PSHlog(Y^{\an})$. By Lemma \ref{lem:FS_functions_are_comparable}, there exists $A>0$ such that $u \le Ag+O(1)$. 
  Fix $c>1$. Define the function $\varphi \colon X^{\an} \to \R \cup \{-\infty\}$ by $\varphi=u-cAg$ on $Y^{\an}$ and $\varphi=-\infty$ on $X^{\an} \setminus Y^{\an}$. Then, Lemma \ref{lem:psh_regularization_with_growth_control} (1) implies that $\varphi \in \LPSH_{cA \theta}(X)$. Recall that $X$ satisfies the envelope property over $K$ (Theorem \ref{thm:smooth_envelope_property}). Then, Theorem \ref{thm:GPSHo_equals_LPSHo} implies that $\varphi \in \GPSH_{cA \theta}(X)$. Since $cA\theta \in \amp(X)$, Corollary \ref{cor:GPSH_plus_potential_of_positive_forms_is_pshlogreg}  implies that $u=\varphi+cAg \in \PSHlogreg(Y^{\an})$.
\end{proof}

\section{Runge domains} \label{sec:Runge_domains}
In this section, In this section, unless indicated otherwise, we assume that $K$ is a discretely valued or trivially valued NA field, and $Y$ is a affine over $K$. Denote by $K[Y]$ the ring of polynomial functions on $Y$.

\subsection{Polynomially convexity and Runge open subset} \label{subsec:polynomially_convexity_and_Runge_open_subset}

\begin{definition}
  \begin{enumerate}
    \item Let $L \subset Y^{\an}$ be a compact subset. We define the polynomial convex hull in $Y^{\an}$ as 
      \[
        \widehat{L}_{Y} \coloneqq \{x \in Y^{\an}, \, |f|(x) \le \sup_{L} |f| \text{ for any } f \in K[Y]\}.
      \] 
    We say that $L$ is polynomially convex in $Y^{\an}$ if $L=\widehat{L}_{Y}$. 
    \item Let $D \subset Y^{\an}$ be an open subset. We say that $D$ is a (polynomially) Runge open subset in $Y^{\an}$ if for any compact subset $L \subset D$, we have $\widehat{L}_Y \subset D$. 
  \end{enumerate}
\end{definition}

\begin{lemma} \label{lem:basic_properties_poly_convex_and_Runge}
  \begin{enumerate}
    \item Let $L_1, L_2$ be compact subsets of $Y^{\an}$ with $L_1 \subset L_2$. Then, $\widehat{(L_1)}_{Y} \subset \widehat{(L_2)}_{Y}$.
    \item Let $D_1 \subset D_2 \subset \cdots D_k \subset \cdots$ be a sequence of Runge open subset in $Y^{\an}$. Then, $D \coloneqq \bigcup_{k \ge 1} D_k$ is Runge in $Y^{\an}$.  
    \item Let $Z$ be a closed subvariety of $Y$, and $L \subset Z^{\an}$ be a compact subset. Then, $\widehat{L}_Z=\widehat{L}_Y$. 
    \item In the setting of (3), suppose further that $D \subset Y^{\an}$ is a Runge open subset. Then, $D \cap Z^{\an}$ is Runge in $Z^{\an}$. 
  \end{enumerate}
\end{lemma}
\begin{proof}
  For statement (1), suppose that $x \in \widehat{(L_1)}_{Y}$. Then, for any $f \in K[Y]$, $|f|(x) \le \sup_{L_1}|f| \le \sup_{L_2}|f|$. Therefore, $x \in \widehat{(L_2)}_{Y}$. For statement (2), suppose that $L \subset D$ is a compact subset. Then, there exists $i \ge 1$ such that $L \subset D_i$. Then, by assumption, $\widehat{L}_{Y} \subset D_i \subset D$. 
  
  For statement (3), since $K[Y] \to K[Z]$ is surjective, it is clear that $\widehat{L}_Z  \subset \widehat{L}_Y$. To prove the converse direction,  take $x \in Y^{\an} \setminus \widehat{L}_Z$. We need to find $f \in K[Y]$ with $|f|(x)> \sup_{\widehat{L}_Z}|f|$. and denote $Z' \coloneqq \mathrm{Supp}(x)$. If $Z' \subset Z$, then we have $x \in Z^{\an}$, and by definition there exists $g \in K[Z]$ such that $|g|(x)> \sup_{\widehat{L}_Z}|g|$. Take $f \in K[Y]$ such that $f|_Z=g$. Then, $f$ satisfies the required condition. If $Z' \nsubseteq Z$, then there exists $f \in k[Y]$ such that $f$ does not vanish on $Z'$ and vanishes on $Z$. Then, $|f|(x)>0=\sup_{\widehat{L}_Z}|f|$, and hence $f$ satisfies the required condition. Statement (4) follows directly from statement (3) by definition. 
\end{proof}

\begin{lemma} \label{lem:tropical_polytopes_are_poly_convex}
  Let $P$ be a tropical polytope in $Y^{\an}$. Then, $P$ is polynomially convex in $Y^{\an}$, and $\interior(P)$ is a Runge open subset in $Y^{\an}$.
\end{lemma}
\begin{proof}
  Let $f_1, \cdots, f_k \in K[Y]$ and $t_1, \cdots, t_k \in \Q$ such that $P = \bigcap_{1 \le i \le k} \{-\log|f_i| \ge t_i\}$. Then, for any $x \in Y^{\an} \setminus P$, there exists $1 \le i \le k$ such that $-\log|f_i| < t_i$. Hence, $|f_i|(x) \le \sup_{P}|f_i|$, and $\widehat{P}_Y=P$. 
  
  Now we show that $\interior(P)$ is Runge. Take a compact subset $L \subset \interior(P)$. For each $\epsilon \in \Q_{>0}$, consider $P_{\epsilon}=\bigcap_{1 \le i \le k} \{-\log|f_i| \ge t_i+\epsilon\} \subset \interior(P)$. Since $P_{\epsilon}$ is a closed subset of the compact set $P$, $P_{\epsilon}$ is compact and hence a tropical polytope. Note that  $\interior(P)=\bigcup_{\epsilon \in \Q_{>0}} \interior(P_{\epsilon})$, and $P_{\epsilon_1} \supset P_{\epsilon_2}$ for any $\epsilon_1 > \epsilon_2$. Hence, there exists $\epsilon \in \Q_{>0}$ such that $L \subset P_{\epsilon}$. Then, by Lemma \ref{lem:basic_properties_poly_convex_and_Runge} (1), we have $\widehat{L}_Y \subset \widehat{(P_{\epsilon})}_Y=P_{\epsilon} \subset \interior(P)$. Therefore, $\interior(P)$ is Runge. 
\end{proof}

\begin{prop} \label{prop:exhausion_of_Runge_domains_by_tropical_polytopes}
  \begin{enumerate}
    \item For any compact subset $L \subset Y^{\an}$, $\widehat{L}_{Y}$ is compact.
    \item Let $L$ be a polynomially convex compact subset of $Y^{\an}$, and $U$ be an open neighborhood of $L$. Then, there exists a tropical polytope $P$ such that $L \subset \interior(P) \subset P \subset U$. 
    \item Assume that $K$ has a countable dense subfield.  Let $D \subset Y^{\an}$ be a Runge open subset. Then, there exists a sequence of tropical polytopes $P_1 \subset P_2 \subset \cdots \subset P_i \subset \cdots D$ such that 
    \[
    D=\bigcup_{i \ge 1} P_i, \quad \text{and } \quad P_i \subset \interior(P_{i+1}) \ \text{for each } i \ge 1.
    \]
  \end{enumerate}
\end{prop}
\begin{proof}
  Let $E_1 \subset E_2 \subset \cdots E_r \subset \cdots \subset Y^{\an}$ be a sequence of tropical polytopes constructed in Example \ref{example:tropical_polytopes_in_affine_varieties} with $Y^{\an}=\bigcup_{r \in \N} \interior(E_r)$. 
  
  Now we prove statement (1). Denote $M_f \coloneqq \sup_{L}|f|$ for each $f \in K[Y]$. Then, we have $\widehat{L}_Y=\bigcap_{f \in K[Y]} \{|f| \le M_f\}$. Hence, $\widehat{L}_Y$ is a closed subset of $Y^{\an}$.  On the other hand, there exists $r_1 \in \N$ such that $L \subset E_{r_1}$. Then, Lemma \ref{lem:basic_properties_poly_convex_and_Runge} (1) and Lemma \ref{lem:tropical_polytopes_are_poly_convex} imply that $\widehat{L}_Y \subset \widehat{(E_{r_1})}_Y =E_{r_1}$. Therefore, $\widehat{L}_Y$ is compact. 

  For statement (2),since $Y^{\an}$ is locally compact Hausdorff,  we may assume that $U$ is a compact open neighborhood of $L$. Then, there exists $r_2 \in \N$ such that $U \subset E_{r_2}$. By definition, for each $x \in E_{r_2} \setminus U$, there exists $f_x \in K[Y]$ and $t_x \in \Q$ such that $-\log|f_x|(x)<t_x$ and $-\log|f_x|\ge t_x$ on $L$. Take an open neighborhood $U_x$ of $x$ such that $U_x \subset \{-\log|f_x|<t_x\}$. Then, $\{U_x\}_{x \in E_{r_2} \setminus U}$ is an open cover of the compact set $E_{r_2} \setminus U$. Therefore, there exists a finite subset $\{x_1, \cdots, x_l\} \subset E_{r_2} \setminus U$ such that $E_{r_2} \setminus U \subset \bigcup_{1 \le i \le l} U_{x_i}$. Define $P \coloneqq E_{r_2} \cap (\bigcap_{1 \le i \le l} \{ -\log|f_{x_i}| \le t_{x_i} \} )$. Then, $P$ is compact and thereby a tropical polytope. By construction, $L \subset \interior(P)=\interior(E_{r_2}) \cap (\bigcap_{1 \le i \le l} \{ -\log|f_{x_i}| < t_{x_i} \} )$, $P \subset E_{r_2}$ and $P \cap (E_{r_2} \setminus U)= \emptyset$. Therefore, we have $L \subset \interior(P) \subset P \subset U$.

  Now, we prove statement (3). There exists a sequence of compact subsets $L_1 \subset \cdots \subset L_i \subset \cdots D$ such that $D=\bigcup_{i \ge 1}\interior(L_i)$ and $L_i \subset \interior(L_{i+1})$ for each $i \ge 1$. We construct by induction for each $i \ge 0$ a tropical polytope $P_{i+1}$ such that $P_{i} \subset \interior(P_{i+1})$ and $L_{i+1} \subset P_{i+1}$ (we start by $P_0=\emptyset$). 
  By definition, $\widehat{(P_i \cup L_{i+1})}_{Y} \subset D$. By statement (1), $\widehat{(P_i \cup L_{i+1})}_{Y}$  is a polynomially convex compact subset of $Y^{\an}$. Then, statement (2) implies that there exists a tropical polytope $P_{i+1}$ such that $(P_i \cup L_{i+1}) \subset \widehat{(P_i \cup L_{i+1})}_{Y} \subset \interior(P_{i+1}) \subset P_{i+1} \subset D$. Therefore, the proposition is proved. 
\end{proof}

\begin{cor} \label{cor:FS_convex_hull}
  Let $L \subset Y^{\an}$ be a compact subset. Define the Fubini-Study convex hull of $L$ by
  \[
   \widehat{L}_Y^{\FS} \coloneqq \{x \in Y^{\an}, \, u(x) \le \sup_{L} u \text{ for any } u \in \FS(Y^{\an}) \}.
  \]
  Then, $\widehat{L}_Y=\widehat{L}_Y^{\FS}$.
\end{cor}
\begin{proof}
  Let $x \in Y^{\an}$. Suppose that $x \notin \widehat{L}_Y^{\FS}$, then there exists  $u \in \FS(Y^{\an})$ and $\lambda \in \Q$ such that $u(x)> \lambda$ and $u < \lambda $ on $L$. By Lemma \ref{lem:tropical_polytopes_and_FS_functions}, $P \coloneqq \{u \le \lambda\}$ is a tropical polytope, and by construction, $L \subset \interior(P)$. Then,  since $\interior(P)$ is Runge by Lemma \ref{lem:tropical_polytopes_are_poly_convex}, we have $\widehat{L}_Y \subset \interior(P)$. This proves that $x \notin \widehat{L}_Y$ and hence $\widehat{L}_Y \subset \widehat{L}_Y^{\FS}$. Conversely, suppose that $x \notin \widehat{L}_Y$, then one can take an open neighborhood $V$ of $\widehat{L}_Y$ such that $x \notin V$. Proposition \ref{prop:exhausion_of_Runge_domains_by_tropical_polytopes} (2) implies that there exists a tropical polytope such that $\widehat{L}_Y \subset P \subset V$, and Lemma \ref{lem:tropical_polytopes_and_FS_functions} implies that there exists  $u \in \FS(Y^{\an})$ such that  $u=0$ on $V$ and $u(x)>0$. This proves that $x \notin \widehat{L}_Y^{\FS}$ and $\widehat{L}_Y^{\FS} \subset \widehat{L}_Y$. 
\end{proof}

\subsection{Rossi's local maximum principle} \label{subsec:Rossi_local_maximum_principle}
The following is the non-archimedean version of Rossi's local maximum principle. See \cites{Ros60, Ros06}. The proof below was inspired by the proof in the complex setting given in \cite{Ros06}. 
\begin{thm} \label{thm:rossi_local_maximum_principle}
  Let $Y$ be an affine variety over $K$. Let $L$ be a compact subset of $Y^{\an}$ and let $z \in \widehat{L}_Y \setminus L$. Let $V$ be a relatively compact neighborhood of $z$ in $Y^{\an}$ with $V \cap L =\emptyset$.  Denote by $\partial V$ the boundary of $V$ in $Y^{\an}$.  Then, for any $f \in K[Y]$, we have $|f(z)| \le \sup_{\widehat{L}_Y \cap \partial V} |f|$.
\end{thm}
\begin{proof}
  Suppose that there exists $f \in k[Y]$ such that $\log|f|(z)>t > \sup_{\widehat{L}_Y \cap \partial V} \log |f|$ for some $t \in \Q$. 
  Since $\log|f|$ is continuous, there exists an open neighborhood $U$ of $\widehat{L}_Y \cap \partial V$ such that $\log |f|<t$ on $U$. Since $\partial V \setminus U$ is closed and  $(\partial V \setminus U) \cap \widehat{L}_Y=\emptyset$, there exists an open neighborhood $U'$ of $\widehat{L}_Y$ such that $U' \cap (\partial V \setminus U)=\emptyset$. Appyling Proposition \ref{prop:exhausion_of_Runge_domains_by_tropical_polytopes} (2) to $(\widehat{L}_Y, U')$, we obtain a tropical polytope with $\widehat{L}_Y \subset P \subset U'$. By Lemma \ref{lem:tropical_polytopes_and_FS_functions}, there exists $u \in \FS(Y^{\an})$ such that $u=0$ on $P$ and $u>0$ on $Y^{\an} \setminus P$. In particular, $u>0$ on the compact set $\partial V \setminus U$. Thus, there exists $C>0$ such that $Cu>\log|f|-t$ on $\partial V \setminus U$. On the other hand, $Cu \ge 0 > \log|f|-t$ on $U \cap \partial V$. Therefore, we have $Cu> \log|f|-t$ on $\partial V$.  Since $Cu$ and $\log|f|$ are continuous,  this implies that $\overline{W} \subset V$, where $W \coloneqq \{x \in V, \, Cu(x) < \log|f|(x)-t\}$.

  Now, we define the function $v$ on $Y^{\an}$ by
  \[
  v=Cu \text{ on } Y^{\an} \setminus V, \qquad v=\max\{Cu, \log|f|-t\} \text{ on } V. 
  \]
  By construction, $v \ge 0$ on $Y^{\an}$. Since $Cu(z)=0$, $v(z)=\log|f|(z)-t$.  Since $L \subset Y^{\an} \setminus V$, we have $v=Cu=0$ on $L$. On the other hand, by Remark \ref{rmk:extension_of_psh_functions_by_sup}, we have $v \in \PSH \cap \PL(Y^{\an})$. Since $\overline{V}$ is compact in $Y^{\an}$, and $v =Cu$ on $Y^{\an} \setminus V$ with $Cu \in \FS(Y^{\an})$,  Lemma \ref{lem:psh_PL_functions_are_pshlogreg} (1) implies that there exists $v' \in \FS(Y^{\an})$ such that $\left\|v-v'\right\|_{C^0(Y^{\an})}<\frac{1}{3} (\log|f|(z)-t)$. Then, we have $v'(z)> \sup_{L}v'$. By Corollary \ref{cor:FS_convex_hull}, $z \notin \widehat{L}_{Y}^{FS}=\widehat{L}_Y$. This contradicts with the assumption, and the theorem is proved. 
\end{proof}

\subsection{Distance functions}
\begin{prop} \label{prop:Runge_domains_to_psh}
  Assume that $K$ admits a countable dense subfield. Let $Y$ be an affine variety over $K$. Denote $\widetilde{Y} \coloneqq Y \times \BA_{K}^{1}$. Let $p \colon  \widetilde{Y} \to Y$ be the projection to the first component, and $w \colon \widetilde{Y} \to \BA^1_K$ be the projection to the second component. We view $w$ as a element in $K[\widetilde{Y}]$. Let $D \subset \widetilde{Y}^{\an}$ be a Runge domain such that $(Y \times \{0\})^{\an} \subset D$.  Define the distance function on $Y^{\an}$ by
  \[
  d_D(y) \coloneqq \sup \left\{ s \in \R \,|\, p^{-1}(y)  \cap \{|w|<s \} \subset  D \right\} \in (0,+\infty].
  \]
  Then, $-\log d_D \in \PSH(Y^{\an})$. Moreover, suppose that there exists $v \in \FS(Y^{\an})$ such that $-\log d_D \le v+ O(1)$,   then $-\log d_D \in \PSHlogreg(Y^{\an})$.
\end{prop}

\begin{proof} [Proof of Proposition \ref{prop:Runge_domains_to_psh}]
  By Proposition \ref{prop:exhausion_of_Runge_domains_by_tropical_polytopes} (3), there exists a sequence of polytopes $P_1 \subset \cdots \subset  P_r \subset \cdots \subset D$ exhausting $D$. For each $r \ge 1$, denote the open subset $\Omega_r \coloneqq p^{\an} \left( \interior(P_r) \cap (Y \times \{0\})^{\an} \right) \subset Y^{\an}$. Since $(Y \times \{0\})^{\an} \subset D$, $p^{\an}|_{(Y \times \{0\})^{\an}}$ is an isomorphism, and $\bigcup_{r \ge 1} \interior(P_r)=D$, we have $\bigcup_{r \ge 1} \Omega_r=Y^{\an}$. 
  
  We claim that for each $r \ge 1$, $-\log d_{\interior(P_r)}   \in \TConv \cap \PL (\Omega_r)$. Granted this, we show how to finish the proof of the proposition. Let $E_1 \subset \cdots \subset E_r \subset \cdots \subset Y^{\an}$ be a sequence of tropical polytopes exhausting $Y^{\an}$ constructed in Example \ref{example:tropical_polytopes_in_affine_varieties}. Fix $r_1 \ge 1$. Since $E_{r_1}$ in compact $Y^{\an}$, there exist $r_2$ such that $E_{r_1} \subset \Omega_{r_2}$. Let $y \in \interior(E_{r_1})$. It's clear that $\left\{ d_{\interior(P_r)}(y) \right\}_{r \ge r_2}$ is a increasing sequence.  Note that  
  \[
    p^{-1}(y) \cap D=\bigcup_{r \ge r_2} \left( p^{-1}(y) \cap \interior(P_r) \right),
  \]
  For any $s < d_D(y)$, $p^{-1}(y) \cap \{|w| \le s\}$ is a compact subset of $p^{-1}(y) \cap D$. Therefore, there exists $r \ge r_2$ such that $p^{-1}(y) \cap \{|w| \le s\} \subset p^{-1}(y) \cap \interior(P_r)$, and hence $s  \le d_{\interior(P_r)}(y)$. This implies that 
  \[
    \lim_{r \to+\infty} d_{\interior(P_r)}(y)= d_D(y),
  \]
  and hence  $\left\{ -\log \left( d_{\interior(P_r)} \right) |_{\interior(E_{r_1})}   \right\}_{r \ge r_2} \subset \TConv \cap \PL (\interior(E_{r_1}))$ is a decreasing function sequence converging pointwise to $-\log d_D$ on $\interior(E_{r_1})$. Therefore, $-\log d_D \in \PSH(Y^{\an})$ by definition, and Lemma \ref{lem:approximable_on_exhausting_tropical_polytopes_implies_pshlogreg}  implies that if $-\log d_D \le v +O(1)$ for some $v \in \FS(Y^{\an})$, then $-\log d_D \in \PSHlogreg(Y^{\an})$. 

  Now, it remains to prove the claim. Fix $r \ge 1$ and suppose that $P_r = \bigcap_{1 \le j \le m} \{\log|g_j| \le c_j \}$, with $g_j \in K[\widetilde{Y}]$ and $c_j \in \Q$. Suppose that $g_j=f_{j,0}+f_{j,1}w+ \cdots f_{j,k_{j}}w^{k_j}$ with $f_{j,i} \in K[Y]$. Then, by definition, $\Omega_r =\bigcap_{1 \le j \le m} \{\log|f_{j,0}| < c_j \}$.  Let $y \in \Omega_r$, and let $\CH(y)$ be the completed residue field at $y$. Note that $p^{-1}(y)$ is isomorphic to the Berkovich analytification of $\Spec \CH(y)[w]$  over $\CH(y)$. For any $g \in K[\widetilde{Y}]=K[Y][w]$, denote by $\overline{g}$ the image of $g$ in $\CH(y)[w]$. Note that $\left(\log|g| \right)|_{p^{-1}(y)}=\log|\overline{g}|$. Therefore, we have 
  \[
  p^{-1}(y) \cap \interior(P_r)=\bigcap_{1 \le j \le m} \{\log|\overline{g_j}| < c_j \} \subset \left( \Spec \CH(y)[w] \right)^{\an},
  \]
  and 
  \[
    -\log d_{\interior(P_r)} = 
    \inf \left\{t \in \R \,|\,  \{-\log|w|>t\} \subset \bigcap_{1 \le j \le m} \{\log|\overline{g_j}| < c_j \}  \right\}.
  \]
  Then, the claim is implied by Lemma \ref{lem:tropicalization_one_variable_polynomial}. 
  \end{proof}

  \begin{lemma} \label{lem:tropicalization_one_variable_polynomial}
    Let $F$ be a complete NA field. For each $1 \le j \le m$, let $g_j=f_{j,0}+f_{j,1}w+ \cdots f_{j,k_{j}}w^{k_j} \in F[w]$ with $f_{j,i} \in F$.   Suppose that $\log |f_{j,0}| < c_j$ for each $1 \le j \le m$. Denote $P \coloneqq \bigcap_{1 \le j \le m} \{\log|g_j| \le c_j\} \subset (\BA_{F}^1)^{\an}$. Then, we have 
    \[
      \inf \left\{ t \in \R \,|\, \{ -\log|w| >t \} \subset \interior(P) \right\} =\max_{1 \le j \le m} \max_{1 \le i \le k_j} \left\{ \frac{1}{i} \left( \log|f_{j,i}|-c_j \right) \right\}.
    \]
  \end{lemma}
  \begin{proof}
    For each $a \in \BT^{1}$, denote by $v_{a} \in (\BA^1_{F})^{\an}$ the quasi-monomial semivaluation of weight $a$ with respect to the coordinate $w$. Then, $v_{a}(w)=-\log|w|(v_{a})=a$, and for each $1 \le j \le m$, we have
    \[
      -\log|g_j| (v_a) = \min_{1 \le i \le k_j} \{ ia -\log|f_{j,i}| \}.
    \]
    Denote by $\iota  \colon \BT^1 \hookrightarrow (\BA^1_F)^{\an}$ the inclusion map defined by $\iota(a)=v_{a}$. Since $\log|f_{j,0}|< c_j$, it's not hard to see that 
    \[
    \iota^{-1} \left( \{\log|g_j| < c_j\}  \right) = \left(  \max_{1 \le i \le k_j} \left\{ \frac{1}{i} \left( \log|f_{j,i}|-c_j \right) \right\}, +\infty \right].
    \]
    Thus, 
    \bens
    \iota^{-1} \left( \interior(P) \right)=\left(  \max_{1 \le j \le m} \max_{1 \le i \le k_j} \left\{ \frac{1}{i} \left( \log|f_{j,i}|-c_j \right) \right\}, +\infty \right].
    \eens
    On the other hand, note that for each $v \in (\BA_F^1)^{\an}$ with $a = -\log|w|(v) \in \iota^{-1} \left( \interior(P) \right)$, we have $-\log|g_j|(v) \ge -\log|g_j|(v_a) >-c_j$ for each $1 \le j \le m$, which implies that $v \in \interior(P)$. Therefore, the lemma is proved. 
\end{proof}

\subsection{Psh functions and polynomially convex subsets}

\begin{lemma} \label{lem:psh_to_Runge_domain_smooth_varieties}
  Let $Y$ be an affine variety over $K$. Let $u \in \PSHlogreg(Y^{\an})$. Then, 
  \begin{enumerate}
    \item $\{u <0\} \subset Y^{\an}$ is a Runge domain. 
    \item Suppose further that $u \in C^{0}(Y^{\an})$ and $\{u \le 0\}$ is  compact in $Y^{\an}$, then $\{u \le 0\}$ is a polynomially convex in $Y^{\an}$. 
  \end{enumerate}
\end{lemma}
\begin{proof}
  By definition, there exists a decreasing function sequence $\{u_i\}_{i \ge 1}$ converging pointwise to $u$. For each $i \ge 1$, $P_i \coloneqq \{u_i \le 0\}$ is a tropical polytope, and hence $\interior(P_i)$ is Runge by Lemma \ref{lem:tropical_polytopes_are_poly_convex}. Moreover, we have $\interior(P_i)=\{u_i <0\} \subset \interior(P_{i+1})=\{u_{i+1} <0\}$, and  $\{u <0\}=\bigcup_{i \ge 1} \interior(P_i)$. Then, Lemma \ref{lem:basic_properties_poly_convex_and_Runge} (2) implies that $\{u <0\}$ is Runge. This proves statement (1). For statement (2), take $x \in Y^{\an}$ such that $t \coloneqq u(x)>0$. Denote $L \coloneqq \{u \le 0\}$. By Example \ref{example:tropical_polytopes_in_affine_varieties}, there exists a compact subset $E$ such that $L \cup \{x\} \subset E \subset Y^{\an}$. Then, Dini's theorem implies that there exists $v \in \FS(Y^{\an})$ such that $\left\|v-u\right\|_{C^0(E)} \le \frac{t}{3}$. This implies further that  $v(x)> \sup_{L}v$. By Corollary \ref{cor:FS_convex_hull}, $x \notin \widehat{L}_Y^{\FS}=\widehat{L}_Y$. This proves $L=\widehat{L}_Y$. 
\end{proof}

\begin{prop} \label{prop:pshlogreg_runge_domain}
  Let $Y$ be an affine variety over $K$. Let $f \in K[Y]$. Let $Z \coloneqq \{f=0\} \subset Y$ be the hypersurface in $Y$ define by $f$, and let $Y_{f}=Y \setminus Z=\Spec K[Y]_f$ be the open affine subvariety of $Y$.  Let $L \subset (Y_f)^{\an}$ be a compact subset, and let $x \in \widehat{L}_Y \cap (Y_f)^{\an}$. Let $u \in \FS(Y^{\an})$ and let $v \in \PSHlogreg((Y_f)^{\an})$ such that $v \le u+O(1)$ on $(Y_f)^{\an}$. Then, $v(x) \le \sup_{L}v$.
\end{prop}

\begin{proof} [Proof of Proposition \ref{prop:pshlogreg_runge_domain}]
  Fix $x \in \widehat{L}_Y \cap (Y_f)^{\an}$. Since $L \cup \{x\} \subset (Y_f)^{\an}$ is compact and $\log|f|$ is a real-valued  continuous function on $(Y_f)^{\an}$, there exists $M_1>0$ such that $-M_1 \le \log|f| \le M_1$ on $L \cup \{x\}$. Fix $c \in \Q_{>0}$. Note that $u_c \coloneqq \max \{u, -c\log|f|\}=\max\{u, c\log|f^{-1}|\} \in \FS((Y_f)^{\an})$. We have the following lemma.
  
  \begin{lemma} \label{lem:control_Fubini-Study_functions_Y_f}
    For any $w \in \FS((Y_f)^{\an})$ such that $w \le u_c+O(1)$ on $(Y_f)^{\an}$, we have $w(x) \le \sup_{L} w+2cM_1$.
  \end{lemma}
  \begin{proof}
    Let $w=\max_{1 \le i \le k}\{c_i \log|g_i/f^{r_i}|+t_i, \lambda\}$, with $\lambda \in \Q$, $c_i \in \Q_{>0}$, $t_i \in \Q$ and   $g_i/f^{r_i} \in K[Y]_f$ for each $1 \le i \le k$. When $w(x)=\lambda$, by definition, we have $\sup_L w \ge \lambda$. When $w(x)>\lambda$, we may assume that $w(x)=c_1 \log|g_1/f^{r_1}|(x)+t_1$. Denote $w' \coloneqq c_1 \log|g_1/f^{r_1}|+t_1$.  It is enough to show that $w'(x) \le \sup_{L} w'+4cM_1$. 
    
    If $g_1/f^{r_1} \in K[Y]$, then $w(x) \le \sup_{L} w'$ due to the assumption that $x \in \widehat{L}_Y$. Otherwise, we may assume that $g_1 \in K[Y] \setminus I(f)$ and $r_1>0$. Here, $I(f)$ denote the pricipal ideal generated by $f$ in $K[Y]$.  We claim that $c_1r_1  \le c$. Granted this, we have 
    \bens
    w'(x) &=& c_1 \log|g_1|(x) +t_1 -c_1r_1\log|f|(x)    \\
    &\le& c_1 \log|g_1|(x) +t_1 +cM_1 \qquad (\text{since } \left| \log|f|(x) \right| \le M_1 )  \\
    &\le&  \sup_{L} \left( \log|g_1|+t_1 \right)+cM_1 \qquad (\text{since } x \in \widehat{L}_Y ) \\
    &\le& \sup_{L} w' +2cM_1 \qquad \qquad \qquad (\text{since } \sup_{L}\log|f| \le M_1),
    \eens
    and the lemma is proved.

    Now we prove the claim. 
    Since $g_1 \in K[Y] \setminus I(f)$, there exists $z \in Z^{\an}$ such that $\log|g_1|(z)=\epsilon >0$. Take $M_2 >0$ such that $z \in \{u < M_2\} \subset Y^{\an}$.   Since $Y^{\div}$ is dense in $Y^{\an}$, and $\log|g_1|$, $\log|f|$ are continuous on $Y^{\an}$, there exists a sequence $\{z_j\}_{j \ge 1} \subset Y^{\div} \cap \{u < M_2\} \subset (Y_f)^{\an}$ such that 
    \[
    \lim_{j \to _\infty} \log|g_1|(z_j)=\epsilon, \quad \text{and} \quad \lim_{j \to +\infty} \log|f|(z_j)=-\infty.
    \]
    Note that for each $j \ge 1$, we have
    \[
      w'(z_j)-u_c(z_j) \ge c_1 \log|g_1|(z_j)+t_1+ \min \left\{ \left(c-c_1r_1\right)\log|f|(z_j), -c_1r_1 \log|f|(z_j)-M_2 \right\}
    \]
    Suppose that $c_1r_1 > c$, then $\lim_{j \to +\infty}(w'(z_j)-u_c(z_j))=+\infty$. This contradicts with the assumption that $w' \le w \le u_c+O(1)$ on $(Y_f)^{\an}$. Therefore, the claim is proved. 
  \end{proof}
  
  Back to the proof of Proposition \ref{prop:pshlogreg_runge_domain}. If $v(x)=-\infty$, then there is nothing to prove. Thus, we may assume that $v(x)>-\infty$. Since $v \le u+O(1) \le u_c +O(1)$ on $Y^{\an}$, by Lemma \ref{lem:psh_regularization_with_growth_control} (2) and Dini's theorem, there exists  $w \in \FS((Y_f)^{\an})$ such that
  \[
    v \le w \le 2u_c +O(1) \text{ on } (Y_f)^{\an},
  \]
  \[
    v(x) \le w(x) \le v(x)+c, \quad \text{and} \quad  \sup_{L} v \le \sup_{L} w \le \sup_{L} v+c.
  \] 
  Now, Lemma \ref{lem:control_Fubini-Study_functions_Y_f} implies that $w(x) \le \sup_{L} w+4cM_1$. This implies further that $v(x) \le \sup_{L}|v|+(4M_1+1)c$. By letting $c \to 0$, the proposition is proved.
\end{proof}

\section{Weakly psh functions} \label{sec:weakly_psh_functions}
\begin{definition} \label{def:weakly_psh}
  Let $K$ be a complete NA field. Let $Y$ be an affine variety over $K$ and let $U \subset Y^{\an}$ be an open subset. 
  \begin{enumerate}
    \item A function $u \colon U \to \R \cup \{-\infty\}$ is called weakly psh on if there exists a birational proper morphism $\pi \colon \widetilde{Y} \to Y$ such that $\varphi \circ \pi^{\an} \in \PSH((\pi^{\an})^{-1}(U))$. We use $\WPSH(U)$ to denote the set of weakly psh functions on $U$.
    \item We say that $u \in \WPSH(Y^{\an})$ is of \emph{logarithmic growth} if there exists $v \in \FS(Y^{\an})$ such that $u \le v+O(1)$ on $Y^{\an}$. We use $\WPSHlog(Y^{\an})$ to denote the set of weakly psh function of logarithmic growth on $Y^{\an}$. 
  \end{enumerate}
\end{definition}

\begin{prop} \label{prop:psh_equals_weakly_psh_on_smooth_varieties}
  Let $Y$ be an affine variety over $K$. Suppose that $K$ is discretely valued or trivially valued. Let $U \subset Y^{\an}$ be an open subset.
  \begin{enumerate}
    \item Let $u \in \WPSH(U)$. Then, $u$ is usc on $U$. Moreover,  for any $y \in U$, we have 
    \[
    u(y)=\limsup_{x \to y,\, x \in U \cap Y^{\div}}u(x).
    \] 
    \item Suppose further that $K$ is of equicharacteristic $0$ and $Y$ is  smooth  over $K$. Then, we have $\WPSHlog(Y^{\an})=\PSHlogreg(Y^{\an})$. 
  \end{enumerate}
\end{prop}
\begin{proof}
  To see statement (1), let $\pi \colon \widetilde{Y} \to Y$ be a birational proper morphism such that $u \circ \pi^{\an} \in \PSH(\pi^{-1}(U))$. Since $u \circ \pi^{\an}$ is usc, $\{u \circ \pi^{\an} \ge C\} \subset (\pi^{\an})^{-1}(U)$ is a closed subset.  By \cite[Proposition 3.4.7]{Ber90}, $\pi^{\an}$ is a proper map. Since $(\pi^{\an})^{-1}(U)$ is locally compact Hausdorff, $\pi^{\an}$ is a closed map. Therefore, $\{u \ge C\}=\pi^{\an} \left( \{u \circ \pi^{\an} \ge C\} \right) \subset U$ is a closed subset. This implies that $u$ is usc on $U$. 
  Now fix $y \in U$, and take $\widetilde{y} \in (\pi^{an})^{-1}(y)$. By Proposition \ref{prop:local_comparison_divisorial_points_to_full_spaces} (1), there exists a sequence $\{x_i\}_{i \ge 1} \subset U \cap \widetilde{Y}^{\div}$ such that $\lim_{i \to +\infty} (u \circ \pi^{\an}) (\widetilde{x_i})=u \circ \pi^{\an}(\widetilde{y})$. Then, $\{x_i \coloneqq \pi^{\an}(\widetilde{x_i})\}_{i \ge 1} \subset U \cap Y^{\div}$.  Since $\pi^{\an}$ is continuous, we have $x_i \to y$ in $U$, and $\lim_{i \to +\infty} u(x_i)=u(y)$. This proves statement (1). 
  
  Now we prove statement (2). Fix $u \in \WPSHlog(Y^{\an})$. By definition, there exists a closed subvariety $Z \subset Y$   such that $Y \setminus Z$ is an Zariski dense affine open subvariety of $Y$, and  $u|_{(Y \setminus Z)^{\an}} \in \PSH((Y \setminus Z)^{\an})$. Then, Proposition \ref{prop:pshlog_equals_pshlogreg_on_smooth_affine_varieties} implies that $u|_{(Y \setminus Z)^{\an}} \in \PSHlogreg((Y \setminus Z)^{\an})$, and  Theorem \ref{thm:pshlogreg_extension} implies that  $u|_{(Y \setminus Z)^{\an}}$ can be extended uniquely to $u' \in \PSHlogreg(Y^{\an})$.  Combining Proposition \ref{prop:local_comparison_divisorial_points_to_full_spaces} (1) with statement (1), we infer that $u'=u$ on $Y^{\an}$. This proves statement (2). 
\end{proof}

The main theorem of this section is the following NA analogue of Fornaess-Narasimhan theorem. 

\begin{thm} \label{thm:NA_FN_theorem}
  Assume that $K$ is a complete discretely valued or trivially valued NA field of equicharacteristic $0$ such that $K$ admits a countable dense subfield.  Let $u \in \WPSHlog(Y^{\an})$. Then, $\{y \in Y, \, u(y)<0 \}$ is a Runge open subset in $Y$, and $u \in \PSHlogreg(Y^{\an})$.
\end{thm}

Now we introduce weakly $\omega$-psh functions and a global version of Theorem \ref{thm:NA_FN_theorem}.
\begin{definition}
  Let $X$ be a projective variety over $K$ and $\omega \in \nef(X)$. Then, a function $\varphi \colon X^{\an} \to \R \cup \{-\infty\}$ is called weakly $\omega$-psh if for any $x \in X^{\an}$, there exist an open neighborhood $U$ of $x$ and a potential $g$ of $\omega$ over $U$ such that $g+\varphi \in \WPSH(U)$.
\end{definition}

\begin{thm} \label{thm:global_NA_FN_theorem}
  Assume that $K$ is a complete discretely valued or trivially valued NA field of equicharacteristic $0$ such that $K$ admits a countable dense subfield. Let $X$ be a projective variety over $K$ and $\omega \in \nef(X)$ with $c_1(L(\omega)) \in \Amp(X)$. Then, $\WPSHo(X^{\an})=\GPSHo(X^{\an})$.
\end{thm}
\begin{proof}
  Since $c_{1}(\omega) \in \Amp(X)$, we may assume that there exists a closed embedding $i \colon X \hookrightarrow \P^k$ such that $\omega$ is the pullback of the trivial model of $(\P^k, \CO(1))$ via $i$. Let $[z_0: z_1: \cdots : z_k]$ denote the homogeneous coordinate of $\P^k$. For each $0 \le i \le k$, denote the affine chart $Y_i \coloneqq X \cap \{z_i \neq 0\}$. Then, 
  \[
  \rho_i \coloneqq   \max\left\{ \log |z_0|, \log|z_1|, \cdots, \log|z_k| \right\}  -\log|z_i|
  \]
  is a potential of $\omega$ on $Y_i^{\an}$. 
  
  For each $0 \le i \le k$,  by definition, $\varphi+\rho_i \in \WPSHlog(Y_i^{\an})$. Hence, by Theorem \ref{thm:NA_FN_theorem}, $\varphi+\rho_i \in \PSHlogreg(Y_i^{\an})$. Fix $c>1$. Define $\psi_i \colon X^{\an} \to \R \cup \{-\infty\}$ by $\psi_i = \varphi-(c-1)\rho_i$  on $Y_i^{\an}$ and $\psi_i =-\infty$ on $X^{\an} \setminus Y_i^{\an}$.
  Then, Lemma \ref{lem:psh_regularization_with_growth_control} (1) implies that $\psi_i \in \GPSH_{c\omega}(X)$. Therefore,  $\psi \coloneqq \max_{0 \le i \le k}\{\psi_i\} \in \GPSH_{c\omega}(X)$.  On the other hand, note that 
  \[
  \max_{0 \le i \le k}\{-\rho_i\}=\max_{0 \le i \le k} \left\{ \log|z_i|-   \max\left\{ \log|z_0|, \log|z_1|, \cdots, \log|z_k| \right\}  \right\}=0. 
  \]
  This implies that $\psi=\varphi \in \GPSH_{c \omega}(X)$. The proof is completed by letting $c \to 1$. 
\end{proof}

In the rest of Section \ref{sec:weakly_psh_functions}, \emph{except} Section \ref{subsubsec:norms_under_finite_flat_morphisms}, we assume that $K$ is a complete discretely valued or trivially valued NA field of equicharacteristic $0$ such that $K$ admits a countable dense subfield.

\subsection{Strictly psh functions} \label{subsec:strictly_psh_functions}
\subsubsection{Strictly convex functions on $\BT^k$} 
We first consider strictly convex function on $\BT^k$. Let $(t_1, \cdots, t_k)$ be an coordinate of $\BT^k$. 

\begin{lemma} \label{lem:smooth_strictly_convex_functions}
  There exists $\chi_k \in \Conv(\BT^k)$ satisfying the following properties. 
  \begin{enumerate}
    \item $\max\{ 0, -t_1, \cdots, -t_k \} \le \chi_k \le \max\{ 0, -t_1, \cdots, -t_k \}+1$ on $\BT^k$.
    \item For each $\mathbf{t} \in \BT^k$, there exists an affine function $l=c_1t_1 +\cdots c_k t_k +C$ with  $c_1, \cdots, c_k \in \R_{\le 0}$ and $C \in \R$ such that, 
    \[
    \chi_k(\mathbf{t})=l(\mathbf{t}), \quad \text{and} \quad  \chi_k(\mathbf{t}')> l(\mathbf{t}) \ \text{for any }  \mathbf{t}' \in \BT^k \setminus \{\mathbf{t}\}.
    \]
    Moreover, for any open neighborhood $V$ of $\ft$ in $\BT^k$, there exists $\delta>0$ such that $l< \chi_k-\delta$ on $\BT^k \setminus V$. 
  \end{enumerate}
\end{lemma}
\begin{proof}
  Denote $f \coloneqq \max\{ 0, -t_1, \cdots, -t_k \} \in \Conv(\BT^k)$. Following \cite[Lemma 2.7.7]{BPS14}, $f_{\epsilon} \coloneqq f \star \rho_{\epsilon}$ is a smooth strictly convex function on $\R^k$ with $\left\|f_{\epsilon}-f\right\|_{C^0(\R^k)} \le \frac{1}{3}$ for $0< \epsilon \ll 1$. Here $\rho_{\epsilon} = \frac{\epsilon^k}{(2\pi)^{k/2}} \exp(-\frac{\epsilon^2 \left\|\ft\right\|}{2})$ is the Gaussian function on $\R^k$, and $\star$ denotes the convolution. Moreover, \cite[Proposition 5.7]{APW2} shows that $f_{\epsilon}+\frac{1}{2}$ can be extended to a strictly convex function $\chi_k$ on $\BT^k$. This proves property (1) and the first half of property (2). 

  To see the last part of the lemma, we may assume that $\overline{V}$ is compact in $\BT^k$. For each $\delta>0$, set $\CC(\delta) \coloneqq \{\chi_k \le l+\delta\} \subset \BT^k$. Then, $\CC(\delta) \cap \overline{V}$ is compact, and we have  $\bigcap_{\delta >0} \CC(\delta)=\{\ft\}$.  Take an open neighborhood $W$ of $\ft$ such that $W \Subset V$. Then, \cite[Lemma 5.11]{APW2} implies that there exists $\delta>0$ such that $\CC(\delta) \cap \overline{V} \subset W$. On the other hand, it's not hard to see that $\CC(\delta)$ is connected. This implies that $\CC(\delta) \subset W \Subset V$, and the lemma is proved. 
\end{proof}

\subsubsection{Construction of a strictly psh function}
Let $K_1$ be a countable dense subfield of $K$ such that $Y$ is defined over $K_1$. Note that $K_1[Y]$ is a ring of countable cardinal.  Let $z_1, \cdots, z_k $ be generators of $K_1[Y]$,  and let $\{f_i\}_{i \ge 1}$ be an enumeration of all elements in $K_1[Y]$. For each $j \ge 1$, denote the tropicalization map
\bens
  \Trop_j \colon Y^{\an} &\longrightarrow& \BT^{k+j} \\
  x &\longmapsto& (-\log|z_1|(x), \cdots, -\log|z_k|(x), -\log|f_1|(x), \cdots, -\log|f_j|(x))
\eens
Let $j \ge 1$ and $\chi_{k+j} \in \Conv(\BT^{k+j})$ constructed in Lemma \ref{lem:smooth_strictly_convex_functions}. Then, Lemma \ref{lem:property_tropical_conv_functions} (3) and Lemma \ref{lem:approximable_on_exhausting_tropical_polytopes_implies_pshlogreg}   implies that 
\[
\rho_j \coloneqq (\chi_{k+j} \circ \Trop_j) \in C^0 \cap \PSHlogreg(Y^{\an}).
\] 

Denote $v \coloneqq \max\{\log|z_1|, \cdots \log|z_k|, 0\}+1  \in \FS(Y^{\an})$.
Then, by Lemma \ref{lem:FS_functions_are_comparable} (1), for each $j \ge 1$, there exists $0 <\gamma_j < \frac{1}{2^j}$ such that 
\[
  0 \le \gamma_j \max_{1 \le i \le k, \, 1 \le i' \le j}\{\log|z_i|,  \log|f_{i'}|, 0\} \le \frac{1}{2^j} v. 
\]
Then, we have $0 \le \gamma_j \rho_j \le \frac{1}{2^j} (v+1)$. For each $r \ge 1$, denote the compact subset $E_r \coloneqq \{v \le r\} \subset Y^{\an}$. Then $\sum_{i=1}^{+\infty} \gamma_i \rho_i$  converges uniformly to some function $\rho$ on the compact subset $E_r$ for each $r \ge 1$, and we have $\rho \le v+1$. Therefore, by Lemma \ref{lem:approximable_on_exhausting_tropical_polytopes_implies_pshlogreg},
\[
  \rho \coloneqq \sum_{i=1}^{+\infty} \gamma_i \rho_i \in C^0 \cap \PSHlogreg(Y^{\an}).
\]

\begin{lemma} \label{lem:strictly_psh_functions}
  Let $x \in Y^{\an}$ and $U$ be an open neighborhood of $x$. Then, there exists $\rho' \in \PSHlogreg(Y^{\an})$ satisfying 
  \begin{enumerate}
    \item $\rho' \le \rho$ on $Y^{\an}$ and $\rho'(x)=\rho(x)$.
    \item There exists $\delta>0$ such that $\rho'< \rho-\delta$ on $Y^{\an} \setminus U$.
  \end{enumerate}
\end{lemma}
\begin{proof}
  We may assume that there exists $j \ge 1$ and an open neighborhood $V$ of $\mathbf{t} \coloneqq \Trop_j (x)$ in $\BT^{k+j}$ such that $U=\Trop_j^{-1}(V)$ (see the end of Section \ref{subsubsec:tropicalizations}). Let $l$ be an affine support function of $\chi_{k+j}$ at $\mathbf{t}$ constructed by Lemma \ref{lem:smooth_strictly_convex_functions}. Therefore, we have $l=\sum_{1 \le i \le k+j}c_it_i + C$ with $c_1, \cdots c_k \in \R_{\le 0}, C \in \R$, $\chi_{k+j} \ge l$ on $\BT^{k+j}$, $\chi_{k+j}(\mathrm{t})=l(\mathrm{t})$ and $l< \chi_{k+j}-\delta/\gamma_j$ on $\BT^{k+j} \setminus V$ with $\delta>0$.   Define 
  \[
  \rho' \coloneqq \gamma_j \left( l \circ \Trop_{j} \right) + \sum_{i \ge 1, \, i \neq j} \gamma_i \rho_i.
  \]
  Note that  
  \[
  l \circ \Trop_{j}= \left( \sum_{1 \le i \le k} \left(-c_i \log|z_i| \right) + \sum_{1 \le i' \le j}  \left( -c_{i'} \log|f_{i'}| \right) +C \right) \in \PSHlogreg(Y^{\an}).
  \]
  Therefore, we have $\rho' \in \PSHlogreg(Y^{\an})$. By constructions, 
  \[
  \rho-\rho'=\gamma_j \left( \rho_j-l \circ \Trop_j \right)=\gamma_j \left((\chi_{k+j}-l) \circ \Trop_j \right).
  \]
  Therefore, we have $\rho-\rho' \ge 0$ on $Y^{\an}$, $\rho(x)-\rho'(x)=0$, and $\rho-\rho'>\delta$ on $\Trop^{-1} (\BT^{k+j} \setminus V)=Y^{\an} \setminus U$.
\end{proof}

\subsection{}

For any $u \in \WPSHlog(Y^{\an})$ and $c >0$, denote $u_c \coloneqq u+c\rho \in \WPSHlog(Y^{\an})$.

\begin{lemma} \label{lem:aux_lemma_1}
  Let $u \in \WPSHlog(Y^{\an})$ and $c >0$. Let $Z$ be a closed variety of $Y$. Suppose that there exists  an  affine open cover $\{U_{\alpha}\}_{\alpha \in \CI}$ of $Y \setminus Z$ such that $(u_c)|_{U_{\alpha}} \in \PSHlogreg((U_{\alpha})^{\an})$ for each $\alpha \in \CI$. Then, for any compact subset $L \subset Y^{\an}$ and any $x \in \widehat{L}_Y \setminus Z^{\an}$, we have $u_c(x) \le \sup_{L} u_c$. 
\end{lemma}
\begin{proof}
  We prove this lemma by contradiction. Suppose that there exists $x \in \widehat{L}_Y \setminus Z^{\an}$ such that $u_c(x)>\sup_{L}u_c$. Let $f_1, \cdots, f_k \in K[Y]$ such that $Z=\{f_1=f_2= \cdots =f_k=0\}$, and set $v \coloneqq \log(\sum_{1 \le i \le k} |f_i|)$. By Example \ref{example:psh_function_defining_closed_subvariety}, $v \in \PSHlogreg(Y^{\an})$ and $\{v=-\infty\}=Z^{\an}$. In particular, there exists $M >0$ such that $\sup_{L} v \le M$ and $|v(x)| \le M$. Therefore, one can choose $\epsilon >0$ such that $(u_c+\epsilon v)(x) > \sup_{L}(u_c+\epsilon v)$. 
  
  Take $y \in \widehat{L}_Y$ such that $(u_c+\epsilon v)(y)=\sup_{\widehat{L}_Y}(u_c+\epsilon v)$. By construction, we have $\sup_{\widehat{L}_Y}(u_c+\epsilon v) > \sup_{L}(u_c+\epsilon v)$ and $(u_c+\epsilon v)|_{Z^{\an}}=-\infty$. Therefore, $y \in \widehat{L}_Y \setminus (Z^{\an} \cup L)$. 
  By assumption, there exists $y \in U_{\alpha}$ for some $\alpha \in \CI$. We may assume that $U_{\alpha}=Y_g=Y \setminus \{g=0\}$ for some $g \in K[Y]$.  Since $Y^{\an}$ is locally compact Hausdorff, there exists  a relatively compact open neighborhood $W$ of $y$ in $Y^{\an}$ such that $\overline{W} \subset U_{\alpha}$ and $\overline{W} \cap  L= \emptyset$.

  Consider the compact set  $\CC  \coloneqq \partial W \cap \widehat{L}_{Y} \subset U_{\alpha}$. 
  Applying Lemma \ref{lem:strictly_psh_functions} to $(W, y)$, we obtain $\rho' \in \PSHlogreg(Y^{\an})$ and $\delta >0$ such that $\rho'(y)=\rho(y)$ and $\rho' \le \rho-\delta$ on $\partial W$. Denote $w \coloneqq u+c\rho'+\epsilon v \in \WPSHlog(Y^{\an})$. Then, we have 
  \[
  \sup_{\CC} w \le  \sup_{\CC} (u_c+\epsilon v) -c\delta \le \sup_{\widehat{L}_Y} (u_c+\epsilon v) -c\delta= (u_c+\epsilon v)(y)-c\delta =w(y)-c\delta.
  \] 

  On the other hand,   by assumption, we have $u|_{U_{\alpha}} \in \PSHlogreg(U_{\alpha})$, and hence $w|_{U_{\alpha}} \in \PSHlogreg(U_{\alpha})$. By Rossi's local maximum principle (Theorem \ref{thm:rossi_local_maximum_principle}), we have $|f(y)| \le \sup_{\CC} |f|$ for any $f \in K[Y]$. This implies that $y \in \widehat{\CC}_Y$. Then, Proposition \ref{prop:pshlogreg_runge_domain} implies that $w(y) \le \sup_{\CC} w $. This induces contradictions, and hence the lemma is proved. 
\end{proof}

\begin{cor} \label{cor:pshlogreg_on_an_affine_open_cover_induces_pshlogreg}
  Let $u \in \WPSHlog(Y^{\an})$. Suppose that there exists an affine open cover $\{U_{\alpha}\}_{\alpha \in \CI}$ of $Y$ such that $u|_{U_{\alpha}} \in \PSHlogreg((U_{\alpha})^{\an})$ for each $\alpha \in \CI$. Then, $\{x \in Y^{\an} \, |\,  u(x)<0 \}$ is a Runge open subset in $Y^{\an}$, and $u \in \PSHlogreg(Y^{\an})$.
\end{cor}
\begin{proof}
  Applying Lemma \ref{lem:aux_lemma_1} to $Z=\emptyset$, we obtain that $\{u_c < 0\}$ is Runge in $Y^{\an}$ for any $c >0$. Note that $\left\{ u_c \right\}_{c \searrow 0}$ is a decreasing function sequence converging pointwise to $u$. Thus, by Lemma \ref{lem:basic_properties_poly_convex_and_Runge} (2), $\{u <0\}$ is Runge in $Y^{\an}$. This proves the first part of the statement. For the second part, let $p \colon  \widetilde{Y} \to Y$ be the projection to the first component, and let $w \colon \widetilde{Y} \to \BA^1_K$ be the projection to the second component. We view $w$ as a element in $K[\widetilde{Y}]$. By definition $\log|w|+u \circ p^{\an} \in \WPSHlog(\widetilde{Y}^{\an})$. Note that $\{\widetilde{U_{\alpha}} \coloneqq U_{\alpha} \times \BA^1_K\}_{\alpha \in \CI}$ is an affine open cover of $\widetilde{Y}$, and $\left(\log|w|+u \circ p^{\an} \right)|_{\widetilde{U_{\alpha}}} \in \PSHlogreg((\widetilde{U_{\alpha}})^{\an})$ for each $\alpha  \in \CI$. Then, applying the first part of the statement to $\widetilde{Y}$, we obtain  that  
  \[
  D=\{\log|w|+u \circ p^{\an}<0\} \subset \widetilde{Y}^{\an}
  \]
  is Runge. Note that $(Y \times \{0\})^{\an} \subset D$, and $-\log d_D =u$. Then, Proposition \ref{prop:Runge_domains_to_psh} implies that $u \in \PSHlogreg(Y^{\an})$. 
\end{proof}

\subsection{Norms of weakly psh functions}  \label{subsec:norms_of_weakly_psh_functions}

\subsubsection{Norms under finite flat morphisms} \label{subsubsec:norms_under_finite_flat_morphisms}

In this part (Section \ref{subsubsec:norms_under_finite_flat_morphisms}), we assume that $K$ is \emph{any} complete NA field. Let $Y, S$ be affine varieties over $K$, and let $\pi \colon Y \to S$ be a finite flat surjective morphism. Assume that $S$ is irreducible, and $Y$ is equidimensional. In this case, $\pi|_{Y_{\alpha}}$ is a finite surjective morphism for each irreducible component $Y_{\alpha}$ of $Y$. We define the degree of $\pi$ to be $\deg(\pi) \coloneqq \sum_{\alpha} \deg(K(Y_{\alpha})/K(S))$.

For any $s \in S^{\an}$, denote by $\CH_{S}(s)$ the completed residue field at $s$. Then, $A_{s, \pi} \coloneqq \CH_{S}(s) \times_{S} Y$ is a finite algebra over $\CH_{S}(s)$. Moreover, we have a product decomposition   $A_{s, \pi}=\prod_{y \in \pi^{-1}(s)} A_{y, \pi}$ into local finite algebra over $\CH_{S}(s)$, and residue field of the local algebra $A_{y, \pi}$ is isomorphic to the completed residue field $\CH_Y(y)$. Then, we define the multiplicity of $\pi$ at $y$ by
\be   \label{equation:def_multiplicity_flat_locus}
\mult_{y}(\pi) \coloneqq \dim_{\CH_{S}(s)}  A_{y, \pi}.
\ee

\begin{lemma} \label{lem:sum_of_multiplicity_equals_to_degree_flat_morphism}
  Under the above assumptions, we have $\deg(\pi)=\sum_{y \in \pi^{-1}(s)}  \mult_{y}(\pi)$ for any $s \in S^{\an}$.
\end{lemma}
\begin{proof}
  This follows easily from the fact that $\CO_Y$ is a locally free coherent sheaf over $\CO_S$.
\end{proof}

For any function $f \colon Y^{\an} \to \R \cup \{-\infty\}$, define the norm of $f$ under $\pi$ by
\be \label{equation:def_norm_flat_locus}
 N_{\pi}(f) (s) \coloneqq \sum_{y \in \pi^{-1}(s)}  \mult_{y}(\pi) f(y). 
\ee

\begin{prop} \label{prop:norm_finite_flat_morphisms}
  Let $K$ be any complete NA field. Let $\pi \colon Y \to S$ be a finite flat surjective morphism between affine varieties over $K$. Assume that $S$ is irreducible, and $Y$ is equidimensional. Let $u \in \PSHlogreg(Y^{\an})$. Then, 
  \begin{enumerate}
    \item $N_{\pi}(u) \in \PSHlogreg(S^{\an})$. 
    \item Moreover, if $u \in C^0 \cap \PSHlogreg (Y^{\an})$, then $N_{\pi}(u) \in C^0 \cap \PSHlogreg(S^{\an})$.
    \item Let \(\BK/K\) be an extension of complete non-Archimedean fields. Then we have 
    \[
    N_{\pi}(u) \circ r_{\BK/K}=N_{\pi_{\BK}}(u \circ r_{\BK/K}),
    \]
    where \(\pi_{\BK}\) denotes the base change of \(\pi\) to \(\BK\), and \(r_{\BK/K}\) denote the restriction maps on Berkovich spaces induced by the base change $\BK/K$.
  \end{enumerate}
\end{prop}
\begin{proof}
  Let $S \hookrightarrow \BA^{k}_K$ and $Y \hookrightarrow \BA^{k}_K \times \BA^{l}_K$ be closed embeddings such that $\pi=\fp_1|_{Y}$, where $\fp_1 \colon \BA^{k}_K \times \BA^{l}_K \to \BA^{k}_K$ be the projection morphism to the first component. Let $T'$ be the closure of $S$ in $\P^{k}_K$, and let $X'$ be the closure of $Y$ in $\P^{k}_K \times \P^{l}_K$.  Denote by $\tau'=(\fq_1)|_{X'} \colon X' \to T'$, where $\fq_1 \colon \P^{k}_K \times \P^{l}_K \to \P^{k}_K$ is the projection. Note that $\tau'|_Y=\pi$. On the other hand,  since $\pi$ is proper, $Y$ is a closed subvariety of $\BA^{k}_K \times \P^{l}_K$. This implies that $Y=X' \cap (\BA^{k}_K \times \P^{l}_K)$. Equivalently, we have $(\tau')^{-1}(S)=Y$. 
  
  \lnote{Let $X$ be the compactification of $Y$. Then, by definition, $Y \times_S (S \times X) \to X \times S$ is a closed map. The map $Y \to X \times S$ induces a closed embeddinng $Y \to Y \times_S (S \times X)$. }

  By the Raynaud-Gruson flattening theorem (\cite[Théorème 5.2.2, Première partie]{RG71}), there exists a closed subscheme $Z \subset T' \setminus S$ such that $\tau \colon X \to T$ is a flat morphism, where $T \coloneqq \mathrm{Bl}_{Z}T'$, and $X$ is the proper transform.  Flatness together with properness implies that the fiber dimension is locally constant, and hence $\tau$ is quasi-finite. These imply further that $\tau$ is also finite by \cite[Lemma 37.44.1]{stacks-project}. Moreover, by construction, we have $\tau|_Y=\pi$ and $\tau^{-1}(S)=Y$.  

  Let $D$ be an effective ample divisor on $T'$ with $\mathrm{Supp}(D)=T' \setminus S$. By Lemma \ref{lem:blowup_pi_ample_divisor}, there exists an effective ample divisor $H$ on $T$ with $\mathrm{Supp}(H)=T \setminus S$. Take $m \in \N$ such that $L \coloneqq \CO(mH)$ is very ample on $T$ with $s_{mH} \in H^0(T, L)$. Let $\theta \in \amp(T)$ such that it induces  a Fubini-Study metric $|\cdot|_{\theta}$ on $L$.  Then, by Proposition \ref{prop:potential_of_positive_form_is_Fubini-Study} (1), $g \coloneqq -\log|s_{mH}|_{\theta} \in \FS(S^{\an})$. For each $r \ge 1$, set $P_r \coloneqq \{g \le r\}$. Then, Lemma \ref{lem:tropical_polytopes_and_FS_functions} implies that  $P_1 \subset \cdots \subset P_r \subset \cdots \subset S^{\an}$ is a sequence of tropical polytopes exhausting $S^{\an}$. On the other hand, since $\tau$ is finite, $\tau^{*}L$ is an ample line bundle on $X$, and $\tau^{*}\theta \in \nef(X)$. Then, we have $Z(\tau^{*}s_{nH})=X \setminus Y$, and $g  \circ \pi^{\an}=-\log|\tau^{*}s_{mH}|_{\tau^{*}\theta} \in \PSH \cap \PL(Y^{\an})$ is a potential of $\tau^{*} \theta$ on $Y^{\an}$. 

  \smallskip

  Fix $u \in \PSHlogreg(Y^{\an})$ and a decreasing function sequence $\{u_i\}_{i \ge 1} \subset \FS(Y^{\an})$ converging pointwise to $u$ on $Y^{\an}$. By definition, $\{N_{\pi}(u_i)\}_{i \ge 1}$ is a decreasing function sequence converging pointwise to $N_{\pi}(u)$. By Lemma \ref{lem:FS_functions_are_comparable}, we may assume that  $u_1 \le g \circ \pi^{\an}$. By Lemma \ref{lem:sum_of_multiplicity_equals_to_degree_flat_morphism}, $N_{\pi}(g \circ \pi^{\an})=\deg(\pi)\cdot g$. This implies that for each $i \ge 1$, we have $N_{\pi}(u_i) \le N_{\pi}(u_1) \le N_{\pi}(g \circ \pi^{\an}) \le \deg(\pi)\cdot g \in \FS(Y^{\an})$. Thus, by Lemma \ref{lem:psh_PL_functions_are_pshlogreg} (3), to show that $N_{\pi}(u) \in \PSHlogreg(Y^{\an})$,  it suffices to show that for each $i, r \ge 1$, we have $N_{\pi}(u_i)|_{\interior(P_r)} \in \PSH \cap \PL(\interior(P_r))$.
  
  Fix $i,r \ge 1$. Fix $c \in \Q$ such that $c>1$. After adding a constant to $u_i$, we may assume that $u_i \ge 0$ on $Y^{\an}$. To simplify the notations, denote $\omega \coloneqq c\tau^{*}\theta \in \nef(X)$ with $L(\omega) \simeq M \coloneqq c \tau^{*}L \in \Pic_{\Q}(X)$. Define the function $\varphi_i \colon X^{\an} \to \R \cup \{-\infty\}$ by
  \[
  \varphi_i =\max\left\{u_i-c(g \circ \pi^{\an}), \, -cr \right\} \text{ on } Y^{\an}, \qquad \text{and} \qquad \varphi_i=-cr \text{ on } X^{\an} \setminus Y^{\an}.
  \]
  By Lemma \ref{lem:finite_sup_lpsh}, we have $\varphi_i +c(g \circ \pi^{\an}) \in \PSH \cap \PL(Y^{\an})$.  On the other hand, we have $\{u_i-c(g \circ \pi^{\an}) \ge -cr\} \subset \{(c-1)(g \circ \pi^{\an}) \le cr\} \subset (\pi^{\an})^{-1}(P_{r'})$ with $r' =\frac{cr}{r-1}$. Note that $(\pi^{\an})^{-1}(P_{r'})$ is a compact subset of $Y^{\an}$, and hence a closed subset of $X^{\an}$. This implies that $\varphi_i=-cr$ on an open neighborhood of $(X \setminus Y)^{\an}$, and hence   $\varphi_i \in \LPSH_{\omega} \cap \PL(X^{\an})$.  By Theorem \ref{thm:GPSH_PL_equals_to_LPSH_PL}, $\varphi_i \in \GPSH_{\omega} \cap \PL(X^{\an})$. Therefore, $\phi_i \coloneqq |\cdot|_{\omega_{\varphi_i}}=|\cdot|_{\omega} \cdot e^{-\varphi_i}$ is a psh model metric on $M$. 

  Denote by $N_{\tau}(M) \in \Pic_{\Q}(T)$ the norm of the $\Q$-line bundle $M$ under the finite flat morphism $\tau$ (see \cite[Section 6.5]{EGAII}). Recall that in \cite[Section 8.5]{BE21}, the author define the norm metric $N_{\tau}(\phi_i)$ of $\phi_i$ on $N_{\tau}(M)$ satisfying  
  \[
  \log|N_{\tau}(s)|_{N_{\tau}(\phi_i)}= \sum_{y \in \pi^{-1}(s)} \mult_{y}(\pi) \, \log|s|_{\phi_i} \qquad \text{for any } s\in H^0(X, M).
  \]
  \lnote{actually I don't know whether the condition of flatness is optimal. It seems that we can define the norm for line bundles when the target is normal. (on \cite{stacks-project}) However, we are not sure in this case whether the norm will still be given by this multiplicity.}
  Moreover, \cite[Proposition 8.22]{BE21} implies that $N_{\tau}(\phi_i)$ is a psh model metric on $N_{\tau}(L)$. \footnote{More precisely, there exists a model $(\CX, \CM)$ representing $\phi_i$, a model $\CT$ of $T$, and a proper flat morphism $\CX \to \CT$ extending $\tau$ such that $(\CT, N_{\tau}(\CM))$ represents $N_{\tau}(\phi_i)$.}  Combining the above facts and applying Corollary \ref{cor:GPSH_is_LPSH}, we obtain that $N_{\pi}\left(\varphi_i+c(g \circ \pi^{\an}) \right)=\log|N_{\tau}(\tau^{*}s_{rH})|_{N_{\tau}(\phi_i)} \in \PSH \cap \PL(S^{\an})$. On the other hand, since $u_i \ge 0$ on $Y^{\an}$,   we have $\varphi_i=u_i-c(g \circ \pi^{\an})$ on $(\pi^{\an})^{-1}(P_r)$, which implies that $N_{\pi}\left(\varphi_i+c(g \circ \pi^{\an})\right)=N_{\pi}(u_i)$ on $P_r$. Thus, we obtain that  $N_{\pi}(u_i)|_{\interior(P_r)} \in \PSH \cap \PL(\interior(P_r))$, and statement (1) is proved.   

  \smallskip

  For statement (2), for each $r \ge 1$, $\{u_i\}_{i \ge 1}$ converges uniformly to $u$ on $E_r$ by Dini's theorem. Then,  Lemma \ref{lem:sum_of_multiplicity_equals_to_degree_flat_morphism} implies that $\left\|N_{\pi}(u_i)-N_{\pi}(u_j)\right\|_{C^0(P_r)} \le \deg(\pi) \,  \left\|u_i-u_j\right\|_{C^0(E_r)}$.  Therefore, $\{N_{\pi}(u_i)\}_{i \ge 1}$ converges uniformly to $N_{\pi}(u)$ on $P_r$, and statement (2) follows.  Statement (3) follows from the fact that the norm functor for line bundles, and thereby the norm functor for model metrics commutes with base changes (see \cite[Theorem A.21]{BE21}).
\end{proof}

\subsubsection{Norms under finite surjective morphisms with smooth targets}
\begin{prop} \label{prop:norm_finite_morphisms_smooth_target}
  Let $K$ be a discretely valued or trivially valued NA field of equicharacteristic $0$ such that $K$ admits a countable dense subfield. Let $Y$ be an irreducible affine variety over $K$ with $\dim Y=n$. Let $S$ be an affine open subvariety of $\BA^n_K$. Let $\pi \colon Y \to S$ be a finite surjective morphism. Then, there exists a map 
  \bens
  Y^{\an} &\longrightarrow& \Z_{>0} \\
  y &\longmapsto& \mult_y(\pi)
  \eens
  such that
  \begin{enumerate}
    \item There exists an affine open subvariety  $\mathscr{W}$ of $S$ such that $\pi$ is flat over $\mathscr{W}$, and $\mult_y(\pi)=\mult_y(\pi|_{\pi^{-1}(\mathscr{W})})$ for any $y \in \pi^{-1}(\mathscr{W})$, where $\mult_y(\pi|_{\pi^{-1}(\mathscr{W})})$ is the multiplicity at $y$ of the finite flat surjective morphism $\pi|_{\pi^{-1}(\mathscr{W})} \colon \pi^{-1}(\mathscr{W}) \to \mathscr{W}$ defined by Equation \ref{equation:def_multiplicity_flat_locus}.
    \item For any $u \in \WPSHlog(Y^{\an})$, we have $N_{\pi}(u) \in \PSHlogreg(S^{\an})$, with  \\
    $N_{\pi}(u) (s) \coloneqq \sum_{y \in \pi^{-1}(s)}  \mult_{y}(\pi) u(y)$.
  \end{enumerate}
\end{prop}
\begin{proof}
  \textbf{Step 1.} We define $\mult_y(\pi)$ for each $y \in Y^{\an}$. Since $\mathrm{char}\, K=0$ and being unramified is an open condition, these together with the generic flatness (\cite[Proposition 29.28.1]{stacks-project}) imply that there exists an affine open subvariety $\mathscr{W} \subset S$ such that $\pi$ is flat and unramified over $\mathscr{W}$. Then, $\pi$ is \'etale and in particular smooth over $\mathscr{W}$ (\cite[Lemma 29.37.5 and Lemma 29.37.15]{stacks-project}). Thus,    $\mathscr{U} \coloneqq \pi^{-1}(\mathscr{W})$ is a smooth affine variety over $K$.  
  
  Denote by $\pi_{\mathscr{U}} = \pi|_{\mathscr{U}} \colon \mathscr{U} \to \mathscr{W}$ the restriction of $\pi$ on $\mathscr{U}$. By assumption, $\pi_{\mathscr{U}}$ is a finite flat surjective morphism. For any $y \in \mathscr{U}^{\an}$, define $\mult_y(\pi) \coloneqq \mult_y(\pi_{\mathscr{U}})$, where $\mult_y(\pi_{\mathscr{U}})$ is the multiplicity of  $\pi_{\mathscr{U}}$ at $y$ defined by Equation \ref{equation:def_multiplicity_flat_locus}. 

  Now, we define $\mult_z(\pi)$ for $z \in Z^{\an}$, where $Z \coloneqq Y \setminus \mathscr{U}$. Let $s=\pi(z) \in S^{\an}$, and let $\BK \coloneqq \CH(s)$ denote the completed residue field of $s$. Let $\pi_{\BK} \colon Y_{\BK} \to S_{\BK}$ be the base change of $\pi$ to $\BK$, and let $r_{\BK/K}$ denote the restriction maps on Berkovich spaces induced by the base change $\BK/K$. Then, there exists an $\BK$-rational point $s' \in S_{\BK}^{\an}$ such that $r_{\BK/K}(s')=s$. Moreover, it is not hard to see that $r_{\BK/K} \colon Y_{\BK}^{\an} \to Y^{\an}$ induces a homeomorphism from $(\pi_{\BK}^{\an})^{-1}(s')$ to $(\pi^{\an})^{-1}(s)$. 
  
  Recall that $S_{\BK}$ is an affine open subvariety of $\BA_{\BK}^n$. Since $\mathscr{W}_{\BK}$ is Zariski open dense in $S_{\BK}$, we may take an affine line $C_s$ in $S_{\BK} \subset \BA^n_{\BK}$ passing $s'$ such that $C_s$ is contained in $\mathscr{W}_{\BK}$ generically. Let $t_1, \cdots, t_{n-1}$ be linear functions on $S_{\BK} \subset \BA_{\BK}^n$ such that $\BK[C_s] \cong \BK[S]/(t_1, \cdots, t_{n-1})$. Consider the scheme theoretical preimage $C_s \times_{S_{\BK}} Y_{\BK}$.   Then, Krull's  principal ideal theorem implies that each irreducible component of $C_s \times_{S_{\BK}} Y_{\BK}=\Spec \BK[Y]/(t_1, \cdots, t_{n-1})$ is of dimension at least $1$. On the other hand, since $\pi$ is finite, $\dim(C_s \times_{S_{\BK}} Y_{\BK}) \le 1$. This implies that $C_s \times_{S_{\BK}} Y_{\BK}$ is a scheme of equidimension $1$. 

  Let $\CC_s \coloneqq \pi_{\BK}^{-1}(C_s)$, be the reduction of $C_s \times_{S_{\BK}} Y_{\BK}$.  Let $\fp \colon \CC_s^{\nu} \to \CC_s$ be the normalization map, and denote $\fq = (\pi_{\BK}) \circ \fp \colon \CC_s^{\nu} \to C_s$. By construction, $\CC_s$ is a variety of equidimension $1$, and $\fq$ is a finite surjective morphism. Since both $\CC_s^{\nu}$ and $C_s$ are smooth, the miracle flatness (\cite[26.2.11]{Vak25}) implies that $\fq$ is flat. 

  Consider $\mult_{x}(\fq)$, the multiplicity of the finite flat surjective morphism $\fq$ at $x$ defined by Equation \ref{equation:def_multiplicity_flat_locus} for $x \in (\CC^{\nu}_s)^{\an}$. Let $z'$ be the unique point in $(\pi_{\BK}^{\an})^{-1}(s')$ such that $r_{\BK/K}(z')=s'$. Then, $z' \in \CC_s^{\an} = (\pi_{\BK}^{\an})^{-1}(C_s^{\an})$. Now, we define 
  \[
    \mult_{z}(\pi) \coloneqq \sum_{x' \in \fp^{-1}(z')}\mult_{x'}(\fq).
  \]

  We would like to say more about $\mult_{x}(\fq)$ before we start the next step. Since $\pi_{\BK}$ is \'etale over $\mathscr{W}_{\BK}$ and $C_s$ is a smooth variety over $\BK$, we obtain that $(C_s \cap \mathscr{W}_{\BK}) \times_{S_{\BK}} Y_{\BK}$ is a smooth scheme over $K$. In particular, $(C_s \cap \mathscr{W}_{\BK}) \times_{S_{\BK}} Y_{\BK}$ is reduced, and hence the reduction morphism $\CC_s \cap \mathscr{U}_{\BK} = \pi_{\BK}^{-1}(C_s \cap \mathscr{W}_{\BK}) \to (C_s \cap \mathscr{W}_{\BK}) \times_{S_{\BK}} Y_{\BK}$ is an isomorphism. This implies that $\CC_s \cap \mathscr{U}_{\BK}$ is smooth over $K$, and thereby $\fp$ is an isomorphism over $\CC_s \cap \mathscr{U}_{\BK}$. It also implies that $\pi_{\BK}|_{(\CC_s \cap \mathscr{U}_{\BK})} \colon \CC_s \cap \mathscr{U}_{\BK} \to C_s \cap \mathscr{W}_{\BK}$ is a finite flat surjective morphism. \footnote{We remind the readers that $C_s \times_S K$ might not be reduced, and therefore $\pi_{\BK}|_{\CC_s} \colon \CC_s \to C_s$ might not be flat.} Moreover, for each $t \in (C_s \cap \mathscr{W}_{\BK})^{\an}$, we have
  \bens
  \CH_{S_{\BK}}(t) \times_{S_{\BK}} Y_{\BK} =\CH_{C_s}(t) \times_{C_s}(C_s \times_{S_{\BK}} Y_{\BK}) \cong \CH_{C_s}(t) \times_{C_s} \CC_s \cong \CH_{C_s}(t) \times_{C_s} \CC_s^{\nu},
  \eens 
  where $\CH_{S_{\BK}}(t)=\CH_{C_s}(t)$ denote the completed residue field at $t$. 
  In particular, $\CH_{S_{\BK}}(t) \times_{S_{\BK}} Y_{\BK}$ and $\CH_{C_s}(t) \times_{C_s} \CC_s^{\nu}$ admit isomorphic decompositions as products of local finite algebras, indexed by fibres $\pi_{\BK}^{-1}(t)=\fq^{-1}(t)$.  
  Thus, for each $x \in \fp^{-1}(\CC_s \cap \mathscr{U}_{\BK})^{\an}$, we have 
  \be \label{equation:mult_x_q_proof:prop:norm_finite_morphisms_smooth_target}
  \mult_x(\fq)=\mult_{\fp(x)}(\pi_{\BK}|_{\CC_s \cap \mathscr{U}_{\BK}})=\mult_{\fp(x)}(\pi_{\BK}|_{\mathscr{U}_{\BK}}).
  \ee
  Here, $\mult_{\fp(x)}(\pi_{\BK}|_{(\CC_s \cap \mathscr{U}_{\BK})})$ and $\mult_{\fp(x)}(\pi_{\BK}|_{\mathscr{U}_{\BK}})$ denote the multiplicity at $\fp(x)$ of the finite flat surjective morphisms $\pi_{\BK}|_{(\CC_s \cap \mathscr{U}_{\BK})} \colon \CC_s \cap \mathscr{U}_{\BK} \to C_s \cap \mathscr{W}_{\BK}$ and $\pi_{\BK}|_{\mathscr{U}_{\BK}} \colon \mathscr{U}_{\BK} \to \mathscr{W}_{\BK}$ respectively. Combining this with Lemma \ref{lem:sum_of_multiplicity_equals_to_degree_flat_morphism}, we obtain that 
  \[
  \deg(\fq)=\sum_{x \in (\fq^{\an})^{-1}(t)} \mult_{x}(\fq)=\sum_{y \in (\pi^{\an})^{-1}(t)} \mult_{y}(\pi_{\BK}|_{\mathscr{U}})=\deg(\pi),
  \]
  where $t \in (C_s \cap \mathscr{W})^{\an}$.

  \smallskip

  \textbf{Step 2.} Fix $u \in \WPSHlog(Y^{\an})$, and  we would like to  show that $N_{\pi}(u) \in \PSHlogreg(S^{\an})$, where $N_{\pi}(u)(s) \coloneqq \sum_{y \in \pi^{-1}(s)}  \mult_{y}(\pi) u(y)$ with $\mult_{y}(\pi)$ defined in Step 1. To emphasize the role of $\pi$, we use the notation $N(\pi;u) \coloneqq N_\pi(u)$ interchangeably  in the rest of the proof.
  
  Since $\mathscr{U}$ is smooth, Proposition \ref{prop:psh_equals_weakly_psh_on_smooth_varieties} (2) implies that $u|_{\mathscr{U}} \in \PSHlogreg(\mathscr{U}^{\an})$. Since $\pi|_{\mathscr{U}} \colon \mathscr{U} \to \mathscr{W}$ is a finite flat surjective morphism,  Proposition \ref{prop:norm_finite_flat_morphisms} implies that 
  \[
  N_{\pi}(u)|_{\mathscr{W}^{\an}}=N(\pi|_{\mathscr{U}}; u|_{\mathscr{U}^{\an}}) \in \PSHlogreg(\mathscr{W}^{\an}).
  \]  
  
  Fix $g \in \FS(S^{\an})$. By Lemma \ref{lem:FS_functions_are_comparable} (2), there exists $C >0$ such that $u \le C(g \circ \pi^{\an})+O(1)$. Then, by Lemma \ref{lem:sum_of_multiplicity_equals_to_degree_flat_morphism},  we have  
  \[
  N_{\pi}(u)|_{\mathscr{W}^{\an}}=N(\pi|_{\mathscr{U}}; u|_{\mathscr{U}^{\an}}) \le C \cdot N(\pi|_{\mathscr{U}};(g \circ \pi^{\an})|_{\mathscr{U}^{\an}}) +O(1)=C \deg(\pi) \cdot g|_{\mathscr{W}^{\an}}+O(1). 
  \] 
  Therefore,  Theorem \ref{thm:pshlogreg_extension}  implies that there exists $v \in \PSHlogreg(S^{\an})$ such that $v|_{\mathscr{W}^{\an}}=N_{\pi}(u)|_{\mathscr{W}^{\an}}$.

  \smallskip

  \textbf{Step 3.} It remains to show that $N_{\pi}(u)(s)=v(s)$ for any $s \in (S \setminus \mathscr{W})^{\an}$. Fix $s \in (S \setminus \mathscr{W})^{\an}$ and denote $\BK \coloneqq \CH_S(s)$. Let $s' \in S_{\BK}^{\an}$ be the $\BK$-rational point and $C_s \subset S_{\BK}$ be the curve passing $s'$ chosen in Step 1 to define $\mult_{z}(\pi)$ for  $z \in \pi^{-1}(s)$. For any function $f$ on $S^{\an}$ (\emph{resp.} $Y^{\an}$), denote by $f_{\BK} \coloneqq f \circ r_{\BK/K}$.
  
  Then, we have $v_{\BK} \in \PSHlogreg(S_{\BK}^{\an})$, and thereby $v_{\BK}|_{C_{s}} \in \PSHlogreg(C_{s}^{\an})$. By Proposition \ref{prop:norm_finite_flat_morphisms} (3), we have 
  \be \label{equation:v_K_generic_proof:prop:norm_finite_morphisms_smooth_target}
  v_{\BK}|_{\mathscr{W}_{\BK}^{\an}}= \left( N_{\pi}(u)_{\BK} \right)|_{\mathscr{W}_{\BK}^{\an}}=  N( \pi|_{\mathscr{U}} ; u|_{\mathscr{U}^{\an}}) \circ r_{\BK/K}  =N\left(\pi_{\BK}|_{\mathscr{U}_\BK}; (u_{\BK})|_{\mathscr{U}^{\an}} \right).
  \ee
  
  Now, we consider $u_{\BK}|_{\CC_s} \in \WPSHlog(\CC_s^{\an})$. Let $\mu \colon \widetilde{Y} \to Y$ be a proper birational morphism such that $u \circ \mu^{\an} \in \PSHlog(\widetilde{Y}^{\an})$. Let $\widetilde{\CC_s} \subset \widetilde{Y}_{\BK}$ be the proper transform of $\CC_s$ under $\mu_{\BK}$. Then, there exists a proper birational morphism $\fr \colon \CC_s^{\nu} \to \widetilde{\CC_s}$ such that $\fp =\mu_{\BK} \circ \fr$.  By definition, 
  \be
  N_{\pi}(u)(s)=\sum_{z \in \pi^{-1}(s)} \mult_{z}(\pi)\, u(z) =\sum_{x' \in \fq^{-1}(s')} \mult_{x'}(\fq)\, u_{\BK}(x') =N(\fq; u_{\BK} \circ \fp)(s'), 
  \ee
  where $N(\fq; u_{\BK} \circ \fp)$ is the norm under the finite flat surjective morphism $\fq=(\pi_{\BK}) \circ \fp \colon \CC_s^{\nu} \to C_s$. Note that $(u \circ \mu^{\an})_{\BK} \in \PSHlog(\widetilde{Y}_{\BK}^{\an})$, and thereby $u_{\BK} \circ \fp=(u \circ \mu^{\an})_{\BK} \circ \fr \in \PSHlog((\CC_s^{\nu})^{\an})$. Since $\CC_s^{\nu}$ is an affine smooth curve, we have $u_{\BK} \circ \fp \in \PSHlogreg((\CC_s^{\nu})^{\an})$ by Proposition \ref{prop:pshlog_equals_pshlogreg_on_smooth_affine_varieties}. Then, Proposition \ref{prop:norm_finite_flat_morphisms} implies that $N(\fq; u_{\BK} \circ \fp) \in \PSHlogreg(C_s^{\an})$. On the other hand, recall that $\fp$  induces an isomorphism from $\fp^{-1}(\CC_s \cap \mathscr{U}_{\BK})$ to $ \CC_s \cap \mathscr{U}_{\BK}$.  Equation \ref{equation:mult_x_q_proof:prop:norm_finite_morphisms_smooth_target} implies that for any $t \in (C_s \cap \mathscr{W}_{\BK})^{\an}$,  we have
  \be
  N(\fq; u_{\BK} \circ \fp)(t)=\sum_{x\in \fq^{-1}(t)} \mult_x(\fq) \, u_{\BK} ( \fp(x))=\sum_{y \in \pi_{\BK}^{-1}(t)} \mult_{y}(\pi_{\BK}|_{\mathscr{U}_{\BK}})\, u_{\BK}(y) = N(\pi_{\BK}|_{\mathscr{U}_{\BK}} ; u_{\BK}|_{\mathscr{U}^{\an}}).
  \ee

  In summary, we obtain that $v_{\BK}|_{C_s} \in \PSHlogreg(C_s^{\an})$ and $N(\fq; u_{\BK} \circ \fp) \in \PSHlogreg(C_s^{\an})$ satisfying $v_{\BK}|_{C_s \cap \mathscr{W}}=N(\fq; u_{\BK} \circ \fp)|_{C_s \cap \mathscr{W}}$. Moreover, we have $N_{\pi}(u)(s)=N(\fq; u_{\BK} \circ \fp)(s')$. 
  Since $C_s$ is a smooth curve over the complete NA field $\BK=\CH_S(s)$,  it follows from Lemma \ref{lem:psh_extension_on_smooth_curves} $v_{\BK}|_{C_s}  =N(\fq; u_{\BK} \circ \fp)$ on $C_s^{\an}$. Combining these, we obtain that $v(s)=v_{\BK}(s')=N(\fq; u_{\BK} \circ \fp)(s')=N_{\pi}(u)(s)$. This shows that $N_{\pi}(u)=v \in \PSHlogreg(S^{\an})$ and finishes the proof.  
\end{proof}

\begin{lemma} \label{lem:psh_extension_on_smooth_curves}
  Let $F$ be \emph{any} complete NA field, and $C$ be a smooth algebraic curve over $F$. Let $z \in C^{\an}$ be a closed point. Then, for any $u \in \PSH(C^{\an})$, we have
  \[
    u(z) =\limsup_{x \to z,\,  x \in C^{\an} \setminus \{z\}} u(x)
  \]
\end{lemma}
\begin{proof}
  Note that $r \colon (C_{\widehat{\overline{F}}})^{\an} \to C^{\an}$ is a surjective continuous map, and $u \circ r \in \PSH((C_{\widehat{\overline{F}}})^{\an})$, where $\widehat{\overline{F}}$ denotes the completion of the algebraic closure of $F$. Therefore, it is enough to prove the lemma when $F$ is algebraically closed. When $F$ is trivially valued, the lemma is proven in Proposition \ref{prop:local_comparison_divisorial_points_to_full_spaces} (1). Thus, we may assume that $F$ is nontrivially valued and algebraically closed. In this setting, we say that a function $f$ on an open subset $W$ is subharmonic if it is subharmonic on $W$ under the definition of Thuillier (\cite[Definition 3.1.5]{Thu05}, \cite[Definition 3.1.5]{Wan19thesis}). Let $\{W_{\alpha}\}_{\alpha \in \CI}$ be an open cover of $C^{\an}$ such that for each $\alpha \in \CI$, there exists a decreasing function sequence $\{u_{\alpha, i}\}_{i \ge 1} \in \TConv \cap \PL(W_{\alpha})$ converging pointwise to $u|_{W_{\alpha}}$. Since $\MA(u_{\alpha, i})=\dm \dmb u_{\alpha,i } \ge 0$,  \cite[Theorem 2]{Wan19comparison} implies that $u_{\alpha, i}$ is subharmonic on $W_{\alpha}$. Then, \cite[Proposition 3.1.8]{Thu05} implies that $u|_{W_{\alpha}}$ is subharmonic. Therefore, $u$ is subharmonic on $C^{\an}$, and  the lemma follows from \cite[Lemma 3.1.33]{Wan19thesis} (see also \cite[Proposition 8.11]{BR10}).
\end{proof}

\subsection{}
\begin{prop} \label{prop:weakly_psh_induces_Runge_under_injective_conditions}
  Let $Y$ be an irreducible affine variety over $K$ and let $Z$ be a closed subvariety of $Y$. Let $u \in \WPSHlog(Y^{\an})$.  Suppose that there exists an affine open cover $\{U_{\alpha}\}_{\alpha \in \CI}$ of $Y \setminus Z$ such that $u|_{U_{\alpha}} \in \PSHlogreg((U_{\alpha})^{\an})$ for each $\alpha \in \CI$. Suppose further that $Y$ admits a finite surjective morphism $\pi \colon Y \to S$ such that $S$ is an affine open subvariety of $\BA_K^n$, and $\pi|_Z \colon Z \to \pi(Z)$ is an injective scheme morphism. Then, $\{x \in Y^{\an} \, |\,  u(x)<0 \}$ is a Runge open subset in $Y^{\an}$, and $u \in \PSHlogreg(Y^{\an})$.
\end{prop}

\medskip

We first prove the following lemma.

\begin{lemma} \label{lem:aux_lemma_2}
  Under the hypotheses of Proposition \ref{prop:weakly_psh_induces_Runge_under_injective_conditions}, suppose further that $Z$ and \\ 
  $\pi^{-1}(\pi(Z)) \setminus Z$ are disjoint closed subvariety of $Y$. Then,  $\{x \in Y^{\an} \,|\, u(x) <0\}$ is a Runge open subset of $Y^{\an}$. 
\end{lemma}
\begin{proof}
  \textbf{Step 1.} By Lemma \ref{lem:basic_properties_poly_convex_and_Runge} (2), it is sufficient to show that $\{u_c<0\}$ is Runge in $Y^{\an}$ for any $c>0$. We prove this by contracdition. Suppose that there exists $c>0$ such that $\{u_c < 0\}$ is not Runge. After adding a constant to $u_c$, we may assume there exists a compact subset $L \subset Y^{\an}$ such that $\sup_{\widehat{L}_Y} u_c=0$ and  $\sup_{L}u_c <0$.  Let $x \in \widehat{L}_Y$ such that $u_c(x)=\sup_{\widehat{L}_Y} u_c = 0$. Then,  Lemma \ref{lem:aux_lemma_1} implies that  $x \in Z^{\an}$.  

  \smallskip

  \textbf{Step 2.} Consider $\CC \coloneqq \widehat{\{\pi(x)\}}_{S} \subset S^{\an}$. By Lemma \ref{lem:basic_properties_poly_convex_and_Runge} (3),  $\CC  \subset (\pi(Z))^{\an}$. Denote $Z' \coloneqq \pi^{-1}(\pi(Z)) \setminus Z$.  By assumption, $Z^{\an}$ and $(Z')^{\an}$ are disjoint closed subset of $Y^{\an}$. Since $\pi^{\an}$ is a proper map (\cite[1.3.22]{Duc18}),
  \[
    \CC_1 \coloneqq (\pi^{\an})^{-1}(\CC) \cap Z^{\an}, \quad \text{ and } \quad \CC_2 \coloneqq (\pi^{\an})^{-1}(\CC) \cap (Z')^{\an}. 
  \]
  are compact subsets of $Y^{\an}$. By construction,  $\pi^{-1}(\CC)=\CC_1 \cup \CC_2$, and $\CC_1 \cap \CC_2=\emptyset$. 
  Since $Y^{\an}$ is locally compact Hausdorff, there exists open neighborhoods $\fA_1$ and $\fA_2$ of $\CC_1$ and $\CC_2$ respectively such that 
  \[
    \overline{\fA_1} \cap \overline{\fA_2}=\emptyset, \quad \overline{\fA_1} \cap (Z')^{\an}=\emptyset, \quad \overline{\fA_2} \cap Z^{\an}=\emptyset. 
  \]
  By \cite[1.3.22]{Duc18}, $\pi^{\an}$ is a closed map. Thus,   
  \[
    \fA \coloneqq S^{\an} \setminus \left(\pi^{\an} \left(Y^{\an} \setminus (\fA_1 \cup \fA_2) \right) \right) \subset S^{\an}
  \] is an open neighborhood of  $\CC \subset \fA$ satisfying $\pi^{-1}(\fA) \subset \fA_1 \cup \fA_2$.  
  
  By Proposition \ref{prop:exhausion_of_Runge_domains_by_tropical_polytopes} (2), there exists tropical polytopes $P, Q$ such that 
  \[
  \CC \subset \interior(P) \subset P \subset \interior(Q) \subset Q \subset \fA.
  \]
  Moreover, by Lemma \ref{lem:tropical_polytopes_and_FS_functions}, we may assume that there exists $g \in \FS(Y^{\an})$ such that $g \ge 0$ on $Y^{\an}$, $P=\{g=0\}$ and $Q=\{g \le 1\}$.  Denote 
  \[
  U \coloneqq \interior(P), \ \  W \coloneqq \interior(Q); \qquad \quad \text{and} \qquad \quad U^\sharp \coloneqq   \pi^{-1}(U), \ \  W^{\sharp} \coloneqq \pi^{-1}(W).
  \]
  For each $i \in \{1,2\}$, set $U_i \coloneqq U^\sharp \cap \fA_i$, and $W_i \coloneqq W^{\sharp} \cap \fA_i$. Then, by constructions, we have 
  \[
    U=U_1 \cup U_2, \  W=W_1 \cup W_2;  \quad x \in \CC_1 \subset U_1; \quad \overline{W_1} \cap \overline{W_2}=\emptyset,  \ \  \text{ and } \ \  U_i \Subset W_i \Subset \fA_i. 
  \]  
  In particular, we have 
  \[
    \overline{U_1} \cap \overline{U_2}=\emptyset, \ \left( \overline{U_1} \cap (Z')^{\an} \right) \subset \left( \overline{\fA_1} \cap (Z')^{\an} \right)=\emptyset, \ \text{and } \left( \overline{U_2} \cap (Z)^{\an} \right) \subset \left( \overline{\fA_2} \cap (Z)^{\an} \right)=\emptyset.
  \]
  Moreover, since $\pi^{\an}$ is proper, $\overline{U_i}$ and $\overline{W_i}$ are compact subsets of $Y^{\an}$. 

  \smallskip

  \textbf{Step 3.}  Now we show that there exists $\gamma \in \Z_{>0}$ such that 
  \be \label{equation:sum_of_multiplicity_is_locally_contant}
    \sum_{y \in \pi^{-1}(s) \cap U_1} \mult_y(\pi)= \gamma \qquad \text{for each } s \in U. 
  \ee
  Set $h \colon W^{\sharp} \to \R$ by $h=1$ on $W_1$ and $h=0$ on $W_2$. Then, $h \in \TConv \cap \PL(W^{\sharp})$. Denote
  \[
   g^{\sharp} \coloneqq g \circ \pi. 
  \]
  By constructions, we have $\{y \in Y^{\an} \,|\, 3 g^{\sharp}(y)-1 \le 2\}=\pi^{-1}(Q) \subset Y^{\an}$. 
  Set 
  \[
  \mathscr{D} \coloneqq \{y \in W^{\sharp} \,|\,   3 g^{\sharp}(y)-1 < h(y)\} \subset W^{\sharp}. 
  \]
  Then, we have
  \[
    \overline{\mathscr{D}} \subset \{y \in Y^{\an} \,|\,  3 g^{\sharp}(y)-1 \le 1\} \subset \pi^{-1}(\interior(Q)) = W^{\sharp}.    
  \]
  Define $\widetilde{h} \colon Y^{\an} \to \R$ by
  \[
    \widetilde{h}=\max\{3g^{\sharp}-1, h\} \text{ on } W^{\sharp}, \qquad \text{and} \qquad \widetilde{h}=3 g^{\sharp} -1 \text{ on } Y^{\an} \setminus W^{\sharp}. 
  \]
  Then, Remark \ref{rmk:extension_of_psh_functions_by_sup} and Lemma \ref{lem:psh_PL_functions_are_pshlogreg} (2) imply that  $\widetilde{h} \in \PSH \cap \PL(Y^{\an}) \subset \PSHlogreg(Y^{\an})$. On the other hand, since $3g^{\sharp}-1=-1$ on $U^{\sharp}$, we have $\widetilde{h}=h$ on  $U^{\sharp}$. 

  Consider the norm of $\widetilde{h}$ under $\pi$ defined by Proposition \ref{prop:norm_finite_morphisms_smooth_target}. By Proposition \ref{prop:norm_finite_morphisms_smooth_target} (2), we have $N_{\pi}(\widetilde{h}) \in \PSHlogreg(S^{\an})$. Moreover, by Proposition \ref{prop:norm_finite_morphisms_smooth_target} (1), there exists an affine open neighborhood $\mathscr{W}$ of $S$ such that $\pi$ is flat over $\mathscr{W}$, and $N_{\pi}(\widetilde{h})|_{\mathscr{W}}$ equals to the norm of $\widetilde{h}|_{\pi^{-1}(\mathscr{W})}$ under the finite flat surjective morphism $\pi|_{\pi^{-1}(\mathscr{W})}$ defined by Equation \ref{equation:def_multiplicity_flat_locus} and \ref{equation:def_norm_flat_locus}. Since $\widetilde{h}|_{\pi^{-1}(\mathscr{W})}$ is continuous, by  Proposition \ref{prop:norm_finite_flat_morphisms} (2), we obtain that $N_{\pi}(\widetilde{h})|_{\mathscr{W}}$ is continuous. 
  On the other hand,    by constructions,
  \[
   N_{\pi}(\widetilde{h})(s)=\sum_{y \in \pi^{-1}(s) \cap U_1} \mult_y(\pi) \qquad \text{for each } s\in U. 
  \]
  In particular, $N_{\pi}(\widetilde{h})$ takes positive integer values on $U$. Thus, there exists $\gamma \in \Z_{>0}$ such that $N_{\pi}(\widetilde{h})=\gamma$ on $U \cap \mathscr{W}$. Then, Proposition \ref{prop:local_comparison_divisorial_points_to_full_spaces} implies further that 
  \[
    N_{\pi}(\widetilde{h})(s') = \limsup_{s \to s', s \in U^{\div} \subset U \cap \mathscr{W}} N_{\pi}(\widetilde{h})(s) =\gamma  \qquad \text{for any } s' \in U.
  \]
  Therefore, Equation \ref{equation:sum_of_multiplicity_is_locally_contant} is proven.

  \smallskip

  \textbf{Step 4.} Let $\delta >0$ such that $-\delta=\sup_{L} u_c$. Since $u_c$ is usc, there exists an open neighborhood $\fV$ of $x$ in $U_1$ such that $u_c < \delta/2(\gamma-1)$ on $\fV$. Then, $V \coloneqq S^{\an} \setminus \left(\pi^{\an} \left(Y^{\an} \setminus (\fV \cup U_2) \right) \right) \subset U$ is an open neighborhood of $\pi(x)$, and we have $\pi^{-1}(V) \subset \fV \cup U_2$. Set $V_1 \coloneqq \pi^{-1}(V) \cap U_1 \subset \fV$. Then, $x \in V_1$.  Take a compact open neighborhood $\CU$ of $x$ such that $\overline{\CU} \subset V_1$.

  Applying Lemma \ref{lem:strictly_psh_functions} to $(\CU, x)$, we obtain $\rho' \in \PSHlogreg(Y^{\an})$ and $\delta' >0$ such that $\rho'(x)=\rho(x)$, $\rho' \le \rho$ on $Y^{\an}$ and $\rho' \le \rho-\delta'$ on $\partial \CU$. Set $u'_c \coloneqq u+c\rho' \in \WPSHlog(Y^{\an})$. Then, we have 
  \[
  u'_c(x)=0, \quad \quad \text{and} \quad \quad u'_c \le u_c < \delta /2(\gamma-1) \  \text{ on } V_1 \subset \fV. 
  \]
  Define the compact subset $\CK \coloneqq \partial \CU \cap \widehat{L}_Y$. Then, we have
  \[
  \sup_{\CK} u'_c \le \sup_{\CK} u_c -\delta' \le \sup_{\widehat{L}_Y} u_c -\delta' = -\delta'.
  \]

  Set $v \colon W^{\sharp} \to \R \cup \{-\infty\}$ by $v=u'_c$ on $W_1$ and $v=0$ on $W_2$. Then, $v \in \WPSH(W^{\sharp})$ since $W_1 \cap W_2=\emptyset$.  Let $M \coloneqq \sup_{W^{\sharp}} v  \ge 0$. Set $C \coloneqq M+1 +\delta>0$. Then, we have $Q=\{s \in S^{\an} \,|\, Cg(s)-\delta \le M+1\} \subset S^{\an}$. Define $\mathscr{D}' \coloneqq \{y \in W^{\sharp} \,|\,   3 g^{\sharp}(y)-1 < v(y)\} \subset W^{\sharp}$.   Then, we have  $\mathscr{D}' \subset \{y \in Y^{\an} \,|\, Cg^{\sharp}(y)-\delta < M \}$. This implies that 
  \[
    \overline{\mathscr{D}'} \subset \{y \in Y^{\an} \,|\, Cg^{\sharp}(y)-\delta  \le M \} \subset \pi^{-1}(\interior(Q)) = W^{\sharp}.    
  \]
  Define $\widetilde{v} \colon Y^{\an} \to \R$ by
  \[
    \widetilde{v}=\max\{C g^{\sharp}-\delta, v \} \text{ on } W^{\sharp}, \qquad \text{and} \qquad \widetilde{v}=C g^{\sharp}-\delta \text{ on } Y^{\an} \setminus W^{\sharp}. 
  \]
  Then, by Remark \ref{rmk:extension_of_psh_functions_by_sup}, $\widetilde{v} \in \WPSH(Y^{\an})$. Since $\widetilde{v} \le \max\{Cg^{\sharp}-\delta, M\}$ and $Cg^{\sharp}-\delta \in \PSHlog(Y^{\an})$, we have $\widetilde{v} \in \WPSHlog(Y^{\an})$. Since $Cg^{\sharp}-\delta=-\delta$ on $U^{\sharp}$, we have 
  \[
    \widetilde{v}=\max\{u'_c, -\delta\} \text{ on } U_1; \qquad \qquad  \widetilde{v}=0 \text{ on } U_2.
  \]
  In particular, we have
  \be \label{equation:estimation_tilde_v}
    \widetilde{v}(x)=0, \quad \widetilde{v} < \delta/2(\gamma-1) \text{ on } V_1, \quad \text{ and } \sup_{\CK} \widetilde{v} \le -\min\{\delta, \delta'\}.   
  \ee

  \smallskip

  \textbf{Step 5.} Consider the norm $w \coloneqq N_{\pi}(\widetilde{v}) \in \PSHlogreg(S^{\an})$ defined by Proposition \ref{prop:norm_finite_morphisms_smooth_target}. Set $w^{\sharp} \coloneqq w \circ \pi$. Let $\{w_i\}_{i \ge 1} \subset \FS(Y^{\an})$ be a decreasing function sequence converging to $w$ pointwise. Then, for each $i \ge 1$, $w_i \circ \pi \in \PSHlogreg \cap \PL(Y^{\an})$, and  Lemma \ref{lem:psh_regularization_with_growth_control} (3) implies that $w^{\sharp} \in \PSHlogreg(Y^{\an})$. By Rossi's local maximum principle (Theorem \ref{thm:rossi_local_maximum_principle}), we have $x \in \widehat{\CK}_Y$. Then, Lemma \ref{lem:psh_to_Runge_domain_smooth_varieties} implies that 
  \[
  w^{\sharp}(x) \le \sup_{\CK} w^{\sharp}.
  \]
  
  \smallskip

  Since $\pi|_{Z}$ is an injective scheme morphism, \cite[Proposition 3.4.6]{Ber90} implies that $\pi^{\an}|_{Z^{\an}}$ is injective. On the other hand, since $\left( U_1 \cap (Z')^{\an} \right) =\emptyset$, we have 
  \[
    \left( \pi^{-1}(\pi(z)) \cap U_1 \right) \subset \left( \pi^{-1}(\pi(z)) \cap Z^{\an} \right)=\{z\} \qquad \text{ for any } z \in U_1 \cap Z^{\an}.
  \] 
  This  together with Equation \ref{equation:sum_of_multiplicity_is_locally_contant} implies that
  \[
    \mult_z(\pi)=\gamma, \  \ w^{\sharp}(z)=\gamma \widetilde{v}(z) \qquad \text{ for any } z \in U_1 \cap Z^{\an}.
  \]
  In particular, we have $w^{\sharp}(x)=0$, and Equation \ref{equation:estimation_tilde_v} implies that 
  \[
    \sup_{\CK \cap Z^{\an}} w^{\sharp} \le -\gamma \min\{\delta, \delta'\}<0.
  \]

  \smallskip

  Now fix $y \in (\CK \setminus Z^{\an}) \subset (\widehat{L}_Y \setminus Z^{\an})$. Then, $u'_c(y) \le u_c(y) \le \sup_L u_c=-\delta$, where the second inequality is implied by  Lemma \ref{lem:aux_lemma_1}. Therefore, we have $\widetilde{v}(y) =-\delta$. On the other hand,  note that $y \in \CK \subset V_1$, and $\pi(y) \in V$. Thus, 
  we have
  $\left( \pi^{-1}(\pi(y)) \cap U_1 \right) \subset \left( \pi^{-1}(V) \cap U_1 \right)=V_1$. Hence, $\widetilde{v}(y')< \delta/2(\gamma-1)$ for any $y' \in \pi^{-1}(\pi(y)) \cap U_1$. These together with Equation \ref{equation:sum_of_multiplicity_is_locally_contant} imply that
  \bens
    w^{\sharp}(y) &=& \mult_y(\pi) \, \widetilde{v}(y) +\sum_{y' \in \pi^{-1}(\pi(y)), y' \neq y} \mult_{y'}(\pi) \, \widetilde{v}(y') \\
    &\le& -\delta+(\gamma-1) \cdot  \delta/2(\gamma-1) =-\delta/2<0.
  \eens

  In conclusion, we have shown that $\sup_{\CK} w^{\sharp} <0=w^{\sharp}(x)$. This induces a contradiction, and the lemma is proved. 
\end{proof}

\bigskip

Now we prove Proposition \ref{prop:weakly_psh_induces_Runge_under_injective_conditions} in full generality. 

\begin{proof} [Proof of Proposition \ref{prop:weakly_psh_induces_Runge_under_injective_conditions}] \label{proof:prop:weakly_psh_induces_Runge_under_injective_conditions}
  Let $\dim Y=n$. We first  inductively construct a strictly descending sequence of subvarieties
  \[
  Z=Z_1 \supsetneq Z_2 \supsetneq \cdots \supsetneq Z_n \supsetneq Z_{n+1}=\emptyset.
  \]
  such that 
  \begin{enumerate}
    \item $\dim Z_k \ge \dim Z_{k+1} +1$  for each  $k \ge 1$,
    \item $ \pi^{-1} \left( \pi(Z_k \setminus Z_{k+1}) \right) \setminus Z_k $ is a closed subvariety of $Y \setminus \pi^{-1} \left(\pi(Z_{k+1}) \right)$.
  \end{enumerate}

  We start with $Z_1=Z$. Suppose that $k \ge 1$ and $Z_k$ has been constructed. Then, we define
  \[
    Z_{k+1} \coloneqq Z_k \cap \overline{ \left( \pi^{-1} \left( \pi(Z_k) \right) \setminus Z_k \right) }^{\mathrm{Zar}, Y}. 
  \]
  Now we verify the two properties. Let $Z_{k,\alpha}$ be an irreducible component of $Z_k$ with $\dim Z_{k, \alpha}=\dim Z_k$. Since $\pi$ is finite, we have $\dim Z_k =\dim \pi(Z_k)=\dim \pi^{-1} \left(\pi(Z_k)\right)$. This implies that $Z_{k,\alpha}$ is also an irreducible component of $\pi^{-1} \left(\pi(Z_k)\right)$. Thus, we have $Z_{k,\alpha} \nsubseteq \overline{ \left( \pi^{-1} \left( \pi(Z_k) \right) \setminus Z_k \right) }^{\mathrm{Zar}, Y}$, and hence $\dim \left( Z_{k, \alpha} \cap \overline{ \left( \pi^{-1} \left( \pi(Z_k) \right) \setminus Z_k \right) }^{\mathrm{Zar}, Y} \right) \le \dim Z_k -1$. Therefore, we conclude that $\dim Z_{k+1} \le \dim Z_{k}-1$, and thereby $Z_{n+1}=\emptyset$.  

  On the other hand, we have
  \bens
  && \overline{\left( \pi^{-1} \left( \pi(Z_k \setminus Z_{k+1}) \right) \setminus Z_k \right)}^{\mathrm{Zar}, Y \setminus \pi^{-1}\left(\pi(Z_{k+1}) \right)} \\
  &\subset& \overline{\pi^{-1}\!\bigl(\pi(Z_k)\bigr)\setminus Z_k}^{\,\mathrm{Zar},\,Y} \setminus \pi^{-1}\!\bigl(\pi(Z_{k+1})\bigr) \\
  &=& \left( \left( \pi^{-1} \left( \pi(Z_k) \right) \setminus Z_k \right) \cup Z_{k+1} \right) \setminus \pi^{-1}\left(\pi(Z_{k+1}) \right)  \\
  &=&  \pi^{-1} \left( \pi(Z_k) \right) \setminus \left(  \pi^{-1}\left(\pi(Z_{k+1}) \right) \cup Z_k \right) 
  =\left( \pi^{-1} \left( \pi(Z_k \setminus Z_{k+1}) \right) \setminus Z_k \right).
  \eens
Here, the equality between the second line and the third line follows from the definition of $Z_{k+1}$ and the inclusion $ \overline{\pi^{-1}\!\bigl(\pi(Z_k)\bigr)\setminus Z_k}^{\,\mathrm{Zar},\,Y} \subset \pi^{-1} \left( \pi(Z_k) \right)$. This proves the second property.

  \bigskip

  Now consider the following statement.

  \medskip

  \textbf{Statement} $\mathbf{\Pi}_k$: For any closed scheme point $y \in |Y \setminus Z_k|_{\cl}$, there exists an affine open neighborhood  $U$ of $y$ in  $Y \setminus Z_k$ such that  $\{x \in U^{\an} \,|\, u(x)<0 \}$ is Runge in $U^{\an}$, and $u|_{U} \in \PSHlogreg(U^{\an})$. 

  \medskip

  Note that the proposition follows from $\mathbf{\Pi}_{n+1}$ and Corollary \ref{cor:pshlogreg_on_an_affine_open_cover_induces_pshlogreg}. We prove $\mathbf{\Pi}_k$ by induction. When $k=1$, this follows from the assumption of the proposition. Now suppose that $k \ge 1$ and $\mathbf{\Pi}_k$ has been proven. Let $y \in |Y \setminus Z_{k+1}|_{\cl}$. If $y \in \pi^{-1} \left( \pi(Z_{k+1})\right) \setminus Z_{k+1}$. Since $\pi|_Z$ is an injective scheme morphism and $Z_{k+1} \subset Z$, we have $y \in Y \setminus Z$. Then by assumption there exists affine open neighborhood $U_{\alpha}$ of $y$ satisfying required properties. 

  Now we assume that $y \in Y \setminus \pi^{-1} \left( \pi(Z_{k+1})\right)$. Let $V$ be an affine open neighborhood of $\pi(y) \in S \setminus \pi(Z_{k+1})$. Then,  $U \coloneqq \pi^{-1}(V) \subset Y \setminus \pi^{-1} \left( \pi(Z_{k+1})\right)$ is an affine open neighborhood of $y$, and $\pi|_U \colon U \to V$ is a finite surjective morphism.   

  For any variety $\CW$ over $K$, denote $\widetilde{\CW} \coloneqq \CW \times \BA^1_K$. Let $p \colon  \widetilde{Y} \to Y$ be the projection to the first component, and let $w \colon \widetilde{Y} \to \BA^1_K$ be the projection to the second component. We view $w$ as a element in $K[\widetilde{Y}]$. By definition, $\log|w|+u \circ p^{\an} \in \WPSHlog(\widetilde{Y}^{\an})$. Denote $\widetilde{\pi}  =(\pi, \mathrm{id}_{\BA^1_K})\colon \widetilde{Y} \to \widetilde{S}$. Then,  we have $\widetilde{U}=\widetilde{\pi}^{-1}(\widetilde{V})$, and $\widetilde{\pi}|_{\widetilde{U}} \colon \widetilde{U} \to \widetilde{V}$ is a finite surjective map.

  By the induction assumption $\mathbf{\Pi}_k$, for any  $\fm \in |U \setminus Z_k|_{\cl}$, there exists an affine open neighborhood $U_{\fm}$ of $\fm$ in $U$ such that $u|_{U_{\fm}} \in \PSHlogreg(U_{\fm})$. Note that $\left\{\widetilde{U_{\fm}}, \ \fm \in |U \setminus Z_k|_{\cl} \right\}$ is an affine open cover of $\widetilde{U} \setminus \widetilde{Z_k}=(U \setminus Z_k) \times \BA^1_K$, and $\left(\log|w| +u \circ p^{\an}\right)|_{\widetilde{U_{\fm}}} \in \PSHlogreg((\widetilde{U_{\fm}})^{\an})$ for each $\fm \in |U \setminus Z_k|_{\cl}$. 

  Recall that by assumption, $\pi|_Z$ is an injective scheme morphism. Thus, $\widetilde{\pi}|_{\widetilde{U} \cap \widetilde{Z_k}}$ is an injective scheme morphism. \lnote{since $Z_k \subset Z$.} On the other hand, we have 
  \bens
  &&\pi^{-1} \left( \pi(U \cap Z_k ) \right) \setminus (U \cap Z_k)=U \cap \left( \pi^{-1} \left( \pi( Z_k ) \right) \setminus Z_k \right) \\
  &=&U \cap \left( \pi^{-1} \left( \pi(Z_k \setminus Z_{k+1}) \right) \setminus Z_k \right) \qquad  \qquad \qquad \left(\ \text{since } U  \subset Y \setminus \pi^{-1} \left( \pi(Z_{k+1})\right) \ \right).
  \eens
  Since $\pi^{-1} \left( \pi(Z_k \setminus Z_{k+1}) \right) \setminus Z_k $ is Zariski closed in $Y \setminus \pi^{-1} \left( \pi(Z_{k+1}) \right)$ and $U \subset Y \setminus \pi^{-1} \left( \pi(Z_{k+1}) \right)$, we obtain that 
  $\pi^{-1} \left( \pi(U \cap Z_k ) \right) \setminus (U \cap Z_k)$ is Zariski closed in $U$.
  This implies further that 
  \[
  \widetilde{\pi}^{-1} \left( \widetilde{\pi}(\widetilde{U} \cap \widetilde{Z_k} ) \right) \setminus (\widetilde{U} \cap \widetilde{Z_k})= \big( \pi^{-1} \left( \pi(U \cap Z_k ) \right) \setminus (U \cap Z_k) \big) \times \BA^1_K 
  \]
  is Zariski  closed  in $\widetilde{U}$. 

  Applying Lemma \ref{lem:aux_lemma_2} to the irreducible affine variety $\widetilde{U}$, its closed subvariety $(\widetilde{U} \cap \widetilde{Z_k}) \subset \widetilde{U}$, $\left(\log|w|+u \circ p^{\an} \right)|_{\widetilde{U}^{\an}} \in \WPSHlog(\widetilde{U}^{\an})$ and $\widetilde{\pi}|_{\widetilde{U}} \colon \widetilde{U} \to \widetilde{V}$, we obtain  that  
  \[
  D=\{x \in \widetilde{U}^{\an} \,|\, \left( \log|w|+u \circ p^{\an} \right)(x)<0\} 
  \]
  is Runge in $\widetilde{U}^{\an}$. Note that $(U \times \{0\})^{\an} \subset D$, and $-\log d_D =u$ on $U^{\an}$. Then, Proposition \ref{prop:Runge_domains_to_psh} implies that $u \in \PSHlogreg(Y^{\an})$, and Lemma \ref{lem:psh_to_Runge_domain_smooth_varieties} implies further that $\{x \in U^{\an} \,|\, u(x)<0\}$ is Runge in $Y^{\an}$.  This proves statement $\mathbf{\Pi}_{k+1}$, and the induction is completed. 
  \end{proof}

\subsection{Proof of Theorem \ref{thm:NA_FN_theorem}}

We first recall the following version of Noether normalization lemma. 

\begin{lemma} \label{lem:generic_noether_normalization}
  Let $K$ be an infinite field. Let $Y$ be an affine variety of dimension $n$ over $K$, and let $Z \subsetneq Y$ be a nonsingular closed subvariety of $Y$ with $\dim Z \le n-1$. Then, for each closed  scheme point $z \in |Z|_{\cl}$, there exists an affine open neighborhood $U$ of $z$ in $Y$, an affine open subvariety $S$ of $\BA^n_K$ and a finite surjective morphism $\pi \colon U \to S$ such that $\pi|_{Z_U} \colon Z_U \to \pi(Z_U)$ is an isomorphism with $Z_U \coloneqq Z \cap U$. 
\end{lemma}
\begin{proof}
  We may assume that $Y$ is a dense open affine subvariety of a projective variety $X \subset \P_K^d$ with $d \ge n+1$. By assumption, $\CZ \coloneqq \overline{Z} \cup (X \setminus Y)$ the closed subvariety of $X$ with $\dim \CZ \le n-1$, and $z$ is a nonsingular point of $\CZ$. 
  
  For any $K$-rational closed point $p \in \P^d(K)$, one has a projection $r_{p} \colon \P_K^d \setminus \{p\} \to \P_K^{d-1}$ over $K$. We claim that if $d \ge n+1$, then there exists $p \in \P^d(K)$  such that $(r_p)|_X \colon 
  X \to r_p(X)$ is a finite surjective morphism, and $(r_p)|_{\CZ} \colon \CZ \to r_p(\CZ)$ is an isomorphism over a Zariski open neighborhood of $r_p(z)$. When  $K$ is algebraically closed,  \cite[Lemma 43.23.1, Lemma 43.23.2]{stacks-project} show that there exists a Zariski open subset $W \subset \P^d_K$ such that any $p \in \P^d(K)$ satisfies the desired properties. When $K$ is not necessarily algebraically closed, the claim follows from the fact that $\P^d(K)$ is Zariski dense in $\P^d_{\overline{K}}$ if $K$ is infinite.

  Applying the above claim iteratively, we obtain a finite surjective morphism $r \colon X \to \P_K^n$ such that $r|_{\CZ} \colon \CZ \to r(\CZ)$ is an isomorphism over a Zariski open neighborhood $W$ of $r(z)$ in $\P_K^n$. In particular, this implies that $z \notin r(X \setminus Y)$. Since $r(X \setminus Y)$ is Zariski closed, we may take an affine open neighborhood $S$ of $r(z)$ in $W \setminus r(X \setminus Y)$. Then $U \coloneqq r^{-1}(W)$ is an affine open subvariety of $Y$. Set $\pi=r|_U \colon U \to S$. Thus, by construction, $\pi$ is finite surjective and $\pi|_{Z_{U}} $ is an isomorphism. 
\end{proof}

\label{proof:thm:NA_FN_theorem}
Now we are ready to prove Theorem \ref{thm:NA_FN_theorem}. Let $\dim Y=n$. Consider the filtration of $Y$ successive singular loci
\[
Y=Y_0 \supsetneq Y_1 \supsetneq \cdots \supsetneq Y_n \supsetneq Y_{n+1}=\emptyset, \qquad Y_{i+1}=Y_i^{\sing},
\]
and consider the following statement.

\medskip
\textbf{Statement} $\mathbf{\Sigma}_k$: For any closed scheme point $y \in |Y \setminus Y_k|_{\cl}$, there exists an affine open neighborhood  $U$ of $y$ in  $Y \setminus Y_k$ such that  $\{x \in U^{\an} \,|\, u(x)<0 \}$ is Runge in $U^{\an}$, and $u|_{U} \in \PSHlogreg(U^{\an})$. 
\medskip

Note that Theorem \ref{thm:NA_FN_theorem} follows from $\mathbf{\Sigma}_{k+1}$ and Corollary \ref{cor:pshlogreg_on_an_affine_open_cover_induces_pshlogreg}. Now we prove $\mathbf{\Sigma}_k$ inductively. When $k=1$, for any affine open subvariety $U \subset (Y \setminus Y_1)=Y^{\reg}$, Proposition \ref{prop:psh_equals_weakly_psh_on_smooth_varieties} implies that $u|_U \in \PSHlogreg(U^{\an})$, and Lemma \ref{lem:psh_to_Runge_domain_smooth_varieties} implies that $\{u<0\} \cap U^{\an}$ is Runge in $U^{\an}$. Now suppose that $k \ge 1$ and $\mathbf{\Sigma}_k$ has been proven, we would like to prove $\mathbf{\Sigma}_{k+1}$. Take $y \in |Y \setminus Y_{k+1}|_{\cl}$. If $y \in |Y \setminus Y_{k}|_{\cl}$, then such an affine open neighborhood of $y$ exists by the inductive assumption $\mathbf{\Sigma}_k$.  Thus, henceforth we assume  that $y \in |Y_{k} \setminus Y_{k+1}|_{\cl}$.

Take an affine open neighborhood $V$ of $y$ in $Y \setminus Y_{k+1}$. Then,  $Z_V \coloneqq \left( V \cap (Y_{k} \setminus Y_{k+1}) \right)$ is a nonsingular closed variety of $V$ of positive codimension. Also, note that $K$ is an infinite field since we assume that $K$ is of equicharacteristic $0$. Then, we may apply Lemma \ref{lem:generic_noether_normalization} to $y \in Z_V \subsetneq V$, and we obtain   an affine open neighborhood $U$ of $y$ in $V$, an affine open subvariety $S$  of $\BA^n_K$ and a finite surjective morphism $\pi \colon U \to S$ such that $\pi|_{Z_U} \colon Z_U \to \pi(Z_U)$ is an isomorphism, where $Z_U \coloneqq Z_V \cap U$.  

By the inductive assumption $\mathbf{\Sigma}_k$, for each  $\fm \in |U \setminus Z_U |_{\cl} = |U \setminus Y_k|_{\cl}$, there exists an affine open neighborhood $U_{\fm}$ of $\fm$ in $U$ such that $u|_{U_{\fm}} \in \PSHlogreg(U_{\fm}^{\an})$. Note that $\{U_{\fm}, \ \fm \in |U \setminus Z_U|_{\cl} \}$ is an affine open cover of $U \setminus Z_U$. Then, applying Proposition \ref{prop:weakly_psh_induces_Runge_under_injective_conditions} to the irreducible affine variety $U$, its closed subvariety $Z_U \subset U$ and $u|_U \in \WPSHlog(U^{\an})$, we conclude that $\{x \in U^{\an} \,|\, u(x)<0 \}$ is Runge in $U^{\an}$, and $u|_{U} \in \PSHlogreg(U^{\an})$. Therefore, $\mathbf{\Sigma}_{k+1}$ is proven, and the induction is completed.

\section{Envelope conjecture} \label{sec:envelope_conjecture}
Making use of the results in previous sections, we prove the envelope conjecture in the following setting.

\begin{thm}  \label{thm:envelope_conjecture_equichar_zero}
  Let $K$ be a discretely valued or trivially valued  non-archimedean field of equicharacteristic $0$ and $X$ be a geometrically unibranch projective variety over $K$, then $X$ satisfies the envelope property over $K$. 
\end{thm}

\begin{remark}
  Since any normal scheme is geometrically unibranch (\cite[Lemma 28.16.2]{stacks-project}), every projective normal variety satisfies the envelope property over $K$.
\end{remark}

\begin{proof}
  By \cite[Lemma 5.9]{BJ22}, it is enough to prove that $\GPSHo(X)$ satisfies the envelope property for any $\omega \in \nef(X)$ with $c_1(L(\omega)) \in \Amp(X)$.
  
  We first prove the theorem in the case when $K$ admits a countable dense subfield. Let $\{\varphi_i\}_{i \in I} \subset \GPSHo(X)$ be a class of functions.  Let $\pi \colon \widetilde{X} \to X$ be a resolution of singularities, and denote $\widetilde{\varphi}_i \coloneqq \varphi_i \circ \pi^{\an} \in \GPSH_{\pi^{*}\omega}(\widetilde{X})$. By Theorem \ref{thm:smooth_envelope_property}, $\widetilde{\varphi} \coloneqq (\sup_{i \in I} \widetilde{\varphi}_i )^{\star} \in \GPSH_{\pi^{*}\omega}(\widetilde{X})$.  Since $X$ is geometrically unibranch, every geometric fibre of $\pi$ is connected (\cite[Corollary 4.3.7]{EGAIII1}). Then, by Proposition \ref{prop:descent_psh_functions_resolution}, there exists a usc function $\varphi$ on $X^{\an}$ such that $\widetilde{\varphi}=\varphi \circ \pi^{\an}$ and $\varphi$ is the minimum usc extension of $\varphi|_{X^{\div}}$. On the other hand, for any $x \in X^{\div}$,  Theorem \ref{thm:global_divisorial_points_is_nonnegligible} implies that $\varphi(x)=\widetilde{\varphi}(x)=\sup_{i \in I}\widetilde{\varphi}_i(x)=\sup_{i \in I}\varphi_i(x)$. Therefore, $\varphi=(\sup_{i \in I} \varphi_i)^{\star}$. Now, by definition, $\varphi \in \WPSHo(X)$, and Theorem \ref{thm:global_NA_FN_theorem} implies that $\varphi \in \WPSHo(X)=\GPSHo(X)$. Therefore, the theorem is proved when $K$ admits a countable dense subfield. 
  
  Now we consider $K$ without the cardinality restriction.  By Lemma \ref{lem:equivalent_form_of_envelope_conjecture}, it is enough to show that $\envomega(f) \in C^0(X^{\an})$ for each $f \in \PL(X)$.  Since $Y$ is of finite type over $K$, there exists a countable subfield $F_1 \subset K$ such that $X$, $\omega$ and $f$ is defined over $F_1$. More precisely, let $K_1$ be the completion of $F_1$. Denote by $(X_{K_1})^{\an}$ the Berkovich analytification of $X$ over $K_1$, and denote by $r_{K/K_1} \colon X^{\an}=(X_K)^{\an} \to (X_{K_1})^{\an}$ the morphism induced by the base change $K/K_1$. Then, there exists $\omega_{K_1} \in \nef(X_{K_1})$ and  $f_{K_1} \in \PL((X_{K_1})^{\an})$ such that $\omega =\omega_{K_1} \circ r_{K/K_1}$ and $f=f_{K_1} \circ r_{K/K_1}$. 

  Consider the $\omega_{K_1}$-psh envelope $\phi_{K_1} \coloneqq \env_{\omega_{K_1}}(f_{K_1})$ on $(X_{K_1})^{\an}$. Since we have already shown that $\GPSH_{\omega_{K_1}}((X_{K_1})^{\an})$ satisfies the envelope property, Lemma \ref{lem:equivalent_form_of_envelope_conjecture} implies that $\phi_{K_1} \in C^0((X_{K_1})^{\an})$. Now set $\phi_K \coloneqq \env_{\omega_{K_1}}(f_{K_1}) \circ r_{K_1} \in C^{0}(X^{\an})$, and we claim that $\envomega(f)=\phi_K$. By \cite[Lemma 5.19]{BJ22}, it is enough to show that $\psi \le \phi$ for each $\psi \in \PL(X)$. There exists a countable field $F_2$ such that $F_1 \subset F_2 \subset K$ and $\psi$ is defined over $F_2$. Let $K_2$ be the completion of $F_2$, and $\psi_{K_2} \in \PSH_{\omega_{K_2}} \cap \PL((X_{K_2})^{\an})$ such that $\psi_{K_2}=\psi \circ r_{K_2}$. Here, $\omega_{K_2}, f_{K_2}, \phi_{K_2}$ denote the pullback of $\omega_{K_1}, f_{K_1}, \phi_{K_1}$ via   the  morphism $r_{K_2/ K_1} \colon (X_{K_2})^{\an} \to (X_{K_1})^{\an}$  induced by the base change $K_2/K_1$. The orthogonality property over $K_1$ implies that 
  \[
  \mathrm{Supp}(\MA_{\omega_{K_1}}(\phi_{K_1})) \subset \{f_{K_1}=\phi_{K_1}\}  \qquad \text{ on } (X_{K_1})^{\an}. 
  \]
  Note that 
  \[
  (r_{K_2, K_1})_{*}(\MA_{\omega_{K_2}}(\phi_{K_2})) = \MA_{\omega_{K_1}}(\phi_{K_1})  \qquad \text{ on } (X_{K_1})^{\an}.
  \]
  Therefore, on $\qquad \text{ on } (X_{K_2})^{\an}$, we have 
  \[
  \mathrm{Supp}(\MA_{\omega_{K_2}}(\phi_{K_2})) \subset r_{K_2/ K_1}^{-1}(\mathrm{Supp}(\MA_{\omega_{K_1}}(\phi_{K_1})))=r_{K_2/ K_1}^{-1}(\{f_{K_1}=\phi_{K_1}\})=\{f_{K_2}=\phi_{K_2}\}. 
  \]
  This implies that $\psi_{K_2} \le f_{K_2}= \phi_{K_2}$ on $\mathrm{Supp}(\MA_{\omega_{K_2}}(\phi_{K_2}))$.  Since $\GPSH_{\omega_{K_2}}((X_{K_2})^{\an})$ satisfies the envelope property, the domination principle (Theorem \ref{thm:Domination_Principle}  implies that $\psi_{K_2} \le \phi_{K_2}$ on $(X_{K_2})^{\an}$. This proves that $\psi \le \phi_K$ on $X^{\an}$, and $\envomega(f)=\phi_K \in C^0(X^{\an})$. Hence, the theorem is proved.
\end{proof}

\appendix

\section{PSH functions d'apr\`es Chambert-Loir--Ducros} \label{appendix:psh_functions_CLD}
In this appendix, we assume that $K$ is a discretely valued or trivially valued NA field. Let $Y$ be an affine variety over $K$, and let $U \subset Y^{\an}$ be an open subset. Following \cite{CLD}, we say that a function $u$ on $U$ is smooth if for any $x \in U$, there exists an open neighborhood $U_x$ of $x$ in $U$, a tropicalization map $\Trop \colon Y^{\an} \to \BT^k$, an open subset $W \subset \R^k$ and $f \in C^{\infty}(W)$ such that $\Trop(U_x) \subset W \subset \R^k$, and $u=f \circ \Trop$ on $U_x$. We say that $u$ is a smooth convex function if moreover, there exists a polyhedral structure $\Pi$ on the tropical variety $\Trop(Y^{\an})$ such that $f|_{\sigma \cap U_x}$ is convex for each polyhedron $\sigma \in \Pi$. We denote by $C^{\infty}(U)$ (\emph{resp.} $\LConv \cap C^{\infty}(U)$) the space of smooth (\emph{resp.} smooth convex) functions  on $U$.

Following \cite[Definition 8.1.1]{CLD}, we give the following definition.
\begin{definition}
  Let $u \colon U \to \R \cup \{-\infty\}$ be a function. We say that $u$ is a smooth approximable psh function if
  \begin{enumerate}
    \item $u$ is generically finite and usc on $U$. 
    \item For each $x \in U$, there exists an open neighborhood $U_x$ of $x$ in $U$, and a decreasing function sequence  $\{u_i\}_{i \ge 1} \subset \LConv \cap C^{\infty}(U_x)$ which converges pointwise to $u$ on $U_x$. 
  \end{enumerate}
  We use $\PSHsm(U)$ to denote the space of smooth approximable psh functions on $U$. 
\end{definition}

\begin{remark}
  Note that this definition is slightly different from \cite[Definition 8.1.1]{CLD}. 
  \begin{enumerate}
    \item Firstly, the original definition of smooth functions in \cite{CLD} allows the tropicalization map to be given by analytic functions on $U$. However, following similar arguments to the ones in Remark \ref{rmk:psh_is_analytic_invariant}, this doesn't affect the definition of $\PSHsm(U)$ when $U$ is an open subset of a quasi-projective variety.
    \item Secondly, similar to Remark \ref{rmk:decreasing_limit_of_psh_functions_is_not_known_to_be_psh_in_general}, we don't know \emph{a priori} whether the pointwise limit $u$ of a decreasing function sequence $\{u_i\}_{i \ge 1} \subset \PSHsm(U)$ is a smooth approximable psh function on $U$ in our definition, while  such a function is \emph{defined} to be psh in \cite[Definition 8.1.1]{CLD},. However, if $Y$ satisfies the envelope property over $K$ (for example when $Y$ is normal and $K$ is of equicharacteristic $0$ following Theorem \ref{thm:envelope_conjecture_equichar_zero}), Theorem \ref{thm:psh_equals_to_smooth_app_psh} and Corollary \ref{cor:properties_PSH_after_envelope_properties} (1) together imply that $u \in \PSHsm(U)$, and thereby our definition is equivalent to \cite[Definition 8.1.1]{CLD} in this case. \qedhere
  \end{enumerate}
\end{remark}

\begin{thm} \label{thm:psh_equals_to_smooth_app_psh}
  For any open subset $U \subset Y^{\an}$, we have $\PSH(U)=\PSHsm(U)$.
\end{thm}

The rest of this appendix is devoted to prove this theorem.

\smallskip

Let $u \in \TConv \cap \PL(U)$. Following the definition and Remark \ref{rmk:PL_functions_near_boundaries}, there exists an open neighborhood $U_x$ of $x$ in $U$, a tropicalization map $\Trop \colon Y^{\an} \to \BT^k$ and $f \in \Conv \cap \PL(\Trop(U_x))$ such that $\Trop(U_x) \subset \R^k$ and $u=f \circ \Trop$ on $U_x$. Note that any convex function on $\R^k$ can be approximated locally uniformly by a sequence of smooth convex functions $\{f_i\}_{i \ge 1}$. In particular, the restriction of $f_i$ on any polyhedron of $\Trop(U)$ is convex, and thereby $f_i \circ \Trop$ is a smooth convex function on their domain of definition.  This shows that any tropically convex PL function can be approximated locally uniformly by a sequence of smooth convex functions on $U$. Thus, we obtain $\PSH(U) \subset \PSHsm(U)$.

\smallskip

It remains to show that $\PSHsm(U) \subset \PSH(U)$. We first recall the following definition introduced in \cite{APW2}.  For a tropical variety $P \subset \BT^k$ and $f \in C^0(W \cap P)$ for an open subset $W \subset \BT^k$, we say that $f$ is a \emph{locally convex function} on $W \cap P$ if for any $\ft \in W \cap P$, there exists an open neighborhood $V_{\ft}$ of $\ft$ in $\BT^k$ and $l \in \Aff_{\R}(V_{\ft})$ such that $f(\ft)=l(\ft)$, and $f(\ft') \ge l(\ft')$ for any $\ft' \in V_{\ft} \cap P$. Denote by $\LConv(W \cap P)$ the space of locally convex functions on $W \cap P$. We remark that the restriction of any convex function on $W$ to $W \cap P$ is a locally convex function on $W \cap P$, but the converse does not hold. We refer the readers to \cite{APW2} for more properties about locally convex functions on tropical varieties.

Now fix $u \in C^{\infty}(U)$. We may assume that there exists a tropicalization map $\Trop \colon Y^{\an} \to \BT^k$, an open subset $W \subset \R^k$ and $f \in C^{\infty}(W)$ such that  $U=\Trop^{-1}(W)$, $u=f \circ \Trop$ and $f|_{W \cap \sigma}$ is convex for each polyhedron $\sigma \in \Pi$, with $\Pi$ being a polyhedral structure on  $\Trop(Y^{\an})$. Fix $x \in U$, and denote  $\ft \coloneqq \Trop(x) \in W \cap \Trop(Y^{\an})$. Set the affine function $l \coloneqq df_{\ft}+ f(\ft)$ on $\BT^k$, where $df_{\ft}$ is the differential of the smooth function $f$ at $\ft$. Let $\Pi^{(\ft)}$ be the union of the support of the polyhedra in $\Pi$ containing $x$. After subdividing $\Pi$, we may assume that $\Pi^{(\ft)} \subset W$, and we may take an open neighborhood $V_{\ft}$ of $\ft$ in $\BT^k$ satisfying $V_{\ft} \cap \Trop(Y^{\an}) \subset \Pi^{(\ft)}$. Note that for any $\sigma \in \Pi$ satisfying $x \in \sigma$, $f|_{\sigma}$ is a smooth convex function with differential $\left(df_{\ft} \right)|_{\sigma}$ at $\ft$. This implies that $f \ge l$ on $\sigma$, and thereby $f \ge l$ on $V_{\ft}$.  This proves that $f \in \LConv(W \cap \Trop(Y^{\an}))$. 

Following \cite[Proposition 5.22]{APW2}, for any $\ft \in W \cap \Trop(Y^{\an})$, there exists an open neighborhood $\CW_{\ft}$ of $\ft$ in $\BT^k$ such that  $f$ is the uniform limit of a function  sequence $\{f_i\}_{i \ge 1} \subset \LConv \cap \PL(\CW_{\ft})$. On the other hand, we may assume that the tropicalization map $\Trop \colon Y^{\an} \to \BT^k$ is associated to a closed embedding $Y \hookrightarrow \BA^k_{K}$. Let $X$ be the Zariski closure of $Y$ in $\P^k_{K}$, and $\omega \in \nef(X)$ be the restriction of the closed semipositive $(1,1)$-form on $\P^k$ corresponding to the trivial model of $(\P^k, \CO(1))$. Then, there exists a potential of $\omega$ on $Y^{\an}$ factoring through $\Trop$. 

Combining \cite[Proposition 5.20]{APW2} with \cite[Theorem 6.13]{APW1}, there exists an open neighborhood $\CV_{\ft}$ of $\ft$ with $\CV_{\ft} \Subset \CW_{\ft}$, a constant $C>0$, $g \in \PL(\CW_{\ft})$ and $\varphi_i \in \GPSHo \cap \PL(X)$ for each $i \ge 1$ such that $g \circ \Trop$ is a potential of $\omega$ on $\Trop^{-1}(\CW_{\ft})$, and $\varphi_i+g=C u_i$ on $\Trop^{-1}(\CV_{\ft})$ for each $i \ge 1$, where  $u_i \coloneqq f_i \circ \Trop \in \PSH \cap \PL(\Trop^{-1}(\CV_{\ft}))$.  Then, by Proposition \ref{prop:regularization_omega_psh_by_omega_FS_functions} (1) and Dini's theorem, for each $i \ge 1$, there exists a compact open neighborhood $U_x \Subset \Trop^{-1}(\CV_{\ft})$ of $x$ and a function sequence $\{v_{i,j}\}_{j \ge 1} \in \TConv \cap \PL(U_x)$ converging uniformly to $u_i$ on $U_x$. As $\{u_i\}_{i \ge 1}$ converges uniformly to $u$ on $U_x$, this proves that $u$ can be approximated locally uniformly by a sequence of tropically convex PL functions. Thus, we obtain $\PSHsm(U) \subset \PSH(U)$, and the theorem is proved.

\bibliographystyle{alpha}
\bibliography{bibliography}

\end{document}